\documentclass{article}
\usepackage{arxiv}
\usepackage{geometry}
\usepackage[document]{ragged2e}
\usepackage{graphicx,subfigure}
\usepackage{amssymb}
\usepackage{amsmath}
\usepackage{amsthm}
\usepackage{esvect}
\usepackage{mathtools}
\usepackage{subfigure}
\usepackage{bm, bbm}
\usepackage{xcolor}
\usepackage{appendix}
\usepackage{parskip}
\usepackage{mathrsfs}
\usepackage{amsfonts}
\usepackage{fancyhdr}
\usepackage{empheq}
\usepackage{multirow}
\usepackage{afterpage}
\usepackage{hyperref}
\usepackage{diagbox}
\usepackage{stackengine}
\usepackage{tikz-cd}
\usepackage{capt-of}
\usepackage{rotating,tabularx}
\usepackage{pdflscape}
\usepackage{float}
\usepackage{nccmath}
\usepackage{enumitem}
\usepackage[aboveskip=5pt,labelfont=bf,labelsep=period,justification=raggedright,singlelinecheck=off, font={it}]{caption}
\usepackage{booktabs}
\usepackage{siunitx}
\usepackage{calc}
\usepackage{layouts}
\usepackage{etoolbox}
\usepackage{xcolor}
\usepackage{algpseudocode}
\usepackage{algorithm}
\usepackage[breakable]{tcolorbox}
\usepackage{xr}
\usepackage[numbers]{natbib}
\usepackage{doi}

\afterpage{\cfoot{\thepage}}

\ifpdf
\DeclareGraphicsExtensions{.eps,.pdf,.png,.jpg}
\else
\DeclareGraphicsExtensions{.eps}
\fi

\setlist{itemsep=-2.5pt, leftmargin=\labelwidth+\labelsep}
\usetikzlibrary{shapes.geometric, arrows}
\tikzstyle{startstop} = [rectangle, rounded corners, minimum width=2cm, minimum height=1cm,text centered, draw=black, fill=white!30, text width=2cm]
\tikzstyle{startstopnew} = [rectangle, rounded corners, minimum width=2cm, minimum height=1cm,text centered, draw=black, fill=white!30, text width=5cm]
\tikzstyle{process} = [rectangle, minimum height=1cm, text centered, draw=black, fill=white!30, text width=4cm]
\tikzstyle{process_new} = [rectangle, minimum height=1cm, text centered, draw=black, fill=white!30, text width=2.1cm]
\tikzstyle{decision} = [diamond, minimum width=2cm, minimum height=0.5cm, text centered, draw=black, fill=white!30]
\tikzstyle{arrow} = [thick,->,>=stealth]

\newtheorem{lmm}{Lemma}

\newtheorem{theorem}{Theorem}
\newtheorem{corollary}{Corollary}
\newtheorem*{remark}{Remark}

\tcbuselibrary{theorems}
\tcbuselibrary{skins}
\tcbset{
	commonstyle/.style={
		theorem style=plain,
		enhanced jigsaw,
		fonttitle=\bfseries,
		fontupper=\itshape,
		halign=justify,
		separator sign=:,
		description delimiters none,
		description font=\bfseries, 
		terminator sign={.\hspace{0.25em}},
		arc=0mm,outer arc=0mm,
		boxrule=0pt,toprule=0pt,bottomrule=0pt,leftrule=0pt,rightrule=0pt,
		titlerule=0pt,toptitle=0pt,bottomtitle=0pt,top=0pt,
		colback=white,coltitle=black,
		boxsep=0pt, bottom=0pt, left=0pt, %
	}
}
\newtcbtheorem[]{myproblem}{Problem}%
{center, commonstyle, fonttitle=\bfseries}{pb}
\newtcbtheorem[]{mydefinition}{Definition}%
{center, commonstyle, fonttitle=\bfseries}{pb}

\renewcommand{\inf}{\mathop{\mathrm{inf}\vphantom{\mathrm{sup}}}}
\newcommand\mytilde[1]{\stackrel{\sim}{\smash{#1}\rule{0pt}{1.05ex}}}

\renewcommand{\vec}[1]{\underline{#1}}

\robustify\bfseries
\newcolumntype{R}[1]{>{\RaggedRight\arraybackslash}m{#1}}
\newcolumntype{P}{>{\RaggedRight\arraybackslash}p{(\linewidth-2.2cm-14\tabcolsep)/6}}

\title{Space--time reduced basis methods for parametrized unsteady Stokes equations}

\author{ \href{https://orcid.org/0000-0002-2784-9261}{\includegraphics[scale=0.06]{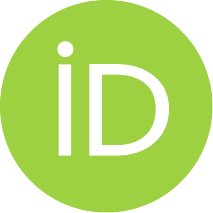} \hspace{1mm} Riccardo Tenderini}\\
	Institute of Mathematics\\
	Ècole Polytechnique Fédérale de Lausanne (EPFL)\\
	CH--1015 Lausanne, Switzerland \\
	\And
	\href{https://orcid.org/0000-0002-7798-0858}{\includegraphics[scale=0.06]{img/orcid.pdf} \hspace{1mm} Nicholas Mueller} \\
	Department of Mathematics\\
	Monash University\\
	VIC--3800, Clayton, Victoria, Australia \\
	\And
	\href{https://orcid.org/0000-0002-2832-6630}{\includegraphics[scale=0.06]{img/orcid.pdf} \hspace{1mm} Simone Deparis} \\
	Institute of Mathematics\\
	Ècole Polytechnique Fédérale de Lausanne (EPFL)\\
	CH--1015 Lausanne, Switzerland \\
}

\date{}

\renewcommand{\headeright}{ST--RB methods for parametrized unsteady Stokes equations}
\renewcommand{\undertitle}{}

\afterpage{\cfoot{\thepage}}

\numberwithin{equation}{section}

\begin{document}

\maketitle

\begin{abstract}
	\justifying
	In this work, we analyse space--time reduced basis methods for the efficient numerical simulation of h{\ae}modynamics in arteries. The classical formulation of the reduced basis (RB) method features dimensionality reduction in space, while finite differences schemes are employed for the time integration of the resulting ordinary differential equation (ODE). Space--time reduced basis (ST--RB) methods extend the dimensionality reduction paradigm to the temporal dimension, projecting the full--order problem onto a low--dimensional spatio--temporal subspace. Our goal is to investigate the application of ST--RB methods to the unsteady incompressible Stokes equations, with a particular focus on stability. High--fidelity simulations are performed using the Finite Element (FE) method and BDF2 as time marching scheme. We consider two different ST--RB methods. In the first one --- called ST--GRB --- space--time model order reduction is achieved by means of a Galerkin projection; a spatio--temporal velocity basis enrichment procedure is introduced to guarantee stability. The second method --- called ST--PGRB --- is characterized by a Petrov--Galerkin projection, stemming from a suitable minimization of the FOM residual, that allows to automatically attain stability. The classical RB method --- denoted as SRB--TFO ---  serves as a baseline for the theoretical development. Numerical tests have been conducted on an idealized symmetric bifurcation geometry and on the patient--specific one of a femoropopliteal bypass. The results show that both ST--RB methods provide accurate approximations of the high--fidelity solutions, while considerably reducing the computational cost. In particular, the ST--PGRB method exhibits the best performance, as it features a better computational efficiency while retaining accuracies in accordance with theoretical expectations.
\end{abstract}  


\justifying

\section{Introduction}
\label{sec:introduction}
Patient--specific high--fidelity numerical simulations of h{\ae}modynamics are traditionally performed solving parametrized unsteady incompressible Navier--Stokes equations. The Finite Elements (FE) or the Finite Volumes (FV) methods are typically used for spatial discretization, while implicit linear multistep (such as Backward Diffentiation Formulas -- BDF) or multistage (such as Runge--Kutta -- RK) methods are employed for the time integration of the resulting ordinary differential equation (ODE). Overall, this defines the so--called full--order model (FOM). Depending on the features of the problem at hand --- like the shape and dimension of the physical domain, the complexity of the physical processes of interest or the length of the simulation interval --- solving the FOM problem can be extremely expensive from a computational standpoint, even exploiting parallel computations \cite{lassila2013model}. In case of parametric dependence, Reduced Order Models (ROMs) are widely employed to lighten the computational burden of the simulations, yet retaining a good degree of accuracy. Projection--based ROMs (PROMs) --- such as the traditional Reduced Basis (RB) method --- reduce the spatial dimensionality of the dynamical system by means of a projection process, leading to a low--dimensional ODE. However, the resolution of the latter involves time integration, which is typically performed adopting the same scheme and with the same timestep size used in the FOM. As a result, traditional space--reduced ROMs feature a much smaller dimensionality in space than the FOM, but the same one in time. Therefore, in problems where either the time simulation interval is very large (as for instance in many applications in computational fluid dynamics), or the timestep size should be very small in order to capture some relevant behaviours happening at small time scales (as for instance in molecular dynamics), significant computational gains are difficult to realize.

The issue linked to temporal complexity has already been addressed in various ways. A comprehensive literature review on the topic can be found in \cite{choi2019space}, where the authors investigate the pros and cons of several approaches. Among those, we may cite time--parallel methods (such as parareal \cite{maday2002parareal}, PITA \cite{falgout2014parallel} and MGRIT \cite{farhat2003time}) and ``forecasting'' approaches \cite{carlberg2015decreasing, carlberg2017galerkin}, even though they only allow to reduce the wall--time of simulations and not the temporal complexity of the problem. A reduction of both the spatial and the temporal dimension of the FOM characterizes the space--time RB methods presented in \cite{urban2012new, urban2014improved, yano2014space, yano2014space2}, that also feature error bounds that grow linearly, rather than exponentially, with respect to the number of timesteps. However, such methods exhibit some relevant drawbacks, the major ones being the need for a (uncommon) FE discretization of the time domain and the absence of hyper--reduction techniques to efficiently handle non--linearities.

In \cite{choi2019space} the authors propose a novel approach --- called Space--Time Least--Squares Petrov--Galerkin (ST--LSPG) method --- to tackle parametrized non--linear dynamical systems. The idea is to minimize the FOM residual, computed from the FOM reconstruction of the space--time reduced solution, in a weighted spatio--temporal $\ell^2$--norm. Different choices for the reduced basis construction and for the weighting matrix assembling are proposed and analysed. \emph{A priori} error bounds, bounding the solution error by the best space--time approximation error, are also retrieved and the stability constant features a subquadratic growth with respect to the total number of time instances. Even though the performances of the ST--LSPG method are found to be good on 1D non--linear dynamical systems, a deterioration is expected when dealing with 2D or 3D geometries. This claim is particularly strong if the space--time collocation approach is employed for hyper--reduction, since sampling techniques notoriously suffer the curse of dimensionality. In \cite{shimizu2021windowed}, the drawbacks of the ST--LSPG method are addressed adopting a time--windowed strategy (Windowed ST--LSPG -- WST--LSPG). The simulation interval is partitioned into windows; within each window, a low--dimensional spatio--temporal subspace is defined and the residual is minimized in a weighted $\ell^2$--norm. Numerical experiments, carried out also considering 2D compressible Navier--Stokes equations for the flow around an airfoil, demonstrate that the WST--LSPG method is better than the ST--LSPG one both in terms of accuracy and of efficiency. However, the coupling between the different time windows is not taken into account and this could easily deteriorate the performances if their number is large. Finally, in \cite{choi2021space, kim2021efficient}, the time--complexity bottleneck of the RB method is addressed by performing, respectively, a Galerkin projection and a least--squares Petrov--Galerkin projection of the FOM onto a low--dimensional spatio--temporal subspace. The resulting methods exhibit good performances on 2D linear dynamical systems and \emph{a priori} error bounds, featuring a subquadratic dependency on the total number of time instants, are derived. However, only linear problems have been considered.  

The goal of this work is to investigate the application of Space--Time Reduced Basis (ST--RB) methods to the unsteady parametrized incompressible Stokes equations, a well--known linearization of the Navier--Stokes equations that models Newtonian flow at small Reynolds' numbers. The problem parametrization affects the inflow/outflow rates and a linear reaction term added to the momentum conservation equation in order to model the presence of blood clots; the geometry is assumed fixed. Additionally, non--homogeneous Dirichlet boundary conditions (BCs) are weakly imposed by means of Lagrange multipliers \cite{pegolotti2021model}. The application of two different ST--RB methods --- which rely on a Galerkin and on a Petrov--Galerkin projection respectively --- to the problem at hand is detailed and the well--posedness of the resulting problem is investigated. This last step is particularly challenging since, dealing with a twofold saddle point problem \cite{gatica2008characterizing, howell2011inf}, it resorts to \emph{inf--sup} stability analysis.

The manuscript is structured as follows. In Section \ref{sec:section2}, we introduce the unsteady parametrized incompressible Stokes equations and we discuss their full--order discretization. In Section \ref{sec:section3}, we investigate the application of the aforementioned ST--RB approaches to the problem at hand. In particular, we describe the construction of the space--time reduced bases and, for both methods, we detail the assembling of the linear systems to be solved. Additionally, we discuss the stability of the two approaches. Section \ref{sec:section4} presents the numerical results, obtained on two different test cases, characterized by different geometries and parametrizations of the boundary data. Finally, Section \ref{sec:section5} provides a summary, lists some limitations of the work and proposes possible future developments.

\section{Unsteady Parametrized Incompressible Stokes Equations}
\label{sec:section2}
Even if h{\ae}morheology indicates that blood is a non--Newtonian fluid, the latter can be approximated as a Newtonian one if the vessels where it flows are sufficiently large. Under this assumption, the blood flow is governed by unsteady incompressible Navier--Stokes equations. In this work, we make an additional simplification: we neglect the non--linear convective term characterizing the Navier--Stokes equations, hence modelling blood flow by means of the unsteady incompressible Stokes equations. This system of linear PDEs well describes incompressible Newtonian flow at small Reynolds' numbers, i.e. in a regime where inertial forces are negligible with respect to viscous ones (see e.g. \cite{manzoni2012reduced, quarteroni2000computational}). In this regard, we remark that blood flow typically features Reynolds' numbers of $\approx 10^2-10^3$; the Stokes assumption is thus not valid in practice and an extension of the present work to deal with Navier--Stokes equations is being planned. Moreover, we consider the presence of blood clots of variable density, which are modelled by suitable reaction terms in the momentum conservation equation.

\subsection{Strong and weak formulation}\label{subs:strong_weak_stokes}
We consider an open, simply connected and bounded domain $\Omega \subset \mathbb{R}^d$ and we denote its boundary by $\partial\Omega$. We define the parameter space $\mathcal{D} \subset \mathbb{R}^p$ and we generically denote by $\bm{\mu}$ one of its elements. The unsteady parametrized incompressible Stokes equations in $\Omega$ read as:
\begin{equation} 
\label{eq: strong_form_stokes}
\begin{cases}
\rho\vec{u}_t^{\bm{\mu}} - \nabla\cdot (2\mu\nabla^s\vec{u}^{\bm{\mu}}) + \rho_c^{\bm{\mu}}(\vec{x}) \vec{u}^{\bm{\mu}} + \nabla p^{\bm{\mu}} = \vec{f}^{\bm{\mu}}(\vec{x})  & \text{in} \ \Omega \times [0,T] \\
\nabla \cdot \vec{u}^{\bm{\mu}} = 0  & \text{in} \ \Omega \times [0,T] \\
\vec{u}^{\bm{\mu}} = \vec{g}^{\bm{\mu}}(\vec{x})  & \text{on} \ \Gamma_\text{D} \times [0,T] \\
\sigma(\vec{u}^{\bm{\mu}}, p^{\bm{\mu}}) \vec{n} = \vec{h}^{\bm{\mu}}(\vec{x}) & \text{on} \ \Gamma_\text{N} \times [0,T] \\
\vec{u}^{\bm{\mu}} = \vec{u}_0^{\bm{\mu}} & \text{in} \ \Omega \times \{0\}
\end{cases} 
\end{equation}
Here $\vec{u}^{\bm{\mu}}: \Omega \times [0,T] \to \mathbb{R}^d$ and $p^{\bm{\mu}}: \Omega \times [0,T] \to \mathbb{R}$ are the velocity and the pressure of the fluid ($\vec{u}^{\bm{\mu}}_t$ denotes the partial derivative of the velocity in time); $\rho$ and $\mu$ are the fluid's density and viscosity respectively; $\nabla^s\vec{u} = (\nabla\vec{u}+\nabla^T\vec{u})/2$ is the strain rate tensor; $\sigma(\vec{u}, p) = 2\mu\nabla^s\vec{u} - pI$ is the Cauchy stress tensor;  
$\rho_c^{\bm{\mu}}: \Omega \to \mathbb{R}$ is the ``global'' blood clot density, defined as $\rho_c^{\bm{\mu}}(x) := \sum_{q=1}^{N_c} \rho_c^q(\bm{\mu}) \mathbbm{1}_{\Omega_c^q}(x)$ with $\rho_c^q : \mathcal{D} \to \mathbb{R}^+$, where $\mathbbm{1}_A: \mathbb{R}^d \to \mathbb{R}$ denotes the indicator function of a set $A \subset \mathbb{R}^d$, while $\Omega_c^q \subset \Omega$, $\rho_c^q$ are the location and the density of the $q$--th clot, respectively; $\vec{f}^{\bm{\mu}}: \Omega \times [0,T] \to \mathbb{R}^d$ is a forcing term; $\vec{g}^{\bm{\mu}}: \Gamma_D \times [0,T] \to \mathbb{R}^d$ and $\vec{h}^{\bm{\mu}}: \Gamma_N \times [0,T] \to \mathbb{R}^d$ are the Dirichlet and Neumann boundary data, respectively; $\vec{u}_0^{\bm{\mu}}: \Omega \to \mathbb{R}^d$ is the initial condition; $\vec{n}$ is the outward unit normal vector to $\partial\Omega$. $\{\Gamma_D, \Gamma_N\}$ is a partition of $\partial\Omega$ which defines the Dirichlet and the Neumann boundaries, respectively. Since we deal with cardiovascular simulations, it is useful to define the inlet boundary $\Gamma_{\text{IN}}$, the outlet boundary $\Gamma_{\text{OUT}}$ and the vessel wall boundary $\Gamma_{W}$. In this work, we always impose homogeneous Dirichlet BCs on $\Gamma_W$ (i.e. no--slip BCs, so that the artery is approximated as a rigid object) and non--homogeneous Dirichlet BCs on $\Gamma_{IN}$. The nature of the BCs imposed on $\Gamma_{OUT}$ depends instead on the test case at hand. We also define the non--homogeneous Dirichlet boundary $\mytilde{\Gamma}_D := \bigcup_{k=1}^{N_D} \mytilde{\Gamma}_D^k$, where $\mytilde{\Gamma}_D^k$ denotes the $k$--th inlet/outlet boundary where non--homogeneous Dirichlet BCs are imposed; $N_D$ is the number of such boundaries. It is worth highlighting that here we consider the parametric dependency to exclusively characterize the Dirichlet datum $\vec{g}^{\bm{\mu}}$ and the blood clot density $\rho_c^{\bm{\mu}}$. This restriction allows for an efficient offline/online splitting, hence making the problem more amenable for (space--time) model order reduction; other choices are of course possible. For the sake of conciseness, from now on we drop all the $(\cdot)^{\bm{\mu}}$ superscripts, except when referring to the parameter--dependent data $\vec{g}^{\bm{\mu}}$ and $\rho_c^{\bm{\mu}}$.

Let us introduce the following spaces:
\begin{equation}
\label{eq: hilbert spaces}	
\mathcal{V}^g :=  (H^1|_{\Gamma_D}(\Omega))^d = \left\{ \vec{v} \in (H^1(\Omega))^d \ \ \text{s.t.} \   \vec{v}=\vec{g}^{\bm{\mu}} \  \text{on} \ \Gamma_\text{D}\right\};
\qquad
\mathcal{Q} := L^2(\Omega);
\end{equation}
equipped with the usual inner products $(\cdot, \cdot)_{\mathcal{V}^g} = (\cdot, \cdot)_{(H^1)^d}$ and $(\cdot, \cdot)_{\mathcal{Q}} = (\cdot, \cdot)_{L^2}$. Let us also define $\mathcal{V}_{0} := (H_0^1(\Omega))^d$. The weak formulation of Eq.\eqref{eq: strong_form_stokes} has the structure of a non--symmetric and non--coercive saddle point problem. Instead of relying on the definition of a lifting function, we choose to impose non--homogeneous Dirichlet BCs weakly, using Lagrange multipliers. Such a choice is driven by the possibility of using a similar formulation to couple several domains  \cite{pegolotti2021model}. This approach translates into the following weak formulation: 
\begin{myproblem}{}{weak_form_stokes_2}
	Given $\vec{f}$, $\vec{g}^{\bm{\mu}}$, $\vec{h}$ regular enough, find $\left(\vec{u}, p, \vec{\lambda} \right) \in \mathcal{V}^g \times \mathcal{Q} \times \mathcal{L}$, such that $\forall t \in [0,T]$:
	\begin{equation} 
	\label{eq: weak_form_stokes_2}
	\begin{cases} 
	\rho \int_{\Omega}\vec{u}_t\cdot\vec{v}+\mu\int_{\Omega}\nabla^s\vec{u}:\nabla^s\vec{v} + \int_{\Omega} \rho_c^{\bm{\mu}} \vec{u} \cdot \vec{v} -
	\int_{\Omega}p\nabla \cdot \vec{v}+\int_{\mytilde{\Gamma}_D}\vec{\lambda}\cdot\vec{v} = \int_{\Omega}\vec{f}\cdot\vec{v} + \int_{\Gamma_N}\vec{h}\cdot\vec{v} & \ \\
	\int_{\Omega}q\nabla \cdot \vec{u} = 0 & \ \\
	\int_{\mytilde{\Gamma}_D}\vec{u}\cdot \vec{\xi} = \int_{\mytilde{\Gamma}_D} \vec{g}^{\bm{\mu}} \cdot \vec{\xi} & \
	\end{cases} 
	\end{equation}
	for all $\left(\vec{v}, \vec{q}, \vec{\xi}\right) \in \mathcal{V}^0 \times \mathcal{Q} \times \mathcal{L}$ and $\vec{u}=\vec{u}_0$ for $t=0$.
\end{myproblem}
\noindent
Based on the theory on primal hybrid methods, a natural choice for the space of Lagrange multipliers is $\mathcal{L} := \prod_{k=1}^{N_D} \mathcal{L}^k$, with $\mathcal{L}^k = \left(H^{-1/2}_{00}(\mytilde{\Gamma}_D^k)\right)^d$. We refer the reader to \cite{pegolotti2021model} for details. We remark that Problem \ref{pb:weak_form_stokes_2} features a twofold saddle point structure, as the dual space $\mathcal{Q} \times \mathcal{L}$ is a product space \cite{gatica2008characterizing,howell2011inf}.
\subsection{High--fidelity Numerical Discretization} \label{subs: discretization_stokes}
Since Problem \ref{pb:weak_form_stokes_2} is time--dependent, its discretization involves both the spatial and the temporal dimension. Concerning spatial discretization, we rely on the FE method. Therefore, we introduce the following finite dimensional  approximations of velocity, pressure and Lagrange multipliers (for $k \in \{1, \dots, N_D\}$): 
\begin{equation*}
\label{eq: finite_element_approximation}
\vec{u}_h(\bm{x},t) = \sum\limits_{i=1}^{N_u^s} u_i(t) \vec{\varphi}_i^u(\bm{x}) \in \mathcal{V}^g_h;
\ \ 
p_h(\bm{x},t) = \sum\limits_{i=1}^{N_p^s}p_i(t) \varphi_i^p(\bm{x}) \in \mathcal{Q}_h;
\ \ 
\vec{\lambda}_h^k(\bm{x},t) = \sum\limits_{i=1}^{N_\lambda^k} \lambda^k_i(t) \vec{\eta}^k_i(\bm{x}) \in \mathcal{L}_h^k.
\end{equation*}
The definitions of $\mathcal{V}^g_h$ and $\mathcal{Q}_h$ are of key importance for the accuracy of the approximation. Moreover, in the case of saddle point problems, the quality of the discretization is critical to ensure well--posedness; we refer to Section \ref{subs: well posedness of the FOM} for details. The discretization of the spaces of Lagrange multipliers is based on the definition of orthonormal basis functions on the unit disk, which are constructed from Chebyshev polynomials of the second kind. Since $\mathcal{L} := \prod_{k=1}^{N_D} \mathcal{L}^k$, we have that $\mathcal{L}_h = \prod_{k=1}^{N_D} \mathcal{L}^k_h$, where $\mathcal{L}^k_h$ is a finite--dimensional approximation of $\mathcal{L}^k$. This implies that the dimensionality of $\mathcal{L}_h$ is $N_\lambda = \sum_{k=1}^{N_D} N_\lambda^k$, being $N_\lambda^k :=  dim(\mathcal{L}_h^k)$. We denote the basis functions of $\mathcal{L}_h^k$ as $\{\vec{\eta}_i^k\}_{i=1}^{N_\lambda^k}$; they are orthonormal in $L^2(\mytilde{\Gamma}^k_D)$--norm. We refer the reader to \cite{pegolotti2021model} for additional details. We can then introduce the following matrices and vectors
\begin{equation}
\label{eq: FOM matrices}
\begin{alignedat}{2}
\bm{A} \in \mathbb{R}^{N_u^s \times N_u^s} :\quad \bm{A}_{ij} &= 2\mu\int_{\Omega}\nabla^s(\vec{\varphi}_j^u):\nabla^s(\vec{\varphi}_i^u) 
\qquad &
\bm{M} \in \mathbb{R}^{N_u^s \times N_u^s} :\quad  \bm{M}_{ij} &= \rho\int_{\Omega}\vec{\varphi}_j^u \cdot \vec{\varphi}_i^u \\
\bm{B} \in \mathbb{R}^{N_p^s \times N_u^s} :\quad \bm{B}_{ij} &= -\int_{\Omega}\varphi_i^p \ \nabla\cdot\vec{\varphi}_j^u
\qquad &
\bm{C}^k \in \mathbb{R}^{N_\lambda^k \times N_u^s} :\quad \bm{C}^k_{ij} &= \int_{\mytilde{\Gamma}_D^k}\vec{\varphi}_j^u \cdot \vec{\eta}_i^k \\
\bm{f} \in \mathbb{R}^{N_u^s} :\quad \bm{f}_i &= \int_{\Omega}\vec{f}\cdot \vec{\varphi}_i^u + \int_{\Gamma_N}\vec{h}\cdot\vec{\varphi}_i^u 
\qquad &
\tilde{\bm{g}}^{k,\bm{\mu}} \in \mathbb{R}^{N^k_\lambda} :\quad \tilde{\bm{g}}_i^{k,\bm{\mu}} &= \int_{\mytilde{\Gamma}_D^k}\vec{g}^{k,\bm{\mu}} \cdot \vec{\eta}_i^k
\end{alignedat}
\end{equation}
where $\vec{g}^{k,\bm{\mu}}$ represents the Dirichlet datum on $\mytilde{\Gamma}_D^k$. To ease the notation, we also define
\begin{equation}
\label{eq: FOM matrices 2}
\bm{C} = \left[ \left(\bm{C}^1\right)^T \vert \cdots \vert \left(\bm{C}^{N_D}\right)^T \right]^T \in \mathbb{R}^{N_\lambda \times N_u^s}, \qquad
\tilde{\bm{g}}^{\bm{\mu}} = \left[ \left(\tilde{\bm{g}}^{1, \bm{\mu}}\right)^T \vert \cdots \vert \left(\tilde{\bm{g}}^{N_D, \bm{\mu}}\right)^T \right]^T \in \mathbb{R}^{N_\lambda}.
\end{equation} 
Moreover, we define the matrices
\begin{equation}
\bm{R}^q \in \mathbb{R}^{N_u^s \times N_u^s} : \quad \bm{R}^q_{ij} = \int_{\Omega_c^q} \vec{\varphi}_j^u \cdot \vec{\varphi}_i^u \qquad \text{with} \quad q \in \left\{1, \dots, N_c\right\}~,
\end{equation}
that model the presence of $N_c$ blood clots of unitary density in the subdomains $\left\{\Omega_c^q\right\}_{q=1}^{N_c}$. We also define the ``global'' reaction matrix $\bm{R}(\bm{\mu}) := \sum_{q=1}^{N_c} \rho_c^q(\bm{\mu}) \bm{R}^q$; this expression shows the affine parametric dependency of the left-hand side term of the full--order problem.

\begin{remark}
Notice that, in order to strongly enforce homogeneous Dirichlet BCs on the vessel wall $\Gamma_W$, we need to suitably modify the rows of the matrices $\bm{M}$, $\bm{A}$, $\bm{B}^T$, $\bm{C}^T$, $\bm{R}(\bm{\mu})$ and $\bm{f}$. As a result, $\bm{B}^T$ and $\bm{C}^T$ do not exactly correspond to the transposed of $\bm{B}$ and $\bm{C}$, respectively.
\end{remark}

Regarding temporal discretization, let us introduce a sequence of timesteps $\{t_n\}_{n=0}^{N^t}$ such that $t_0 =0$, $t_{N^t}=T$ and $t_{n+1} = t_n + \delta$; $\delta$ is called timestep size. Let $\bm{w}_n := \left[\bm{u}_n, \bm{p}_n, \bm{\lambda}_n\right]^T \in \mathbb{R}^{N^s}$ be the solution at time $t_n$, with $N^s := N_u^s + N_p^s + N_\lambda$ and for $n \in \{0,\cdots,N^t\}$. We apply the second--order implicit BDF2 scheme. So, given $\bm{w}_{n-j+1}$ for $j=1,2$, $\bm{w}_{n+1}$ is such that:
\begin{equation}
\label{eq: residual BDF2}
\bm{r}(\bm{w}_{n+1}) := \bm{H} \bm{w}_{n+1} - \sum\limits_{j=1}^2 \alpha_j \bm{H} \bm{w}_{n-j+1} - \delta \beta \bm{\mathring{f}}(t_{n+1}, \bm{w}_{k+1}) = \bm{0}~,
\end{equation}
where
\begin{equation}
\label{eq: BDF2 matrices}
\bm{H} = \begin{bmatrix}
\bm{M} & & \\
& & & \\ 
& & &
\end{bmatrix} ;
\qquad 
\bm{\mathring{f}}(t_n, \bm{w}_n) = \begin{bmatrix}
\bm{f}(t_n) \\ \\ \tilde{\bm{g}}^{\bm{\mu}}(t_n)
\end{bmatrix}
- 
\begin{bmatrix}
\bm{A} + \bm{R}(\bm{\mu}) & \bm{B}^T & \bm{C}^T \\
\bm{B} & & \\
\bm{C} & &
\end{bmatrix}
\begin{bmatrix}
\bm{u}_n \\ \bm{p}_n \\ \bm{\lambda}_n
\end{bmatrix}.
\end{equation}
In particular, for the BDF2 method, we have $\alpha_1 = 4/3$, $\alpha_2 = -1/3$ and $\beta = 2/3$.

\newpage

Eq.\eqref{eq: residual BDF2} can be rewritten in the form of a single linear system of dimension $N^{st} \times N^{st}$ --- being $N^{st} := N^s N^t$ the number of DOFs of the space--time FOM problem --- as
\begin{equation}
\label{eq: monolitic_FOM_system}
\begin{bmatrix} 
\bm{A}^{st}_1 + \bm{R}^{st}(\bm{\mu}) & \bm{A}^{st}_2 & \bm{A}^{st}_3 \\ 
\bm{A}^{st}_4 &  &  \\ 
\bm{A}^{st}_7 &  &  
\end{bmatrix}
\begin{bmatrix}
\bm{u}^{st} \\ \bm{p}^{st} \\ \bm{\lambda}^{st}
\end{bmatrix}
= 
\begin{bmatrix} 
\bm{F}^{st}_1 \\ \ \\ \bm{F}^{st}_3(\bm{\mu}) \end{bmatrix}
\end{equation}
or, even more compactly, as $\bm{A}^{st} \bm{w}^{st} = \bm{F}^{st}$. The blocks of matrix $\bm{A}^{st}$, that features a twofold saddle point structure of type 1 \cite{howell2011inf}, can be written as follows:
\begin{align}
\label{eq: monolitic_FOM_blocks_matrix}
\bm{A}_1^{st} &= diag\biggl(\underbrace{\bm{M}, \cdots, \bm{M}}_{N^t}\biggr) + 
\dfrac23 \delta \ diag\biggl(\underbrace{\bm{A}, \cdots, \bm{A}}_{N^t}\biggr) \notag \\[-12pt]
& \quad \ - \dfrac43 \ subdiag^{(1)}\biggl(\underbrace{\bm{M}, \cdots, \bm{M}}_{N^t-1}\biggr) +
\dfrac13 \ subdiag^{(2)}\biggl(\underbrace{\bm{M}, \cdots, \bm{M}}_{N^t-2}\biggr) \in \mathbb{R}^{N_u^sN^t \times N_u^sN^t} \\
\bm{A}_2^{st} &= \dfrac23 \delta \ diag\biggl(\underbrace{\bm{B}^T, \cdots, \bm{B}^T}_{N^t}\biggr) \in \mathbb{R}^{N_u^sN^t \times N_p^sN^t}  \quad 
\bm{A}_3^{st} = \dfrac23 \delta \ diag\biggl(\underbrace{\bm{C}^T, \cdots, \bm{C}^T}_{N^t}\biggr) \in \mathbb{R}^{N_u^sN^t \times N_\lambda N^t} \notag \\
\bm{A}_4^{st} &=  diag\biggl(\underbrace{\bm{B}, \cdots, \bm{B}}_{N^t}\biggr) \in \mathbb{R}^{N_p^sN^t \times N_u^sN^t}  \qquad \qquad 
\bm{A}_7^{st} =  diag\biggl(\underbrace{\bm{C}, \cdots, \bm{C}}_{N^t}\biggr) \in \mathbb{R}^{N_\lambda N^t \times N_u^sN^t} \notag \\
\bm{R}^{st}(\bm{\mu}) &= \frac23 \delta \ diag\biggl(\underbrace{\bm{R}(\bm{\mu}), \cdots, \bm{R}(\bm{\mu})}_{N^t} \biggr) = 
\frac23 \delta \sum\limits_{q=1}^{N_c} \rho_c^q(\bm{\mu}) \ diag\biggl(\underbrace{\bm{R}^q, \cdots, \bm{R}^q}_{N^t} \biggr) \in \mathbb{R}^{N_u^sN^t \times N_u^sN^t} \notag
\end{align}
Here $diag: \mathbb{R}^{r_1 \times c_1} \times \cdots \times \mathbb{R}^{r_K \times c_K} \to \mathbb{R}^{(r_1+\cdots+r_K)\times(c_1+\cdots+c_K)}$ is the function that builds a block diagonal matrix from a set of $K$ input matrices; $subdiag^{(n)}$ ($n \in \mathbb{N}$) is equivalent to $diag$, but with respect to the $n$--th lower--diagonal. Before reporting the expressions of the blocks appearing in the right--hand side vector of Eq.\eqref{eq: monolitic_FOM_system}, let us make some additional assumptions:
\begin{enumerate}
	\item We assume that $\bm{f}(t) = \bm{0} \ \forall t \in [0,T]$. This implies, for instance, that we neglect the effect of gravity. We also assume that $\bm{h}(t) = \bm{0} \ \forall t \in [0,T]$, which means that we only deal with homogeneous Neumann BCs. 
	\item We assume that $\bm{u}_0 = \bm{0}$ in $\Omega$. Moreover, we suppose that $\bm{u}(t) = \bm{0} \ \forall t \leq 0$; in this way the BDF2 scheme can be employed also in the first iteration.
	\item For all $k \in \{1, \dots, N_D\}$, we assume that the Dirichlet datum $\vec{g}^{k, \bm{\mu}}$ on $\mytilde{\Gamma}_D^k$ can be factorized as:
	\begin{equation}
	\label{eq: factorization_dirichlet_datum}
	\vec{g}^{k, \bm{\mu}}(\bm{x},t) = \vec{g}^s_k(\bm{x})g^t_k(t;\bm{\mu}) \qquad \quad \text{for} \ (\bm{x},t) \in \ \mytilde{\Gamma}_D^k \times [0,T]~.
	\end{equation}
	Thus, we have that $\tilde{\bm{g}}^{k,\bm{\mu}}(t) = \tilde{\bm{g}}^s_k g^t_k(t; \bm{\mu})$, where $\tilde{\bm{g}}^s_k \in \mathbb{R}^{N_{\lambda}^k}$ is such that $\left(\tilde{\bm{g}}^s_k\right)_i = \int_{\mytilde{\Gamma}_D^k} \vec{g}^s_k \cdot \vec{\eta}_i^k$ .
\end{enumerate}
Assumptions 1 and 2 imply that $\bm{F}^{st}_1 = \bm{0} \in \mathbb{R}^{N_u^s}$. Therefore, the only non--zero block in $\bm{F}^{st}$ is $\bm{F}^{st}_3 \in \mathbb{R}^{N_\lambda N^t}$ and it writes as
\begin{equation}
\label{eq: monolitic_FOM_blocks_vector}
\bm{F}^{st}_3(\bm{\mu}) =
\begin{bmatrix}
\tilde{\bm{g}}^{st}(t_1; \bm{\mu}) \\
\vdots \\ 
\tilde{\bm{g}}^{st}(t_{N^t}; \bm{\mu})
\end{bmatrix}
\quad \text{with} \quad \tilde{\bm{g}}^{st}(t_n; \bm{\mu}) =
\begin{bmatrix}
\tilde{\bm{g}}^s_1 g^t_1(t_n; \bm{\mu})\\
\vdots \\ 
\tilde{\bm{g}}^s_{N_D} g^t_{N_D}(t_n; \bm{\mu})
\end{bmatrix}.
\end{equation}

\subsection{Well--posedness of the FOM problem} 
\label{subs: well posedness of the FOM}

As we pointed out in Subsections \ref{subs:strong_weak_stokes}--\ref{subs: discretization_stokes}, Problem \ref{pb:weak_form_stokes_2} features a twofold saddle point structure of type 1. In particular, the velocity is the primal variable, whereas pressure and Lagrange multipliers --- associated to the weak imposition of inhomogeneous Dirichlet BCs --- are the dual ones. As discussed in \cite{gatica2008characterizing,howell2011inf}, necessary and sufficient conditions for its well--posedness can be expressed by means of suitable \emph{inf--sup} conditions, which need to be satisfied both at continuous and at discrete level. As in \cite{pegolotti2021model}, we assume Problem \ref{pb:weak_form_stokes_2} to be well--posed in the continuous setting and we directly focus on the conditions to be enforced at discrete level. Based on classical theoretical results on the well--posedness of saddle point problems (see e.g. \cite{boffi2013mixed, brezzi1974existence, quarteroni2008numerical}), adopting an algebraic standpoint, the following inequality must hold
\begin{equation}
\label{eq: discrete_inf--sup FOM 1}
\exists \beta_F > 0 : \inf\limits_{(\bm{q}, \bm{\lambda}) \neq \bm{0}} \sup\limits_{\bm{v} \neq \bm{0}}  \dfrac{\bm{q}^T\bm{B}\bm{v} + \bm{\lambda}^T \bm{C} \bm{v}}{||\bm{v}||_{\bm{X}_u}\left(||\bm{q}||_{\bm{X}_p} + ||\bm{\lambda}||_{\bm{X}_\lambda}\right)} \geq \beta_F~,
\end{equation}
where $\bm{X}_u$, $\bm{X}_p$ and $\bm{X}_\lambda$ are symmetric and positive--definite matrices, defined as follows:
\begin{equation}
\label{eq: FOM spatial norms}
\bm{X}_u = \frac{1}{\rho}\bm{M} + \frac{1}{2\mu}\bm{A} \in \mathbb{R}^{N_u^s \times N_u^s}~; \qquad \bm{X}_p = \bm{M}^p \in \mathbb{R}^{N_p^s \times N_p^s}~; \qquad \bm{X}_{\lambda} = \bm{I}_{N_\lambda} \in \mathbb{R}^{N_\lambda \times N_\lambda}~.
\end{equation}
Here $\bm{M}^p$ defines the pressure mass matrix, hence $\bm{M}^p_{ij} = \int_{\Omega} \varphi_i^p \varphi_j^p$ for $i,j \in \{1, \dots, N_p^s\}$.

Theorem 3.1 in \cite{howell2011inf} states that Eq.\eqref{eq: discrete_inf--sup FOM 1} can be equivalently expressed in the following ways:
\begin{subequations}
	\begin{align}
	\label{eq: discrete_inf--sup FOM 2a}
	\exists \beta_F>0 : \inf\limits_{\bm{q} \neq \bm{0}} \sup\limits_{\bm{v} \neq \bm{0}}  \dfrac{ \bm{q}^T\bm{B}\bm{v}}{||\bm{v}||_{\bm{X}_u}||\bm{q}||_{\bm{X}_p}} \geq \beta_F \ \ \text{and}& \ \
	\inf\limits_{\bm{\lambda} \neq \bm{0}} \sup\limits_{\substack{\bm{v} \neq \bm{0} \\ \bm{B}\bm{v} = \bm{0}}}
	\dfrac{ \bm{\lambda}^T\bm{C}\bm{v}}{||\bm{v}||_{\bm{X}_u}||\bm{\lambda}||_{\bm{X}_\lambda}} \geq \beta_F~; \\
	\label{eq: discrete_inf--sup FOM 2b}
	\exists \beta_F>0 : \inf\limits_{\bm{q} \neq \bm{0}} \sup\limits_{\substack{\bm{v} \neq \bm{0} \\ \bm{C}\bm{v} = \bm{0}}}  \dfrac{ \bm{q}^T\bm{B}\bm{v}}{||\bm{v}||_{\bm{X}_u}||\bm{q}||_{\bm{X}_p}} \geq \beta_F \ \ \text{and}& \ \
	\inf\limits_{\bm{\lambda} \neq \bm{0}} \sup\limits_{\bm{v} \neq \bm{0}}
	\dfrac{ \bm{\lambda}^T\bm{C}\bm{v}}{||\bm{v}||_{\bm{X}_u}||\bm{\lambda}||_{\bm{X}_\lambda}} \geq \beta_F~.
	\end{align}
\end{subequations}
Therefore, on the one side the \emph{inf--sup} stability of Problem \ref{pb:weak_form_stokes_2} in the discrete setting is guaranteed if two distinct \emph{inf--sup} inequalities --- one for each of the dual fields --- are satisfied. However, on the other side, one of the two \emph{inf--sup} inequalities features an additional constraint on the primal variable. Thus, the independent fulfilment of the conventional \emph{inf--sup} conditions on the dual variables is not enough to guarantee well--posedness.

In this work, we claim the well--posedness of Problem \ref{pb:weak_form_stokes_2} in the discrete framework by considering Eq.\eqref{eq: discrete_inf--sup FOM 2a}, where the first inequality (related to pressure) is ``standard'', while the second one (related to Lagrange multipliers) features a supremum taken over velocities that satisfy the constraint $\bm{B}\bm{v} = \bm{0}$ (i.e. weakly divergence--free velocities). Proceeding as in \cite{pegolotti2021model}, on the one hand we guarantee the fulfilment of the first inequality in Eq.\eqref{eq: discrete_inf--sup FOM 2a} by employing $P2-P1$ Taylor--Hood Lagrangian finite elements \cite{hood1974navier} (i.e continuous piecewise polynomials of order $2$ for the velocity and of order $1$ for the pressure, built on a tetrahedral triangulation of the domain $\Omega$), which are the most popular example of stable discretization couple. On the other hand, instead, we only assess the second inequality empirically, by computing the condition number of the steady Stokes system matrix (see Eq.\eqref{eq: BDF2 matrices} -- right). In all our numerical tests (see Section \ref{sec:section5}) such a matrix proved to be well--conditioned, hence suggesting the well--posedness of the FOM problem at hand.

\section{Space--time reduced basis methods for parametrized unsteady incompressible Stokes equations}
\label{sec:section3}
Classical applications of PROMs to parametrized PDEs only allow to reduce their dimensionality in space. This could represent a significant limitation in problems where either the simulation interval should be very large or the timestep size should be very small in order to properly capture some relevant behaviours. As discussed in Section \ref{sec:introduction}, several attempts to solve this temporal--complexity bottleneck have been made. Among those, in this work we focus on the Space--Time Reduced Basis (ST--RB) methods introduced and analysed e.g. in \cite{urban2012new, urban2014improved, yano2014space, yano2014space2, choi2019space, choi2021space, kim2021efficient}.

\subsection{ST--RB problem definition}
ST--RB methods allow to reduce the dimensionality also in time by projecting the FOM problem onto a low--dimensional spatio--temporal subspace, spanned by a suitable set of basis functions. We denote the spatio--temporal basis functions by $\bm{\pi}_i^u \in \mathbb{R}^{N_u^s N^t}$ ($i \in \{1,\cdots,n_u^{st}\}$) for the velocity, $\bm{\pi}_i^p \in \mathbb{R}^{N_p^s N^t}$ ($i \in \{1,\cdots,n_p^{st}\}$) for the pressure and $\bm{\pi}_i^\lambda \in \mathbb{R}^{N_\lambda N^t}$ ($i \in \{1,\cdots,n_\lambda^{st}\}$) for the Lagrange multipliers. The basis functions are generated by applying a truncated POD algorithm to the matrices storing the FOM solutions got during the so--called offline phase of the method for $N_{\bm{\mu}}$ randomly selected parameter values $\{\bm{\mu}_i\}_{i=1}^{N_{\bm{\mu}}}$, with $\bm{\mu}_i \in \mathcal{D}$. By construction the basis functions are orthonormal with respect to a suitable norm; a more detailed explanation is given in Subsection \ref{subs: basis generation with POD}. The discrete manifold of FOM solutions is approximated by the following low--dimensional subspace
\begin{equation}
\label{eq: approximate spatio--temporal subspace}
\begin{aligned}
\mathcal{ST}_{h,\delta} = \mathcal{ST}_{h,\delta}^u \times \mathcal{ST}_{h,\delta}^p \times \mathcal{ST}_{h,\delta}^\lambda \qquad \text{with} \quad \mathcal{ST}_{h,\delta}^w = span\{\bm{\pi}^w_i\}_{i=1}^{n_w^{st}}, \quad w\in \{u,p,\lambda\}~.
\end{aligned}
\end{equation}
Additionally, we define $n^{st} := n_u^{st} + n_p^{st} + n_\lambda^{st}$ as the dimension of $\mathcal{ST}_{h,\delta}$. The basis of $\mathcal{ST}_{h,\delta}$ can be encoded in the matrix $\bm{\Pi}$, defined as follows:
\begin{equation}
\label{eq: RB basis matrix Pi}
\bm{\Pi} = diag \left(\left[\bm{\pi}^u_1|...|\bm{\pi}^u_{n^{st}_u}\right], \left[\bm{\pi}^p_1|...|\bm{\pi}^p_{n^{st}_p}\right], \left[\bm{\pi}^\lambda_1|...|\bm{\pi}^\lambda_{n^{st}_\lambda}\right] \right) = diag \left( \bm{\Pi}^u, \bm{\Pi}^p, \bm{\Pi}^\lambda \right) \in \mathbb{R}^{N^{st} \times n^{st}}.
\end{equation}
The function $diag$ is defined as in Eq.\eqref{eq: monolitic_FOM_blocks_matrix}. We remark that, thanks to the orthonormality of the basis functions, the matrix $\bm{\Pi}$ identifies an orthogonal projection operator (with respect to a suitable norm), going from the spatio--temporal FOM space of dimension $N^{st}$ to the spatio--temporal reduced subspace $\mathcal{ST}_{h,\delta}$, of dimension $n^{st} \ll N^{st}$.

Adopting an algebraic perspective and employing the compact notation introduced in Eq.\eqref{eq: monolitic_FOM_system}, the application of ST--RB methods amounts at solving the following problem:
\begin{equation}
\label{eq: space time reduced problem}
\text{Find } \widehat{\bm{w}} \in \mathbb{R}^{n^{st}} \text{ such that } \qquad 
\mytilde{\bm{\Pi}}^T \left(\bm{F}^{st} - \bm{A}^{st}\bm{\Pi}\widehat{\bm{w}}\right) = \bm{0},
\end{equation}
where $\mytilde{\bm{\Pi}} \in \mathbb{R}^{N^{st} \times n^{st}}$ is a projection matrix, possibly parameter--dependent. Based on the definition of $\bm{\Pi}$ given in Eq.\eqref{eq: RB basis matrix Pi}, the term $\bm{\Pi}\widehat{\bm{w}}$ represents a linear combination of the reduced basis elements, whose weights are given by the entries of $\widehat{\bm{w}} \in \mathbb{R}^{n^{st}}$. Such term represents then the FOM reconstruction of the space--time reduced solution. So, the quantity $\bm{r}^{st} := \bm{F}^{st} - \bm{A}^{st}\bm{\Pi}\widehat{\bm{w}}$ identifies the FOM residual associated to the space--time reduced solution. In Eq.\eqref{eq: space time reduced problem} we are imposing such a residual to be orthogonal to some low--dimensional spatio--temporal subspace, whose basis is encoded in the matrix $\mytilde{\bm{\Pi}}$. We remark that, if we focus on spatial discretization, Eq.\eqref{eq: space time reduced problem} corresponds to a (Petrov--)~Galerkin projection, equivalent to the one performed in the classical formulation of the RB method, featuring dimensionality reduction only along the spatial dimension \cite{quarteroni2015reduced}. In particular, a Galerkin projection is performed if $\mytilde{\bm{\Pi}} = \bm{\Pi}$ and a Petrov--Galerkin projection otherwise. However, the employment of a finite differences scheme --- which does not stem from the weak formulation of the problem, but rather from the strong one --- to perform time integration prevents from identifying Eq.\eqref{eq: space time reduced problem} with an actual space--time (Petrov--)~Galerkin projection.

Ultimately, rearranging the terms in Eq.\eqref{eq: space time reduced problem}, the ST--RB problem writes as follows.
\begin{myproblem}{}{compact_space_time_reduced_system}
	Find $\widehat{\bm{w}} \in \mathbb{R}^{n^{st}}$ such that
	\begin{equation}
	\begin{aligned}
	\label{eq: compact space time reduced system}
	\widehat{\bm{A}}^{st}\widehat{\bm{w}} =  \widehat{\bm{F}}^{st} \quad \text{with} \quad \widehat{\bm{A}}^{st} &= \  \mytilde{\bm{\Pi}}^T\bm{A}^{st}\bm{\Pi} \ \in \mathbb{R}^{n^{st} \times n^{st}} \\
	\widehat{\bm{F}}^{st} &=  \ \mytilde{\bm{\Pi}}^T\bm{F}^{st} \ \ \ \ \in \mathbb{R}^{n^{st}}~.
	\end{aligned}
	\end{equation}
\end{myproblem}

\subsection{Offline phase: reduced basis generation with POD}
\label{subs: basis generation with POD}

Let us consider a set of $N_{\bm{\mu}}$ FOM velocity snapshots $\{\bm{u}_h^{st}(\bm{\mu}_k)\}_{k=1}^{N_{\bm{\mu}}}$, with $\bm{u}_h^{st}(\bm{\mu}_k) \in \mathbb{R}^{N_u^s\times N^t}$, computed by solving Eq.\eqref{eq: monolitic_FOM_system} for $N_{\bm{\mu}}$ distinct parameter values. These snapshots are stored in a third--order tensor $\bm{\mathcal{X}}^u \in \mathbb{R}^{N_u^s \times N^t \times N_{\bm{\mu}}}$ so that
\begin{equation}
\label{eq: snapshots_tensor}
\bm{\mathcal{X}}^u_{ijk} = \left(\bm{u}_h^{st}(\bm{\mu}_k)\right)_{ij}, \qquad i \in \{1,\cdots,N_u^s\}, \ j \in \{1,\cdots,N^t\}, \  k \in \{1,\cdots,N_{\bm{\mu}}\}.
\end{equation}
Firstly, let us focus on the construction of the reduced basis in space, that at discrete level can be encoded in the matrix $\bm{\Phi}^u \in \mathbb{R}^{N_u^s\times n_u^s}$. We want the reduced basis to be orthonormal with respect to the norm induced by the matrix $\bm{X}_u$ (see Eq.\eqref{eq: FOM spatial norms}), which is indeed the $\left(H^1(\Omega)\right)^d$--norm. Denoting with $\bm{H}_u \in \mathbb{R}^{N_u^s \times N_u^s}$ the upper triangular matrix arising from the Cholesky decomposition of $\bm{X}_u$ (i.e. $\bm{X}_u = \bm{H}_u^T \bm{H}_u$) and with $\bm{\mathcal{X}}^u_{(1)} \in \mathbb{R}^{N_u^s \times N^t N_{\bm{\mu}}}$ the mode--1 unfolding of $\bm{\mathcal{X}}^u$ (see \cite{choi2019space}), we perform the Singular Value Decomposition of $\bm{H}_u \bm{\mathcal{X}}^u_{(1)}$, so
\begin{equation}
\label{eq: spatial SVD velocity}
\bm{H}_u \bm{\mathcal{X}}_{(1)}^u = \bm{V} \bm{\Sigma} \bm{Z}^T .
\end{equation}
$\bm{V} \in \mathbb{R}^{N_u^s \times N_u^s}$ and $\textbf{Z} \in \mathbb{R}^{N^t N_{\bm{\mu}}\times N^t N_{\bm{\mu}}}$ are orthogonal matrices, whereas $\bm{\Sigma} \in \mathbb{R}^{N_u^s \times N^t N_{\bm{\mu}}}$ is the pseudo--diagonal matrix storing the singular values $\{\sigma_i\}_{i=1}^{N_\sigma}$ of $\bm{H}_u \bm{\mathcal{X}}^u_{(1)}$ (with $N_\sigma := \min(N_u^s, N^t N_{\bm{\mu}})$ and $\sigma_i \geq 0 \ \forall i \in \{1,\cdots,N_\sigma\}$). If the singular values are sorted in decreasing order and if we denote by $\mytilde{\bm{\Phi}}^u \in \mathbb{R}^{N_u^s \times n_u^s}$ the matrix formed by the first $n_u^s \ll N_u^s$ columns of $\bm{V}$, then the velocity reduced basis in space can be computed as $\bm{\Phi}^u = \bm{H}_u^T \mytilde{\bm{\Phi}}^u = [\bm{\phi}^u_1|\cdots|\bm{\phi}^u_{n_u^s}]$. The columns $\{\bm{\phi}_i^u\}_{i=1}^{n_u^s}$ of the matrix $\bm{\Phi}^u$ represent the $n_u^s$--dimensional orthonormal basis that minimizes the total projection error of the snapshots --- with respect to the norm induced by $\bm{X}_u$ --- onto the column space of $\bm{\mathcal{X}}^u_{(1)}$ \cite{quarteroni2015reduced}. A common strategy consists in selecting $n_u^s$ as the smallest integer $N$ such that:
\begin{equation}
\label{eq: POD tolerance}
\dfrac{\sum_{j=1}^{N} \sigma_j^2}{\sum_{j=1}^{N_\sigma} \sigma_j^2} \geq 1 - \varepsilon_u^2~,
\end{equation}
where $\varepsilon_u \in \mathbb{R}^+$ is a tolerance to be chosen \emph{a priori}. The left--hand side of Eq.\eqref{eq: POD tolerance} represents the relative information (or energy) content of the POD basis.
\begin{remark}
Eq.\eqref{eq: spatial SVD velocity} represents a huge SVD problem, that oftentimes is computationally prohibitive. However, we are only interested in performing a truncated SVD, i.e. in computing the $n_u^s \ll N_u^s$ most significant modes of $\bm{H}_u\bm{\mathcal{X}}^u_{(1)}$ and the associated singular values. Such a task can be performed at a reasonable computational cost via iterative \cite{golub2013matrix} and/or randomized algorithms \cite{halko2011finding}. In this work, we employed the randomized POD algorithm proposed in \cite{halko2011finding}.
\end{remark} 

Among the possible strategies to compute the velocity reduced basis in time, we choose the \emph{Fixed temporal subspace via ST-HOSVD} algorithm, proposed in \cite{choi2019space}. It consists in performing a truncated POD (with tolerance $\varepsilon_u$) to the mode--2 unfolding of the spatial projection of the snapshots' tensor $\bm{\mathcal{X}}^u(\bm{\Phi}^u)_{(2)} \in \mathbb{R}^{N^t \times n_u^s N_{\bm{\mu}}}$. The third--order tensor $\bm{\mathcal{X}}^u(\bm{\Phi}^u) \in \mathbb{R}^{n_u^s \times N^t \times N_{\bm{\mu}}}$ is defined as $\bm{\mathcal{X}}^u(\bm{\Phi}^u) = \bm{\mathcal{X}}^u \times_1 \bm{\Phi}^u$, so that $\bm{\mathcal{X}}^u(\bm{\Phi}^u)_{(1)} = (\bm{\Phi}^u)^T \bm{X}^u \bm{\mathcal{X}}^u_{(1)}$. The number $n_u^t$ of temporal reduced basis elements is chosen according to the criterion in Eq.\eqref{eq: POD tolerance} and the velocity temporal reduced basis is encoded in the matrix $\bm{\Psi}^u = [\bm{\psi}^u_1|\cdots|\bm{\psi}^u_{n_u^t}] \in \mathbb{R}^{N^t \times n_u^t}$, such that $\left(\bm{\Psi}^u\right)^T \bm{\Psi}^u = \bm{I}_{n_u^t}$. 

In order to derive an expression for the spatio--temporal reduced basis, we work under a space--time factorization assumption and we suppose that the discrete FOM solution manifold can be well approximated by the low--dimensional vector space
\begin{equation}
\label{eq: space--time subspace factorization}
\mathcal{ST}^u_{h,\delta} = \mathcal{S}^u_h \otimes \mathcal{T}^u_\delta \qquad \text{with} \ \ \mathcal{S}^u_h = span\big\{\bm{\phi}_i^u\big\}_{i=1}^{n_u^s} \ \ \text{and} \ \ \mathcal{T}^u_\delta = span\big\{\bm{\psi}_j^u\big\}_{j=1}^{n_u^t}~.
\end{equation}
A generic element of the velocity space--time reduced basis can then be written as
\begin{equation}
\label{eq: space--time RB element velocity}
\bm{\pi}^u_{\mathcal{F}_u(i,j)} = \bm{\phi}_i^u \otimes \bm{\psi}_j^u \in \mathbb{R}^{N_u^s \times N^t} \qquad i \in \{1,\cdots,n_u^s\}, \ \ j \in \{1,\cdots,n_u^t\}~,
\end{equation}
where $\mathcal{F}_u: (i,j) \mapsto (i-1)n_u^t + j$ is a bijective mapping from the space and time bases indexes to the space--time basis index and $\otimes: \mathbb{R}^N \phantom{\hspace{-2pt}} \times \mathbb{R}^M \to \mathbb{R}^{N \times M}$ (with $N,M \in \mathbb{N}$) denotes the outer product operator, i.e. $\left(\bm{u} \otimes \bm{v}\right)_{ij} = \bm{u}_i \bm{v}_j$.

The same procedure is followed to assemble the reduced bases for the other unknowns in Eq.\eqref{eq: monolitic_FOM_system}, i.e. pressure and Lagrange multipliers. In particular:
\begin{itemize}
	\item For the pressure, we define $\mathcal{ST}^p_{h,\delta} = \mathcal{S}^p_h \otimes \mathcal{T}^p_\delta$. We orthonormalize the reduced basis in space with respect to the $L^2(\Omega)$--norm (see Eq.\eqref{eq: FOM spatial norms}).The generic pressure space--time reduced basis element reads as
	\begin{equation}
	\label{eq: space--time RB element pressure}
	\bm{\pi}^p_{\mathcal{F}_p(i,j)} = \bm{\phi}_i^p \otimes \bm{\psi}_j^p \in \mathbb{R}^{N_p^s \times N^t} \qquad i \in \{1,\cdots n_p^s\}, \ \ j \in \{1,\cdots n_p^t\}~,
	\end{equation}
	with $\mathcal{F}_p: (i,j) \mapsto (i-1)n_p^t + j$. $\varepsilon_p \in \mathbb{R}^+$ is the pressure POD tolerance (in space and in time). We denote the pressure reduced basis in space by $\bm{\Phi}^p \in \mathbb{R}^{N_p^s \times n_p^s}$ and the pressure reduced basis in time by $\bm{\Psi}^p \in \mathbb{R}^{N^t \times n_p^t}$.
	\item We compute a different reduced basis for each set of Lagrange multipliers, corresponding to different portions of $\mytilde{\Gamma}_D$. So, for all $k \in \{1, \dots, N_D\}$, we define the spaces $\mathcal{ST}_{h,\delta}^{\lambda_k}$ such that $\mathcal{ST}_{h,\delta}^\lambda = \prod_{k=1}^{N_D} \mathcal{ST}_{h,\delta}^{\lambda_k}$, with $\mathcal{ST}^{\lambda_k}_{h,\delta} = \mathcal{S}^{\lambda_k}_h \otimes \mathcal{T}_\delta^{\lambda_k}$. Since the space of Lagrange multipliers has been discretized by means of a relatively small number of basis functions (see~\cite{pegolotti2021model} for details), we only compute the temporal reduced bases. So, we define the generic space--time reduced basis element for the $k$--th set of Lagrange multipliers as:
	\begin{equation}
	\label{eq: space--time RB element lagmult}
	\bm{\pi}^{\lambda_k}_{\mathcal{F}_{\lambda_k}(i,j)} = \bm{e}_i \otimes \bm{\psi}_j^{\lambda_k} \in \mathbb{R}^{N_{\lambda}^k \times N^t} \qquad i \in \{1,\cdots,N_{\lambda}^k\}, \ \ j \in \{1,\cdots n_{\lambda_k}^t\}~, 
	\end{equation}
	where $\bm{e}_i \in \mathbb{R}^{N_\lambda^k}$ is the $i$--th canonical basis vector and $\mathcal{F}_{\lambda_k}: (i,j) \mapsto (i-1)n_{\lambda_k}^t + j$. We define the dimension of the space--time reduced basis as $n_{\lambda_k}^{st} := N_\lambda^k n_{\lambda_k}^t$. The space--time reduced basis for $\mathcal{ST}_{h,\delta}^{\lambda}$ can be then assembled exploiting the definition of the latter and its dimension is equal to $n_\lambda^{st} := \sum_{k=1}^{N_D} n_{\lambda_k}^{st}$. We select the same POD tolerance $\epsilon_\lambda \in \mathbb{R}^+$ for every set of multipliers. We denote the reduced basis in space as $\bm{\Phi}^\lambda = \bm{I}_{N_\lambda}$ --- since no reduction in space takes place --- and the reduced basis in time as $\bm{\Psi}^{\lambda} = diag\left(\bm{\Psi}^{\lambda_1}, \cdots,\bm{\Psi}^{\lambda_{N_D}}\right) \in \mathbb{R}^{N^t \times n_\lambda^t}$, with $n_\lambda^t := \sum_{k=1}^{N_D} n_{\lambda_k}^t$. Finally, for the sake of conciseness we define $\mathcal{F}_\lambda : \mathbb{R}^{N_\lambda} \times \mathbb{R}^{n_\lambda^t} \to \mathbb{R}^{n_{\lambda}^{st}}$ as the global index mapping for Lagrange multipliers.
\end{itemize}

Ultimately, the global space--time reduced basis can be encoded in the matrix $\bm{\Pi} \in \mathbb{R}^{N^{st} \times n^{st}}$ defined in Eq.\eqref{eq: RB basis matrix Pi}, where $\bm{\Pi}^u = \bm{\Phi}^u \otimes \bm{\Psi}^u, \ \ \bm{\Pi}^p = \bm{\Phi}^p \otimes \bm{\Psi}^p, \ \ \bm{\Pi}^{\lambda} = \bm{I}_{N_{\lambda}} \otimes \bm{\Psi}^{\lambda}$.

\subsection{Offline Phase: assembling of parameter--independent quantities}
\label{subs: assembling parameter-independent quantities}
The second step of the offline phase consists in assembling, once and for all, the space--time reduced parameter--independent quantities. In this work, we consider a problem featuring an affinely parametrized left--hand side term and non--affinely parametrized Dirichlet data, whose information is stored in the right--hand side vector. In the general case, approximate affine decompositions can be retrieved by exploiting the (M)DEIM algorithm \cite{chaturantabut2010nonlinear}. Let us define the parameter--independent space--reduced matrices
\begin{equation}
\label{eq: ROM matrices}
\begin{alignedat}{4}
\widehat{\bm{A}} &= (\bm{\Phi}^u)^T\bm{A}\bm{\Phi}^u  \in \mathbb{R}^{n_u^s \times n_u^s} \quad &
\widehat{\bm{B}}^T &= (\bm{\Phi}^u)^T\bm{B}^T\bm{\Phi}^p  \in \mathbb{R}^{n_u^s \times n_p^s}
\quad & 
\widehat{\bm{B}} &= (\bm{\Phi}^p)^T\bm{B}\bm{\Phi}^u  \in \mathbb{R}^{n_u^s \times n_p^s} 
\\
\widehat{\bm{M}} &= (\bm{\Phi}^u)^T\bm{M}\bm{\Phi}^u  \in \mathbb{R}^{n_u^s \times n_u^s}
\quad &
\widehat{\bm{C}}^T &= (\bm{\Phi}^u)^T\bm{C}^T  \in \mathbb{R}^{n_u^s \times N_\lambda}
\quad &
\widehat{\bm{C}} &= \bm{C}\bm{\Phi}^u  \in \mathbb{R}^{N_\lambda \times n_u^s}
\end{alignedat}
\end{equation}
and the time--reduced matrices
\begin{equation}
\label{eq: temporal bases combined}
\begin{alignedat}{2}
\bm{\Psi}^{u,p} &= (\bm{\Psi}^u)^T\bm{\Psi}^p \ \in \mathbb{R}^{n_u^t \times n_p^t} \qquad \qquad & \bm{\Psi}^{u,\lambda} &= (\bm{\Psi}^u)^T\bm{\Psi}^\lambda \ \in \mathbb{R}^{n_u^t \times n_\lambda^t}~.
\end{alignedat}
\end{equation}
Also, we define the space--reduced affine components of the parametrized reaction matrix
\begin{equation}
\label{eq: ROM matrix reaction}
\widehat{\bm{R}}^q = \left(\bm{\Phi}^u\right)^T \bm{R}^q \bm{\Phi}^u \in \mathbb{R}^{n_u^s \times n_u^s} 
\qquad \text{such that} \qquad 
\widehat{\bm{R}}(\bm{\mu}) = \left(\bm{\Phi}^u\right)^T \bm{R}(\bm{\mu}) \bm{\Phi}^u = \sum_{q=1}^{N_{c}} \rho_c^q (\bm{\mu}) \widehat{\bm{R}}^q~.
\end{equation}
Leveraging the saddle point structure of $\bm{A}^{st}$ (see Eq.\eqref{eq: monolitic_FOM_blocks_matrix}), the space--time reduced left--hand side matrix $\widehat{\bm{A}}^{st}$ can be written as follows
\begin{equation}
\begin{aligned}
\label{eq: reduced A}
\widehat{\bm{A}}^{st} 
&=
\begin{bmatrix}
\big(\mytilde{\bm{\Pi}}^u\phantom{\hspace{-1mm}}\big)^T \bm{A}^{st}_1 \bm{\Pi}^u & \big(\mytilde{\bm{\Pi}}^u\phantom{\hspace{-1mm}}\big)^T\bm{A}^{st}_2\bm{\Pi}^p
& \big(\mytilde{\bm{\Pi}}^u\phantom{\hspace{-1mm}}\big)^T \bm{A}^{st}_3\bm{\Pi}^\lambda \\
\big(\mytilde{\bm{\Pi}}^p\phantom{\hspace{-1mm}}\big)^T\bm{A}^{st}_4\bm{\Pi}^u & & \\
\big(\mytilde{\bm{\Pi}}^\lambda\phantom{\hspace{-1mm}}\big)^T\bm{A}^{st}_7\bm{\Pi}^u & & 
\end{bmatrix}
+
\begin{bmatrix}
\big(\mytilde{\bm{\Pi}}^u\phantom{\hspace{-1mm}}\big)^T \bm{R}^{st}(\bm{\mu})\bm{\Pi}^u & & \\
& & \\
& & 
\end{bmatrix}\\
&= 
\begin{bmatrix}
\widehat{\bm{A}}_1^{st} & \widehat{\bm{A}}_2^{st}
& \widehat{\bm{A}}_3^{st} \\
\widehat{\bm{A}}_4^{st} & & \\
\widehat{\bm{A}}_7^{st} & & 
\end{bmatrix}
+
\begin{bmatrix}
\widehat{\bm{R}}^{st}(\bm{\mu}) & & \\
& & \\
& & 
\end{bmatrix}~.
\end{aligned}
\end{equation}
If the left--hand side projection matrices $\mytilde{\bm{\Pi}}^u$, $\mytilde{\bm{\Pi}}^p$, $\mytilde{\bm{\Pi}}^\lambda$ are parameter--independent (as for instance in the case of a Galerkin projection), all the parameter--independent blocks of $\widehat{\bm{A}}^{st}$ can be efficiently computed once and for all, by combining the matrices in Eqs.\eqref{eq: ROM matrices}--\eqref{eq: temporal bases combined} as follows:
\begin{equation}
\label{eq: ROM spacetime blocks}
\begin{alignedat}{2}
\left(\widehat{\bm{A}}_1^{st}\right)_{\ell m} &= \left(\widehat{\bm{M}}+\frac23\delta\widehat{\bm{A}}\right)_{\ell_s m_s}\delta_{\ell_t,m_t}
-\dfrac{4}{3}\widehat{\bm{M}}_{\ell_s m_s}(\bm{\psi}^u_{\ell_t})_{2:}^T(\bm{\psi}^u_{m_t})_{:-1}
+\dfrac{1}{3}\widehat{\bm{M}}_{\ell_s m_s}(\bm{\psi}^u_{\ell_t})_{3:}^T(\bm{\psi}^u_{m_t})_{:-2} \\
\left(\widehat{\bm{A}}^{st}_2\right)_{\ell k} &= \frac23 \delta \widehat{\bm{B}}^T_{\ell_s k_s} \bm{\Psi}_{\ell_t k_t}^{u,p} 
\qquad \qquad \quad 
\left(\widehat{\bm{A}}^{st}_4\right)_{k m} = \widehat{\bm{B}}_{k_s m_s} \bm{\Psi}_{m_t k_t}^{u,p} \\
\left(\widehat{\bm{A}}^{st}_3\right)_{\ell j} &= \frac23 \delta \widehat{\bm{C}}^T_{\ell_s j_s} \bm{\Psi}_{\ell_t j_t}^{u,\lambda}
\qquad \qquad \quad \
\left(\widehat{\bm{A}}^{st}_7\right)_{j m} = \widehat{\bm{C}}_{j_s m_s} \bm{\Psi}_{m_t j_t}^{u,\lambda}
\end{alignedat}
\end{equation}
for $\ell = \mathcal{F}_u(\ell_s, \ell_t)$, $m = \mathcal{F}_u(m_s, m_t)$ ($\ell_s,m_s \in \{1,\cdots n_u^s\}$, $\ell_t,m_t \in \{1,\cdots n_u^t\}$), $k = \mathcal{F}_p(k_s, k_t)$ ($k_s \in \{1,\cdots n_p^s\}$, $k_t \in \{1,\cdots n_p^t\}$), $j = \mathcal{F}_\lambda(j_s, j_t)$ ($j_s \in \{1,\cdots N_\lambda\}$, $j_t \in \{1,\cdots n_\lambda^t\}$). The notations $\bm{v}_{i:}$, $\bm{v}_{:-i}$ denote the sub--vector of a given vector $\bm{v}$ containing all the entries from the $i$--th to the last one and from the first one to the $i$--th from last, respectively.

\noindent Exploiting the affine parametrization of the reaction term (see Eq.\eqref{eq: ROM matrix reaction}), it also holds that
\begin{equation}
\label{eq: ROM matrix reaction ST}
\widehat{\bm{R}}_1^{st}(\bm{\mu}) = \sum_{q=1}^{N_c} \rho_c^q (\bm{\mu}) \widehat{\bm{R}}_q^{st} \qquad \text{where} \qquad \widehat{\bm{R}}_q^{st} = \big(\mytilde{\bm{\Pi}}^u\phantom{\hspace{-1mm}}\big)^T \bm{R}^q \bm{\Pi}^u \in \mathbb{R}^{n_u^s n_u^t \times n_u^s n_u^t}~.
\end{equation}
Hence, the matrices $\{\widehat{\bm{R}}_q^{st}\}_{q=1}^{N_c}$ can be pre--assembled during the offline phase, drastically lightening the computational burden. We refer to Appendix \ref{appendix: ST--RB assembling} for a more detailed explanation of the assembling phase, in the particular case of a Galerkin projection.

\begin{remark}
The assembling of the space--time reduced blocks in Eqs.\eqref{eq: ROM spacetime blocks}--\eqref{eq: ROM matrix reaction ST} reveals the advantage of exploiting the space--time factorization paradigm (see Eq.\eqref{eq: approximate spatio--temporal subspace}). Indeed, the issue posed by considering a non--factorized space--time reduced subspace is not related to the reduced bases computation, but rather it stems from the projection operations. For instance, let us focus on the assembling of the space--reduced blocks $\{\widehat{\bm{A}}_j^{st}\}_j$ in Eq.\eqref{eq: ROM spacetime blocks}.  
The adoption of a non--factorized approach imposes to compute the projection of the full--order space--time blocks $\{\bm{A}_j^{st}\}_j$ onto the space--time reduced subspace. For any $j$, the inner block structure of $\bm{A}_j^{st}$ can be leveraged  and the explicit storage of the matrix can be avoided by employing \emph{ad hoc} streaming techniques for the projection computation. Nevertheless, the cost of the latter remains huge in practical applications, since it involves matrix--matrix multiplications that depend on the number of space--time full--order DOFs. Taking advantage of space--time factorization allows to drastically lighten the (offline) computational burden. Indeed, space--reduced and time--reduced quantities can be pre--assembled independently and then suitably combined at a later stage (see Eq.\eqref{eq: ROM spacetime blocks}). Ultimately, considering non--factorized space--time reduced subspaces may provide better approximation properties, but the offline computational cost imposes restrictions to the actual applicability of the method. 
\end{remark}

\subsection{Online phase}
During the online phase, we are interested in computing the solution to the problem at hand for a given parameter value $\bm{\mu}^* \in \mathcal{D}$. This comprises three steps:
\begin{enumerate}
	\item The assembling of the reduced parameter--dependent quantities.
	\item The computation of the reduced solution.
	\item The reconstruction of an approximate FOM solution from the space--time reduced one.
\end{enumerate}

Concerning the first step, we refer the reader to Appendix \ref{appendix: ST--RB assembling}. Upon having constructed $\widehat{\bm{A}}^{st}$ and $\widehat{\bm{F}}^{st}$, we are left with solving the $n^{st}$--dimensional dense linear system of Eq.\eqref{eq: monolitic_FOM_system}. Since $n^{st} \ll N^{st}$, significant speedups with respect to the FOM can be realized.

Finally, once the space--time reduced solution $\widehat{\bm{w}}(\bm{\mu}^*) \in \mathbb{R}^{n^{st}}$ is computed, it can be post--processed and re--projected onto the FOM space. This task can be efficiently performed if the vector $\widehat{\bm{w}}(\bm{\mu}^*) \in \mathbb{R}^{n^{st}}$ is suitably reshaped into a matrix $\widehat{\bm{w}}^M(\bm{\mu}^*) \in \mathbb{R}^{n^s \times n^t}$, being $n^s = n_u^s+ n_p^s + N_\lambda$ and $n^t = n_u^t + n_p^t + n_\lambda^t$. Indeed, the FOM reconstruction $\bm{w}_h^{st}(\bm{\mu}^*)$ of $\widehat{\bm{w}}(\bm{\mu}^*)$ writes as follows:
\begin{equation}
\label{eq: FOM reconstruction}
\bm{w}_h^{st}(\bm{\mu}^*) = \bm{\Phi} \widehat{\bm{w}}^M(\bm{\mu}^*) \bm{\Psi}^T \quad \in \mathbb{R}^{N^s \times N^t}~,
\end{equation}
where $\bm{\Phi} := diag\left(\bm{\Phi}^u, \bm{\Phi}^p, \bm{\Phi}^\lambda\right) \in \mathbb{R}^{N^s \times n^s}$ and $\bm{\Psi} := diag\left(\bm{\Psi}^u, \bm{\Psi}^p, \bm{\Psi}^\lambda\right) \in \mathbb{R}^{N^t \times n^t}$ are the global reduced bases in space and in time, respectively.

\subsection{Definition of the norms}
\label{subs: definition of the norms}
Concerning the spatial dimension, we already introduced the norms that we employed in Subsection \ref{subs: basis generation with POD} and we defined the corresponding matrices in Eq.\eqref{eq: FOM spatial norms}.

Since temporal reduced bases are derived by imposing orthonormality in the Euclidean norm, we can define spatio--temporal norms as the ones induced by the following matrices:
\begin{equation}
\label{eq: ST norms matrices}
\bm{X}_u^{st} = diag\biggl(\underbrace{\bm{X}_u, \cdots, \bm{X}_u}_{N^t}\biggr), \quad
\bm{X}_p^{st} = diag\biggl(\underbrace{\bm{X}_p, \cdots, \bm{X}_p}_{N^t}\biggr), \quad 
\bm{X}_\lambda^{st} = diag\biggl(\underbrace{\bm{X}_\lambda, \cdots, \bm{X}_\lambda}_{N^t}\biggr).
\end{equation}
The global spatio--temporal norm matrix is then given by $\bm{X}^{st} := diag(\bm{X}_u^{st}, \bm{X}_p^{st}, \bm{X}_\lambda^{st}) \in \mathbb{R}^{N^{st} \times N^{st}}$.

Let us consider $\bm{\varphi}_j = \mathit{vec}(\bm{\phi}_j \otimes \bm{\psi}_j) \in \mathbb{R}^{N^{st}}$ ($j \in \mathbb{N})$, being $\bm{\phi}_j \in \mathbb{R}^{N^s}$ the vector of DOFs arising from the FE discretization of a spatial function $\phi_j=\phi_j(\bm{x})$ and $\bm{\psi}_j \in \mathbb{R}^{N^t}$ the vector storing the evaluations of a temporal function $\psi_j=\psi_j(t)$ at the equispaced time instants $\{t_n\}_{n=1}^{N^t}$ in $[0,T]$. Here $\mathit{vec}: \mathbb{R}^{N^s \times N^t} \to \mathbb{R}^{N^{st}}$ denotes the vectorization operator. Then, we have that:
\begin{equation}
\label{eq: space--time inner product factorization}
\left(\bm{\varphi}_1, \bm{\varphi}_2\right)_{\bm{X}^{st}} = \sum\limits_{n=1}^{N^t} \left(\bm{\phi}_1(\bm{\psi}_1)_n, \bm{\phi}_2(\bm{\psi}_2)_n\right)_{\bm{X}} = \left(\bm{\phi}_1, \bm{\phi}_2\right)_{\bm{X}}  \sum\limits_{n=1}^{N^t} (\bm{\psi}_1)_n (\bm{\psi}_2)_n~,
\end{equation}
where $\bm{X}^{st} \in \mathbb{R}^{N^{st} \times N^{st}}$ is a block--diagonal matrix constructed from the symmetric and positive definite norm matrix $\bm{X} \in \mathbb{R}^{N^s \times N^s}$, as the ones in Eq.\eqref{eq: ST norms matrices}. A consequence of Eq.\eqref{eq: space--time inner product factorization} is that $\left(\bm{\varphi}_j, \bm{\varphi}_j\right)_{\bm{X}^{st}} = \vert\vert \bm{\varphi}_j \vert\vert_{\bm{X}^{st}}^2 = \vert\vert \bm{\phi}_j \vert\vert_{\bm{X}}^2 \vert\vert \bm{\psi}_j \vert\vert_2^2$. So, the spatio--temporal norm factorizes into the product between the norms of the spatial and of the temporal factors. Incidentally, notice that Eq.\eqref{eq: space--time inner product factorization} entails that the reduced bases encoded by the columns of the matrices $\bm{\Pi}^u$, $\bm{\Pi}^p$, $\bm{\Pi}^\lambda$ are orthonormal with respect to the norms induced by the matrices in Eq.\eqref{eq: ST norms matrices}.

\subsection{Well-posedness of the ST--RB method}
\label{subs: well posedness of the ST--RB methods}
In Subsection \ref{subs: well posedness of the FOM}, we highlighted that Problem \ref{pb:weak_form_stokes_2} features a (twofold) saddle point structure and, as a consequence, it is associated with stability issues, related to the discretization of the spaces of the primal (velocity) and dual (pressure and Lagrange multipliers) fields. Assuming Problem \ref{pb:weak_form_stokes_2} to be well--posed in the continuous setting, well--posedness in the discrete framework can be retained by satisfying the two \emph{inf--sup} conditions in Eqs.\eqref{eq: discrete_inf--sup FOM 2a}-\eqref{eq: discrete_inf--sup FOM 2b}. See Subsection \ref{subs: well posedness of the FOM} for details. However, even if a stable discretization is considered for the FOM, there is no guarantee for the \emph{inf--sup} conditions to hold also for the reduced system (see e.g. \cite{deparis2008reduced,deparis2009reduced,negri2015reduced,rozza2013reduced}). The literature presents several possibilities to deal with the loss of stability of saddle point problems in the context of model order reduction in space. In this work, we considered two of them, namely the supremizers enrichment \cite{rozza2005optimization, ballarin2015supremizer, dalsanto2019hyper} and the employment of least--squares Petrov--Galerkin reduced basis (LS--PG--RB) approaches for residual minimization \cite{dalsanto2019algebraic, carlberg2017galerkin}. In the framework of space--time model order reduction, these two strategies lead to the development of the ST--GRB and ST--PGRB methods for unsteady parametrized incompressible Stokes equations, respectively.

\subsubsection{Velocity reduced basis enrichment}
\label{subs: ST--GRB supremizers}
before tackling space--time model order reduction, let us focus on the stability of the space--reduced formulation. To this aim, we consider the supremizers enrichment approach, whose central principle is to augment the reduced basis for the velocity with additional elements (called \emph{supremizers}) that are computed to ensure \emph{inf--sup} stability also in the reduced framework. 
We remark that, since the ``coupling'' matrices $\bm{B}$ (for pressure) and $\bm{C}$ (for Lagrange multipliers) are not parameter--dependent in the current setting, exact supremizers can be found. In the general case where parametric dependency involves the ``coupling'' matrices (as for instance upon a parametrized geometric transformation of the domain \cite{pegolotti2021model}), approximate supremizers should be computed, in order for the velocity reduced basis to be parameter--independent \cite{ballarin2015supremizer}. To guarantee well--posedness, the two following \emph{inf--sup} inequalities (reduced counterpart of Eq.\eqref{eq: discrete_inf--sup FOM 2b}) have to be satisfied:
\begin{equation}
\label{eq: discrete inf--sup ROM}
\exists \beta_{R}>0 : \inf\limits_{\widehat{\bm{q}} \neq \bm{0}} \sup\limits_{\substack{\widehat{\bm{v}} \neq \bm{0} \\ \widehat{\bm{C}}\widehat{\bm{v}} = \bm{0}}}  \dfrac{\widehat{\bm{q}}^T\widehat{\bm{B}}\widehat{\bm{v}}}{||\widehat{\bm{v}}||_2||\widehat{\bm{q}}||_2} \geq \beta_{R} \ \ \text{and} \ \
\inf\limits_{\widehat{\bm{\lambda}} \neq \bm{0}} \sup\limits_{\widehat{\bm{v}} \neq \bm{0}}
\dfrac{\widehat{\bm{\lambda}}^T\widehat{\bm{C}}\widehat{\bm{v}}}{||\widehat{\bm{v}}||_2||\widehat{\bm{\lambda}}||_2} \geq \beta_{R}~.
\end{equation}
Here $||\cdot||_2$ denotes the Euclidean norm, which is identical to the ones induced by the matrices in Eq.\eqref{eq: FOM spatial norms} because of the orthonormality properties of the reduced bases (see Subsection \ref{subs: basis generation with POD}).\\
Since the problem at hand features a twofold saddle point structure, two distinct sets of supremizers are computed. The first one --- denoted as $\mathcal{S}_h^p := \left\{\bm{s}_j^{u,p}\right\}_{j=1}^{n_p^s}$ --- is assembled by selecting, for each pressure mode $\bm{\phi}_j^p$, the velocity $\bm{s}_j^{u,p}$ that allows to attain the supremum in the pressure \emph{inf--sup} inequality of Eq.\eqref{eq: discrete_inf--sup FOM 2b}.  Its elements are computed from the solutions to the following set of linear systems, featuring a (onefold) saddle point structure:
\begin{equation}
\label{eq: supremizers pressure}
\begin{bmatrix}
\bm{X}_u  & \bm{C}^T \\
\bm{C} &
\end{bmatrix}
\begin{bmatrix}
\bm{s}_j^{u,p} \\ \bm{\lambda}_j
\end{bmatrix} 
= 
\begin{bmatrix}
\bm{B}^T \bm{\phi}_j^p \\ \quad
\end{bmatrix} \qquad \quad \text{with } j \in \{1,\cdots, n_p^s\}~.
\end{equation}
We define $\mathcal{S}_h^{u, p+} = span\left\{\left\{\bm{\phi}_i^u\right\}_{i=1}^{n_u^s}, \left\{\bm{s}_j^{u,p}\right\}_{j=1}^{n_p^s}\right\}$ as the space--reduced velocity subspace, enriched with pressure supremizers. The columns of the matrix $\bm{\Phi}^{u, p+} = \left[\bm{\phi}_1^u\vert\cdots\vert\bm{\phi}_{n_u^s}^u\vert\bm{s}^{u,p}_1\vert\cdots\vert\bm{s}^{u,p}_{n_s^p}\right] \in \mathbb{R}^{N_u^s \times (n_u^s+n_p^s)}$ are then a basis of $\mathcal{S}_h^{u, p+}$. 

The second set of supremizers is instead constructed from the bases of the spaces of Lagrange multipliers $\mathcal{L}_h^k$. Incidentally, it is worth pointing out that the problem at hand actually features a $(N_D+1)$--fold saddle point structure, rather than a twofold one; indeed $N_D$ distinct dual fields are defined to weakly impose inhomogeneous Dirichlet BCs. Therefore, according to \cite{gatica2008characterizing}, the second \emph{inf--sup} inequality in Eq.\eqref{eq: discrete_inf--sup FOM 2b} should be rewritten in terms of the local coupling matrices $\{\bm{C}^k\}_{k=1}^{N_D}$ as follows:
\begin{equation}
\label{eq: discrete inf--sup FOM 3b}
\forall k \in \{1, \cdots, N_D\} \ \ \exists \beta_F^k>0 : \ \ \inf\limits_{\bm{\lambda}_k \neq \bm{0}} \sup\limits_{\substack{\bm{v} \neq \bm{0} \\ \bm{C}^j \bm{v} = \bm{0} \ \forall j < k}}
\dfrac{ \bm{\lambda}_k^T\bm{C}^k\bm{v}}{||\bm{v}||_{\bm{X}_u}||\bm{\lambda}_k||_{\bm{X}_{\lambda_k}}} \geq \beta_F^k~.
\end{equation}
However, for the problem that we are considering, Eq.\eqref{eq: discrete inf--sup FOM 3b} can be equivalently expressed as
\begin{equation}
\label{eq: discrete inf--sup FOM 4b}
\forall k \in \{1, \cdots, N_D\} \ \ \exists \beta_F^k>0 : \ \ \inf\limits_{\bm{\lambda}_k \neq \bm{0}} \sup\limits_{\bm{v} \neq \bm{0}}
\dfrac{ \bm{\lambda}_k^T\bm{C}^k\bm{v}}{||\bm{v}||_{\bm{X}_u}||\bm{\lambda}_k||_{\bm{X}_{\lambda_k}}} \geq \beta_F^k~,
\end{equation}
provided that the inhomogeneous Dirichlet boundaries $\{\mytilde{\Gamma}_D^k\}_{k=1}^{N_D}$ are disjoint. Therefore, for each $k \in \{1, \cdots, N_D\}$, the Lagrange multipliers supremizers $\bm{s}_j^{u,\lambda_k}$ --- with $j \in \{1,\cdots,N_\lambda^k\}$ --- are computed by solving the following linear systems:
\begin{equation}
\label{eq: supremizers multipliers}
\bm{X}_u \bm{s}_j^{u,\lambda_k} = \bm{C}_k^T \bm{e}_j \qquad \quad \text{with } j \in \{1,\cdots, N_\lambda^k\}~,
\end{equation}
where $\bm{e}_j \in \mathbb{R}^{N_\lambda^k}$ is the $j$--th canonical basis element. The global set of Lagrange multipliers supremizers is then defined as $\mathcal{S}_h^\lambda := \{\bm{s}_j^{u,\lambda}\}_{j=1}^{N_\lambda} = \bigcup\limits_{k=1}^{N_D} \left(\{\bm{s}_{j'}^{u, \lambda_k}\}_{j'=1}^{N_\lambda^k}\right)$ and we consider the space--reduced velocity subspace $\mathcal{S}_h^{u, \lambda +} = span\left\{\left\{\bm{\phi}_i^u\right\}_{i=1}^{n_u^s}, \{\bm{s}_{j}^{u,\lambda}\}_{j=1}^{N_\lambda}\right\}$. A basis for such a subspace is represented by the columns of the matrix $\bm{\Phi}^{u, \lambda +} = \left[\bm{\phi}_1^u\vert\cdots\vert\bm{\phi}_{n_u^s}^u\vert\bm{s}^{u,\lambda}_1\vert\cdots\vert\bm{s}^{u,\lambda}_{N_\lambda}\right] \in \mathbb{R}^{N_u^s \times (n_u^s+N_\lambda)}$. 

Ultimately, we consider $\mathcal{S}_h^{u, p\lambda +} := span\left\{\{\bm{\phi}_i^u\}_{i=1}^{n_u^s}, \{\bm{s}_j^{u,p}\}_{j=1}^{n_p^s}, \{\bm{s}_{j'}^{u,\lambda}\}_{j'=1}^{N_\lambda}\right\}$ as the space--reduced velocity subspace. An orthonormal basis --- with respect to the norm induced by $\bm{X}_u$ --- for such a subspace is given by
\begin{equation}
\label{eq: supremizers orthonormal basis}
\bm{\Phi}^{u, p\lambda +} = \left[\bm{\phi}_1^u\vert\cdots\vert\bm{\phi}_{n_u^s}^u\vert\bm{\phi}^{u}_{n_u^s+1}\vert\cdots\vert\bm{\phi}^{u}_{n_u^s+n_s^p}\vert\bm{\phi}^{u}_{n_u^s+n_p^s+1}\vert\cdots\vert\bm{\phi}^{u}_{n_u^s+n_p^s+N_\lambda}\right] \quad \in \mathbb{R}^{N_u^s \times \tilde{n}_u^s}~,
\end{equation}
where $\{\bm{\phi}^{u}_{n_u^s+j}\}_{j=1}^{n_p^s}$ and $\{\bm{\phi}^{u}_{n_u^s+n_p^s+j'}\}_{j'=1}^{N_\lambda}$ are computed, respectively, from $\{\bm{s}_j^{u,p}\}_{j=1}^{n_p^s}$ and $\{\bm{s}_{j'}^{u,\lambda}\}_{j'=1}^{N_\lambda}$, applying the Gram-Schmidt algorithm and $\tilde{n}_u^s := n_u^s + n_p^s + N_\lambda$. 
\begin{remark}
For the supremizers enrichment procedure, one could also consider the inf--sup condition in Eq.\eqref{eq: discrete_inf--sup FOM 2a}, instead of the one in Eq.\eqref{eq: discrete_inf--sup FOM 2b}. Under the assumption that the inhomogeneous Dirichlet boundaries are disjoint, this amounts at solving the following linear systems:
\begin{itemize}[leftmargin=20pt]
	\item Pressure supremizers: $\bm{X}_u \bm{s}_j^{u,p} = \bm{B}^T \bm{\phi}_j^p$, with $j \in \{1, \dots, n_p^s\}$;
	\item Lagrange multipliers supremizers: $\begin{bmatrix} \bm{X}_u & \bm{B}^T \\ \bm{B} & \end{bmatrix} \begin{bmatrix} \bm{s}_j^{u, \lambda_k} \\ \bm{q}_j\end{bmatrix} = \begin{bmatrix} \bm{C}_k^T \bm{e}_j \\ \ \end{bmatrix}$,  with $j \in \{1, \dots, N_\lambda^k\}$ and $k \in \{1, \cdots, N_D\}$. Here $\bm{e}_j \in \mathbb{R}^{N_\lambda^k}$ denotes the $j$--th canonical basis vector.
\end{itemize}
However, being $N_\lambda \ll N_p^s$ in common applications, this approach is more computationally expensive. Indeed, $N_D$ saddle point problems featuring the matrix $\bm{B} \in \mathbb{R}^{N_p^s \times N_u^s}$ as constraint matrix have to be solved, instead of a single one, whose constraint is enforced via the ``shorter'' matrix $\bm{C} \in \mathbb{R}^{N_\lambda \times N_u^s}$. 
\end{remark}

When performing space--time model order reduction, the problem at hand preserves a twofold saddle point structure of type 1 (see Eq.\eqref{eq: reduced A}). In this context, the supremizers enrichment procedure of the velocity reduced basis in space alone is not enough to guarantee well--posedness, since dimensionality reduction in time may affect \emph{inf--sup} stability. However, the following \emph{inf--sup} inequalities hold, under the assumption that the matrices $\bm{\Psi}^{u,p}$ and $\bm{\Psi}^{u,\lambda}$ --- defined in Eq.\eqref{eq: temporal bases combined} --- are full rank.
\begin{lmm}
	\label{lmm: lemma 1}
	Let the velocity reduced basis in space be enriched with pressure supremizers $\mathcal{S}_h^p$. If the columns of the matrix $\bm{\Psi}^{u,p} = \left(\bm{\Psi}^u\right)^T \bm{\Psi}^{p}$ are linearly independent, then 
	\begin{equation}
	\label{eq: space--time inf--sup stability pressure}
	\exists \ \beta_{STR}^p > 0 \quad \text{such that} \quad 
	\inf\limits_{\widehat{\bm{q}} \neq \bm{0}} \ \sup\limits_{\substack{\widehat{\bm{v}} \neq \bm{0} \\ \widehat{\bm{A}}_7^{st}\widehat{\bm{v}} = \bm{0}}}
	\dfrac{\widehat{\bm{q}}^T\widehat{\bm{A}}^{st}_4\widehat{\bm{v}}}{\vert\vert \widehat{\bm{q}} \vert\vert_2 \vert\vert \widehat{\bm{v}} \vert\vert_2} \geq \beta_{STR}^p~.
	\end{equation}
\end{lmm}
\begin{proof}
	To satisfy Eq.\eqref{eq: space--time inf--sup stability pressure}, we need that $\exists \ \beta_{STR}^p > 0$ such that:
	\begin{equation*}
	\forall \widehat{\bm{q}} \neq \bm{0} \ \ \exists \widehat{\bm{v}} \neq \bm{0} \quad \text{such that} \quad \dfrac{\widehat{\bm{q}}^T\widehat{\bm{A}}^{st}_4\widehat{\bm{v}}}{\vert\vert \widehat{\bm{q}} \vert\vert_2 \vert\vert \widehat{\bm{v}} \vert\vert_2 } \geq \beta_{STR}^p \quad \text{ and } \widehat{\bm{A}}_7^{st}\widehat{\bm{v}} = \bm{0}~,
	\end{equation*}
	where $\widehat{\bm{A}}^{st}_4 \in \mathbb{R}^{n_p^{st} \times n_u^{st}}$ and $\widehat{\bm{A}}^{st}_7 \in \mathbb{R}^{n_\lambda^{st} \times n_u^{st}}$ are defined as in Eq.~\eqref{eq: ROM spacetime blocks}. Let $\mathit{vec}_u: \mathbb{R}^{n_u^s \times n_u^t} \to \mathbb{R}^{n_u^s n_u^t}$ and $\mathit{vec}_p: \mathbb{R}^{n_p^s \times n_p^t} \to \mathbb{R}^{n_p^s n_p^t}$ be the vectorizing operators for velocity and pressure, respectively. Given $\widehat{\bm{q}} = \mathit{vec}_p(\widehat{\bm{q}}_s \otimes \widehat{\bm{q}}_t) \in \mathbb{R}^{n_p^s n_p^t}$ and $\widehat{\bm{v}} = \mathit{vec}_u(\widehat{\bm{v}}_s \otimes \widehat{\bm{v}}_t) \in \mathbb{R}^{n_u^s n_u^t}$, we have that
	\begin{equation}
	\label{eq: space--time bilinear form A4}
	\widehat{\bm{q}}^T\widehat{\bm{A}}_4^{st}\widehat{\bm{v}} = \sum\limits_{j=1}^{N^t} \left(\widehat{\bm{q}}_s \left(\widehat{\bm{q}}_t\right)_j\right)^T \widehat{\bm{B}} \left(\widehat{\bm{v}}_s \left(\widehat{\bm{v}}_t\right)_j\right) = \sum\limits_{j=1}^{N^t} \left(\widehat{\bm{q}}_t\right)_j \left(\widehat{\bm{v}}_t\right)_j \left(\widehat{\bm{q}}_s^T \widehat{\bm{B}} \widehat{\bm{v}}_s \right) = \left(\widehat{\bm{q}}_s^T \widehat{\bm{B}} \widehat{\bm{v}}_s \right) \left(\widehat{\bm{q}}_t, \widehat{\bm{v}}_t\right)_2~,
	\end{equation}
	where $(\widehat{\bm{q}}_t)_j$, $(\widehat{\bm{v}}_t)_j$ denote the $j$--th entry of $\widehat{\bm{q}}_t$, $\widehat{\bm{v}}_t$, respectively. Let us define $\widehat{\bm{v}} := \mathit{vec}_u(\widehat{\bm{s}}_s \otimes \widehat{\bm{v}}_t)$, where $\widehat{\bm{s}}_s$ is such that $\widehat{\bm{C}} \widehat{\bm{s}}_s = \bm{0}$ and $\widehat{\bm{q}}_s^T \widehat{\bm{B}} \widehat{\bm{s}}_s \geq \beta_{R}^p \vert\vert \widehat{\bm{q}}_s \vert\vert_2 \vert\vert \widehat{\bm{s}}_s \vert\vert_2$, with $\beta_{R}^p > 0$. We have guarantee that $\widehat{\bm{s}}_s$ exists, thanks to the supremizers enrichment procedure in space (see Eq.\eqref{eq: supremizers pressure}). Firstly, notice that $\widehat{\bm{A}}_7^{st}\widehat{\bm{v}} = \bm{0}$; indeed, for $\ell = \mathcal{F}_\lambda(\ell_s, \ell_t)$, $\left(\widehat{\bm{A}}_7^{st}\widehat{\bm{v}}\right)_\ell = \left(\widehat{\bm{C}} \widehat{\bm{s}}_s\right)_{\ell_s} \left(\left(\bm{\Psi}^{u,\lambda}\right)^T \widehat{\bm{v}}_t\right)_{\ell_t} = \bm{0}$, since $\widehat{\bm{C}} \widehat{\bm{s}}_s = \bm{0}$. So, the additional constraint appearing in the supremum of Eq.\eqref{eq: space--time inf--sup stability pressure} is trivially satisfied. Then, considering Eq.\eqref{eq: space--time bilinear form A4}, we have that
	\begin{equation*}
	\widehat{\bm{q}}^T \widehat{\bm{A}}_4^{st} \widehat{\bm{v}} = \left(\widehat{\bm{q}}_s^T \widehat{\bm{B}} \widehat{\bm{s}}_s \right) \left(\widehat{\bm{q}}_t, \widehat{\bm{v}}_t\right)_2 \geq \beta_{RB}^p \vert\vert\widehat{\bm{q}}_s\vert\vert_2 \vert\vert\widehat{\bm{s}}_s\vert\vert_2\left(\widehat{\bm{q}}_t, \widehat{\bm{v}}_t\right)_2~.
	\end{equation*}
	Therefore, given that $||\widehat{\bm{q}}||_2 = ||\widehat{\bm{q}}_s||_2 \ ||\widehat{\bm{q}}_t||_2$ and $||\widehat{\bm{v}}||_2 = ||\widehat{\bm{s}}_s||_2 \ ||\widehat{\bm{v}}_t||_2$ (see Eq.\eqref{eq: space--time inner product factorization}), we have that:
	\begin{equation*}
	\dfrac{\widehat{\bm{q}}^T\widehat{\bm{A}}^{st}_4 \widehat{\bm{v}}}{\vert\vert \widehat{\bm{q}} \vert\vert_2 \vert\vert \widehat{\bm{v}} \vert\vert_2} \geq \beta_{RB}^p \dfrac{\left(\widehat{\bm{q}}_t, \widehat{\bm{v}}_t\right)_2}{\vert\vert\widehat{\bm{q}}_t\vert\vert_2 \vert\vert\widehat{\bm{v}}_t\vert\vert_2}~.
	\end{equation*}
	Hence, to conclude, Eq.\eqref{eq: space--time inf--sup stability pressure} holds if $\exists \beta_t^p > 0$ such that
	\begin{equation}
	\label{eq: temporal inf--sup inequality}
	\forall \widehat{\bm{q}}_t \neq \bm{0} \ \ \exists \widehat{\bm{v}}_t \neq \bm{0} \quad \text{such that} \quad \dfrac{\left(\widehat{\bm{q}}_t, \widehat{\bm{v}}_t\right)_2}{\vert\vert\widehat{\bm{q}}_t\vert\vert_2 \vert\vert\widehat{\bm{v}}_t\vert\vert_2} \geq \beta_t^p~.
	\end{equation}
	This represents an \emph{inf--sup} condition on the temporal reduced subspaces with respect to the Euclidean norm. Based on the definition of the matrix $\bm{\Psi}^{u,p}$ in Eq.\eqref{eq: temporal bases combined}, Eq.\eqref{eq: temporal inf--sup inequality} is equivalent to the linear independence of the columns of $\bm{\Psi}^{u,p}$.
\end{proof}
\begin{lmm}
	\label{lmm: lemma 2}
	Let $k \in \{1, \cdots, N_D\}$. Let the velocity reduced basis in space be enriched with Lagrange multipliers supremizers $\mathcal{S}_h^{\lambda_k}$. If the columns of the matrix $\bm{\Psi}^{u,\lambda_k} = \left(\bm{\Psi}^u\right)^T \bm{\Psi}^{\lambda_k}$ are linearly independent, then 
	\begin{equation}
	\label{eq: space--time inf--sup stability lagrange multipliers}
	\exists \ \beta_{STR}^{\lambda_k} > 0 \quad \text{such that} \quad
	\inf\limits_{\widehat{\bm{\lambda}}_k \neq \bm{0}} \ \sup\limits_{\widehat{\bm{v}} \neq \bm{0}}
	\dfrac{\widehat{\bm{\lambda}}_k^T\left(\widehat{\bm{A}}^{st}_7\right)^k \widehat{\bm{v}}}{\vert\vert \widehat{\bm{\lambda}}_k \vert\vert_2 \vert\vert \widehat{\bm{v}} \vert\vert_2 }\geq \beta_{STR}^{\lambda_k}~,
	\end{equation}
	where $\left(\widehat{\bm{A}}^{st}_7\right)^k \in \mathbb{R}^{n^{\lambda_k}_{st} \times n^u_{st}}$ is the $k$--th block of $\widehat{\bm{A}}^{st}_7$ along its first dimension.
\end{lmm}
\begin{proof}
	The proof proceeds as the one of Lemma \ref{lmm: lemma 1}.
\end{proof}

Based on Lemmas \ref{lmm: lemma 1} -- \ref{lmm: lemma 2}, the following theorem holds.
\begin{theorem}
	Let the velocity reduced basis in space be enriched with pressure supremizers $\mathcal{S}_h^p$ and Lagrange multipliers supremizers $\mathcal{S}_h^\lambda$, computed by solving Eqs.\eqref{eq: supremizers pressure}--\eqref{eq: supremizers multipliers}, respectively. Assume that the space--reduced problem is well--posed upon the supremizers enrichment procedure. If the columns of the matrices $\bm{\Psi}^{u,p}$, $\bm{\Psi}^{u, \lambda_k}$ $\forall k \in \{1, \cdots, N_D\}$ are linearly independent, then the space--time--reduced problem arising from a Galerkin projection (i.e. Problem \ref{pb:compact_space_time_reduced_system} with $\mytilde{\bm{\Pi}} = \bm{\Pi}$) is well--posed.
\end{theorem}
\begin{proof}
	The proof trivially follows from Lemmas \ref{lmm: lemma 1} -- \ref{lmm: lemma 2}, leveraging the assumption on the well--posedness of the space--reduced problem upon the supremizers enrichment procedure in space.
\end{proof}

Since the temporal reduced bases for velocity, pressure and Lagrange multipliers have been derived independently, we do not have any \emph{a priori} guarantee that the matrices $\bm{\Psi}^{u,p}$ and $\{\bm{\Psi}^{u, \lambda_k}\}_{k=1}^{N_D}$ are indeed full column rank. In order for the \emph{inf--sup} inequalities in Eqs.\eqref{eq: space--time inf--sup stability pressure}-\eqref{eq: space--time inf--sup stability lagrange multipliers} to be satisfied, it is necessary to enrich also the temporal reduced basis of the velocity.

\begin{algorithm}[t]
	\caption{Velocity temporal reduced basis enrichment}
	\label{alg: pressure time supremizers}
	\begin{algorithmic}[1]
		\Function {TemporalStabilizers}{$\bm{\Psi}^u$, $\bm{\Psi}^d$, $\varepsilon^t$}  \Comment{$\bm{\Psi}^d$ is the dual temporal basis matrix} \label{line 1}
		\State{Compute $\bm{\Psi}^{u,d} = \left(\bm{\Psi}^u\right)^T\bm{\Psi}^d = \left[\bm{\xi}_1\vert\cdots\vert\bm{\xi}_{n_d^t}\right]$} \Comment{$n_d^t$ is the number of dual temporal bases}
		\For{$\ell \in \{1,\cdots,n_d^t\}$}
		\If {$\ell =1$}
		\State $\bm{\pi}_{\bm{\xi}} \gets \bm{0}$
		\Else
		\State $\bm{\pi}_{\bm{\xi}} \gets \sum\limits_{j=1}^{\ell-1} \dfrac{(\bm{\xi}_\ell, \bm{\xi}_j)_2}{(\bm{\xi}_j, \bm{\xi}_j)_2} \bm{\xi}_j$  \Comment{Compute the projection of $\bm{\xi}_\ell$ onto $span\{\bm{\xi}_j\}_{j=1}^{\ell-1}$} \label{line 7}
		\EndIf
		\If{$\vert\vert \bm{\xi}_\ell - \bm{\pi}_{\bm{\xi}} \vert\vert_2 \leq \varepsilon^t$} \Comment{Check enrichment condition} \label{line 8}
		\State $\bm{\psi}^\star \gets \bm{\Psi}^d_{:,\ell}$
		\State $\bm{\psi}^+ = \Big(\bm{\psi}^\star - \sum\limits_{j=1}^{n_u^t} (\bm{\psi}^\star, \bm{\psi}_j^u)_2 \bm{\psi}_j^u\Big) \Big/ \Big\vert\Big\vert\bm{\psi}^\star - \sum\limits_{j=1}^{n_u^t} (\bm{\psi}^\star, \bm{\psi}_j^u)_2 \bm{\psi}_j^u\Big\vert\Big\vert_2$  \label{line 10}
		\State $\bm{\Psi}^u \gets [\bm{\Psi}^u \vert \bm{\psi}^+]$  \Comment{Enrich velocity temporal reduced basis} \label{line 11}
		\State $n_u^t \gets n_u^t + 1$
		\State $\bm{\Psi}^{u,d} = \left[\bm{\xi}_1\vert\cdots\vert\bm{\xi}_{n_p^t}\right] \gets \left[\left(\bm{\Psi}^{u,d}\right)^T \Big\vert \ \left(\left(\bm{\psi}^+\right)^T \bm{\Psi}^d\right)^T\right]^T$  \Comment{Update $\bm{\Psi}^{u,d}$} \label{line 12}
		\State $\ell \gets 1$
		\Else
		\State $\bm{\xi}_\ell \leftarrow \bm{\xi}_\ell - \bm{\pi}_{\bm{\xi}}$
		\EndIf
		\EndFor
		\EndFunction
	\end{algorithmic}
\end{algorithm}

\begin{corollary}
	\label{cor: corollary 1}
	Let the velocity reduced basis in space be enriched with pressure supremizers $\mathcal{S}_h^p$. Assume that the columns of $\bm{\Psi}^{u,p}$ are linearly dependent. If the velocity temporal reduced basis $\bm{\Psi}^u$ is enriched according to Algorithm \ref{alg: pressure time supremizers}, setting $\bm{\Psi}^d = \bm{\Psi}^p$ and fixing $\varepsilon^t \geq 0$, then the inf--sup inequality in Eq.\eqref{eq: space--time inf--sup stability pressure} is satisfied.
\end{corollary}
\begin{proof}
	Let us denote the columns of $\bm{\Psi}^{u,p}$ as $\{\bm{\xi}_\ell\}_{\ell=1}^{n_p^t}$. Since the elements of $\{\bm{\xi}_\ell\}_{\ell=1}^{n_p^t}$ are linearly dependent, $\exists \bar{\bm{\alpha}} = [\bar{\alpha}_1, \cdots, \bar{\alpha}_{n_p^t}] \neq \bm{0}$ such that 
	\begin{equation}
	\label{eq: instability condition}
	\bm{\vartheta} := \sum\limits_{\ell=1}^{n_p^t} \bar{\alpha}_\ell \bm{\xi}_\ell = \sum\limits_{\ell=1}^{n_p^t} \bar{\alpha}_\ell \left(\bm{\Psi}^u\right)^T \bm{\psi}^p_\ell = \bm{0} \ \ \in \mathbb{R}^{n_u^t}~.
	\end{equation}
	The goal is to enrich the temporal velocity reduced basis in such a way that Eq.\eqref{eq: instability condition} cannot hold and it can be accomplished by applying Algorithm \ref{alg: pressure time supremizers}, with pressure as dual field. So, in Algorithm \ref{alg: pressure time supremizers}, let us set $\bm{\Psi}^d=\bm{\Psi}^p$ and let $\varepsilon^t\geq0$ (Line \ref{line 1}). We proceed iteratively, showing that, upon a suitable enrichment procedure, $\forall \ell^\star \in \{1, \dots, n_p^t\}$, $\nexists \bar{\bm{\alpha}} = [\bar{\alpha_1}, \dots, \bar{\alpha}_{\ell^\star}] \in \mathbb{R}^{\ell^\star}, \bar{\bm{\alpha}} \neq \bm{0}$ such that
	\begin{equation}
	\label{eq: stability condition iterative}
	\bm{\vartheta} := \sum\limits_{\ell=1}^{\ell^\star} \bar{\alpha}_\ell \bm{\xi}_\ell = \sum\limits_{\ell=1}^{\ell^\star} \bar{\alpha}_\ell \left(\bm{\Psi}^u\right)^T \bm{\psi}^p_\ell = \bm{0} \ \ \in \mathbb{R}^{n_u^t}~.
	\end{equation}
	Firstly, let us consider the case $\ell^\star = 1$. Here, we suppose that $\bm{\xi}_1 = \bm{0}$, so that $\vert\vert \bm{\xi}_1 \vert\vert_2 \leq \varepsilon^t \ \forall \varepsilon^t \geq 0$. This means that the first pressure temporal basis function $\bm{\psi}_1^p$ belongs to the orthogonal complement of the velocity temporal reduced subspace. In such a case, the matrix $\bm{\Psi}^{u,p}$ cannot be full column rank and any $\bm{\bar{\alpha}} = [\bar{\alpha}_1] \in \mathbb{R}^{1}$ trivially satisfies Eq.\eqref{eq: stability condition iterative}. Let us now enrich the velocity temporal reduced basis with $\bm{\psi}_1^p$ (Lines \ref{line 10}--\ref{line 11}). Upon the enrichment, the last entry of $\bm{\xi}_1$ equals $\left(\bm{\psi}^p_1, \bm{\psi}^p_1\right)_2 =  \vert\vert\bm{\psi}^p_1\vert\vert_2^2 = 1$, so $\bm{\xi}_1 \neq \bm{0}$. Hence, for $\ell^\star=1$ Eq.\eqref{eq: stability condition iterative} is satisfied and by Lemma \ref{lmm: lemma 1} \emph{inf--sup} stability with respect to the space spanned by $\bm{\psi}^p_1$ is attained.
	
	Let us now consider $\ell^\star \in \{2, \dots, n_p^t\}$ and let us suppose that Eq.\eqref{eq: stability condition iterative} holds for $\ell^\star-1$. Let $\bm{\Psi}^u \in \mathbb{R}^{N^t \times n_u^t}$ be the matrix encoding the temporal velocity reduced basis at the $\ell^\star$--th step of the algorithm; notice that the value of $n_u^t$ may have changed during the application of the algorithm. Now, if the first $\ell^\star$ columns of $\bm{\Psi}^{u,p}$ are linearly independent, then Eq.\eqref{eq: stability condition iterative} holds by definition and via Lemma \ref{lmm: lemma 1} \emph{inf--sup} stability with respect to the space spanned by $\{\bm{\psi}_j^p\}_{j=1}^{\ell^\star}$ is guaranteed. Otherwise, $\bm{\xi}_{\ell^\star}$ can be expressed as a linear combination of the previous linearly--dependent columns $\{\bm{\xi}_\ell\}_{\ell=1}^{\ell^\star-1}$. Such a condition can be verified by comparing $\bm{\xi}_{\ell^\star}$ with its orthogonal projection $\bm{\pi}_{\bm{\xi}_{\ell^\star}}$ onto the $(\ell^\star-1)$--dimensional subspace spanned by $\{\bm{\xi}_\ell\}_{\ell=1}^{\ell^\star-1}$. Indeed, if Eq.\eqref{eq: stability condition iterative} does not hold, then $\bm{\xi}_{\ell^\star}$ is such that $||\bm{\xi}_{\ell^\star} - \bm{\pi}_{\bm{\xi}_{\ell^\star}}||_2 = 0$ and so $||\bm{\xi}_{\ell^\star} - \bm{\pi}_{\bm{\xi}_{\ell^\star}}||_2 \leq \varepsilon^t \ \forall \varepsilon^t\geq 0$ (Lines \ref{line 7}--\ref{line 8}). In this case, in Algorithm \ref{alg: pressure time supremizers} we enrich the velocity temporal reduced basis with the pressure temporal basis function $\bm{\psi}^p_{\ell^\star}$ associated to $\bm{\xi}_{\ell^\star}$, i.e. such that $\bm{\xi}_{\ell^\star} = \left(\bm{\Psi}^u\right)^T \bm{\psi}^p_{\ell^\star}$ (Lines \ref{line 10}--\ref{line 11}). 
	
	We have to verify that, upon the enrichment procedure, Eq.\eqref{eq: stability condition iterative} holds for the current value of $\ell^\star$.
	Let us consider the last entry of $\bm{\vartheta}$, that corresponds to the novel velocity temporal basis function $\bm{\psi}_{n_u^t+1}^u = \bm{\psi}_{\ell^\star}^p$. Exploiting the orthonormality of the pressure temporal basis functions, we have that
	\begin{equation*}
	\left(\bm{\vartheta}\right)_{n_u^t+1} = \sum_{\ell=1}^{\ell^\star} \bar{\alpha}_\ell \left(\bm{\xi}_\ell\right)_{n_u^t+1} = \sum_{\ell=1}^{\ell^\star} \bar{\alpha}_\ell \left(\bm{\psi}_{\ell^\star}^p, \bm{\psi}_\ell^p\right)_2 = \sum_{\ell=1}^{\ell^\star} \bar{\alpha}_\ell \delta_{\ell \ell^\star} = \bar{\alpha}_{\ell^\star} = 0 \ \ \implies \ \ \bar{\alpha}_{\ell^\star} = 0~.
	\end{equation*}
	Therefore, $\bar{\alpha}_{\ell^\star} = 0$ and Eq.\eqref{eq: stability condition iterative} reduces to $\bm{\vartheta} = \sum_{\ell=1}^{\ell^\star-1} \bar{\alpha}_\ell \bm{\xi}_\ell = \bm{0}$.  However, since the first $n_u^t$ components of the vectors $\{\bm{\xi}_\ell\}_{\ell=1}^{\ell^\star-1}$ are linearly independent by hypothesis, we also have that $\bar{\alpha}_\ell=0$ $\forall \ell \in \{1,\cdots, \ell^\star-1\}$. Hence, Eq.\eqref{eq: stability condition iterative} is satisfied for the current value of $\ell^\star$ and, upon the enrichment, the first $\ell^\star$ columns of $\mytilde{\bm{\Psi}}^{u,p}$ are linearly independent. Proceeding by induction up to $\ell^\star = n_p^t$, we can then prove that, upon the enrichment procedure described in Algorithm \ref{alg: pressure time supremizers}, Eq.\eqref{eq: instability condition} cannot hold. Therefore, the \emph{inf--sup} inequality in Eq.\eqref{eq: space--time inf--sup stability pressure} is satisfied. 
	
	One remark is necessary in order to conclude the proof. In Line \ref{line 12} of Algorithm \ref{alg: pressure time supremizers}, the velocity temporal reduced basis is not enriched with the ``critical'' pressure basis function $\bm{\psi}^p_{\ell^\star}$ (as we assumed before), but with the normalized orthogonal complement $\bm{\psi}^+$ of the latter with respect to the space spanned by $\{\bm{\psi}^u_j\}_{j=1}^{n_u^t}$ (Line \ref{line 10}). In this way, the velocity temporal reduced basis remains orthonormal upon the enrichment procedure. However, this does not impact the linear independence of the columns of $\bm{\Psi}^{u,p}$, since the subspaces spanned by $\{\{\bm{\psi}^u_j\}_{j=1}^{n_u^t}, \bm{\psi}^p_{\ell^\star}\}$ and by $\{\{\bm{\psi}^u_j\}_{j=1}^{n_u^t}, \bm{\psi}^+\}$ trivially coincide. 
\end{proof}

Notice that the velocity reduced basis enrichment in space and in time is performed according to different paradigms. Indeed, in space we augment the basis with solutions to optimization problems stemming from the \emph{inf--sup} inequalities, while in time we iteratively select elements to be added to the basis in order to guarantee numerical stability, but without any optimization procedure being involved. For this reason, we will refer to the enrichment in space as the ``supremizers enrichment'' and to the enrichment in time as the ``stabilizers enrichment''.

\begin{remark}
The choice $\varepsilon^t=0$ is enough to retain \emph{inf--sup} stability, since it guarantees that $\beta_t^p$ in Eq.\eqref{eq: temporal inf--sup inequality} is strictly positive. However, if the columns of $\bm{\Psi}^{u,p}$ are ``almost'' collinear, then $\beta_t^p \gtrsim 0$ and the accuracy of the method is compromised. Hence, from a numerical standpoint, selecting $\varepsilon^t>0$ is crucial. We numerically investigated the effect of $\varepsilon^t$ in Subsection \ref{subs:results bif}. 
\end{remark}
\begin{remark}
In principle, the choice of computing fixed temporal POD modes rather than tailored ones, specific to each spatial POD mode, can be disadvantageous. Higher--index spatial modes are indeed associated to higher--frequency temporal modes and constructing temporal bases tailored to the elements of the spatial one may improve the accuracy of the approximation and also reduce the computational cost of the simulation~\cite{choi2019space, tenderini2022pde}. However, choosing a fixed temporal basis simplifies the enrichment procedure, since Algorithm \ref{alg: pressure time supremizers} can be performed only once (per dual field), rather than $n_u^s n_d^s$ times, being $n_d^s$ the number of spatial modes of the dual field of interest.
\end{remark}
\bigskip
\begin{corollary}
	\label{cor: corollary 2}
	Let $k \in \{1,\cdots,N_D\}$. Let the velocity reduced basis in space be enriched with the $k$--th Lagrange multiplier supremizers $\mathcal{S}_h^{\lambda_k}$. Assume that the columns of $\bm{\Psi}^{u,\lambda_k}$ are linearly dependent. If the velocity temporal reduced basis $\bm{\Psi}^u$ is enriched according to Algorithm \ref{alg: pressure time supremizers}, setting $\bm{\Psi}^d=\bm{\Psi}^{\lambda_k}$ and fixing $\varepsilon^t \geq 0$, then the inf--sup inequality in Eq.\eqref{eq: space--time inf--sup stability lagrange multipliers} is satisfied. 
\end{corollary}
\begin{proof}
	The proof proceeds as the one of Corollary \ref{cor: corollary 1}, taking advantage of Lemma \ref{lmm: lemma 2}.
\end{proof} 

One remark is worth to follow. While supremizers enrichment in space has to be performed with respect to all the dual fields in order to guarantee \emph{inf--sup} stability, in time we can often consider only one of them. Indeed, once the velocity temporal reduced basis has been enriched with respect, say, to pressure (so that $\bm{\Psi}^{u,p}$ is full column rank), it is often the case that also the columns of the matrices $\bm{\Psi}^{u, \lambda_k}$ (with $k \in \{1, \cdots, N_D\}$) are linearly independent. If so, no further stabilizers enrichment is necessary. As a consequence, the ``stabilized'' velocity temporal reduced basis depends on the order in which the dual fields are considered. We numerically investigated this aspect in Subsection \ref{subs:results bif}.

\subsubsection{Least--squares Petrov--Galerkin projection}
\label{subs: ST--PGRB stability}
Least--squares (LS) Petrov--Galerkin (PG) reduced basis (RB) methods have already been proposed in the framework of space--time model order reduction in \cite{choi2019space, kim2021efficient}. Similarly to the space--reduced case, the main idea is to compute the reduced solution $\widehat{\bm{w}} \in \mathbb{R}^{n^{st}}$ by least--squares minimization of the FOM residual $\bm{r}^{st}(\widehat{\bm{w}})$ (see Eq.\eqref{eq: space time reduced problem}) in a suitable norm.  

In this work, we decided to extend the algebraic LS--PG--RB method that has been proposed for steady parametrized Stokes equations in \cite{abdulle2015petrov, dalsanto2019algebraic} to the time--dependent case, in the framework of space--time model order reduction. In the steady case, the key idea of the method is to define a global parameter--dependent supremizing operator $T_{\bm{\mu}}: \mathcal{S}_h \to  \mathcal{S}_h$ such that
\begin{equation}
\label{eq: global supremizing operator space}
\left(T_{\bm{\mu}}(\vec{z}_h), \vec{w}_h\right)_{\mathcal{S}_h} = \mathcal{A}_{\bm{\mu}}(\vec{z}_h, \vec{w}_h)~~.
\end{equation}
Here $\mathcal{S}_h$ is the finite--dimensional subspace where the FOM solutions are sought, equipped with the inner product $(\cdot, \cdot)_{\mathcal{S}_h}$, and $\mathcal{A}_{\bm{\mu}}$ is the global (potentially parameter--dependent) steady Stokes operator. The LS--PG--RB method stems from a Petrov--Galerkin projection, where the trial subspace $\mathcal{S}_n := span\left\{\vec{\xi}_i, \ i \in \{1, \cdots, n\}\right\} \subset \mathcal{S}_h$ is computed in a standard fashion (e.g. via truncated POD of the mode--1 unfolding of the snapshots' tensor), whereas the test one is defined as $\mytilde{\mathcal{S}}_n^{\bm{\mu}} := span\left\{T_{\bm{\mu}}(\vec{\xi}_i), \ i \in \{1,\cdots,n\} \right\}$. From an algebraic standpoint, the basis of the test space can be encoded in the matrix $\bm{X}_h^{-1} \bm{A}_h \bm{\Phi} \in \mathbb{R}^{N^s \times N^s}$, where $\bm{X}_h \in \mathbb{R}^{N^s \times N^s}$ is the FOM norm matrix, $\bm{A}_h \in \mathbb{R}^{N^s \times N^s}$ is the FOM discretization of the global steady Stokes operator and $\bm{\Phi} \in \mathbb{R}^{N^s \times n^s}$ is the matrix encoding the reduced basis. Ultimately, the solution is retrieved by solving the following linear system
\begin{equation}
\begin{alignedat}{2}
\label{eq: linear system LS--PG--RB}
\bm{A}_n \bm{w}_n = \bm{f}_n \quad \text{with} \quad \bm{A}_n &= \left(\bm{A}_h\bm{\Phi}\right)^T \bm{X}_h^{-1}\bm{A}_h \bm{\Phi} \quad \in\mathbb{R}^{n^s\times n^s}~, \\
\bm{f}_n &= \left(\bm{A}_h\bm{\Phi}\right)^T\bm{X}_h^{-1}\bm{f}_h \qquad \in \mathbb{R}^{n^s}~,
\end{alignedat}
\end{equation}
where $\bm{f}_h = \bm{f}_h({\bm{\mu}}) \in \mathbb{R}^{N^s}$ is the (potentially parameter--dependent) FOM right--hand side. A key property of the LS--PG--RB method is that the choice of the test space $\mytilde{\mathcal{S}}_n^{\bm{\mu}}$ automatically guarantees \emph{inf--sup} stability. Therefore, no supremizers enrichment of the velocity reduced basis has to be performed to retain well--posedness. Additionally, it can be shown that the solution to Eq.\eqref{eq: linear system LS--PG--RB} minimizes the FOM residual in the norm induced by $\bm{X}_h^{-1}$. We refer to \cite{abdulle2015petrov, quarteroni2015reduced} for further details.

Focusing on the problem at hand and in the context of space--time model order reduction, the application of the LS--PG--RB method amounts at solving the following minimization problem:
\begin{equation}
\label{eq: functional minimization ST--PGRB}
\text{Find } \widehat{\bm{w}}^{pg} \in \mathbb{R}^{n^{st}} \ \text{such that:} \qquad \widehat{\bm{w}}^{pg} = \underset{\widehat{\bm{v}} \in \mathbb{R}^{n^{st}}}{\text{argmin}}\dfrac12 \vert\vert \bm{r}^{st}(\widehat{\bm{v}}) \vert\vert^2_{\left(\bm{X}^{st}\right)^{-1}}~.
\end{equation}
We refer to Eq.\eqref{eq: functional minimization ST--PGRB} as the ST--PGRB problem. Exploiting the convexity of the functional to be minimized, $\widehat{\bm{w}}^{pg}$ can be computed as the solution to the following linear system:
\begin{equation}
\begin{alignedat}{2}
\label{eq: linear system ST--PGRB}
\widehat{\bm{A}}^{pg} \widehat{\bm{w}}^{pg} = \widehat{\bm{F}}^{pg} \quad \text{with} \quad \widehat{\bm{A}}^{pg} &= \left(\bm{A}^{st}\bm{\Pi}\right)^T \left(\bm{X}^{st}\right)^{-1}\bm{A}^{st} \bm{\Pi} \quad \ \in \mathbb{R}^{n^{st} \times n^{st}}~,\\
\widehat{\bm{F}}^{pg} &= \left(\bm{A}^{st}\bm{\Pi}\right)^T\left(\bm{X}^{st}\right)^{-1}\bm{F}^{st} \quad \quad \ \in \mathbb{R}^{n^{st}}~.
\end{alignedat}
\end{equation}
The block structure of the problem at hand can be exploited to ease the assembling of $\widehat{\bm{A}}^{pg}$ and $\widehat{\bm{F}}^{pg}$. As in the ST--GRB approach, parameter--independent blocks can be pre--computed during the offline phase of the method. Nonetheless, we remark that, compared to ST--GRB, more blocks of the left--hand side matrix feature parametric dependency, hence inducing an increase of the online computational cost. However, the latter can be reduced and made independent from the number of space--time full--order DOFs by leveraging the affine parametrization of the reaction term. We refer to Appendix \ref{appendix: ST--PGRB assembling} for details. By analogy with the steady case (see Eq.\eqref{eq: global supremizing operator space}), we can define a space--time parameter--dependent global supremizing operator $\mathcal{T}^{st}_{\bm{\mu}}: \mathcal{ST}_{h,\delta} \to \mathcal{ST}_{h,\delta}$ such that
\begin{equation}
\label{eq: global supremizing operator space--time}
\left(\mathcal{T}^{st}_{\bm{\mu}}(\vec{z}_{h,\delta}),\vec{w}_{h,\delta}\right)_{\mathcal{ST}_{h,\delta}} = \mathcal{A}^{st}_{\bm{\mu}}(\vec{z}_{h,\delta}, \vec{w}_{h,\delta})~,
\end{equation}
where $\mathcal{A}^{st}_{\bm{\mu}}$ is a parameter--dependent bilinear form corresponding to the space--time global Stokes operator, whose full--order algebraic counterpart is given by the matrix $\bm{A}^{st}$ (see Eq.\eqref{eq: monolitic_FOM_system}). From an algebraic standpoint, the basis of the test space --- constructed as orthonormal with respect to the norm induced by $\bm{X}^{st}$ (see Eq.\eqref{eq: ST norms matrices}) --- can be encoded in the matrix $\mytilde{\bm{\Pi}} = \left(\bm{X}^{st} \right)^{-1} \bm{A}^{st} \bm{\Pi} \in \mathbb{R}^{N^{st} \times n^{st}}$.
\begin{theorem}
	\label{thm: theorem 3}
	Assume that the conditions in Eq.\eqref{eq: discrete_inf--sup FOM 2a} hold. Define $\mytilde{\bm{\Pi}} := \left(\bm{X}^{st} \right)^{-1} \bm{A}^{st} \bm{\Pi}$. Then, the ST--PGRB problem in Eq.\eqref{eq: functional minimization ST--PGRB} is inf--sup stable, i.e.
	\begin{equation}
	\label{eq: inf--sup stability petrov-galerkin}
	\exists \ \beta_{STPG} > 0 \quad \text{ such that} \quad \inf\limits_{\widehat{\bm{w}}\neq\bm{0}} \sup\limits_{\widehat{\bm{y}}\neq\bm{0}} \dfrac{\widehat{\bm{w}}^T \widehat{\bm{A}}^{pg} \widehat{\bm{y}}}{\vert\vert\widehat{\bm{w}}\vert\vert_2 {\vert\vert\widehat{\bm{y}}}\vert\vert_2} \geq \beta_{STPG}~.~
	\end{equation}
\end{theorem}
We refer to \cite{abdulle2015petrov, quarteroni2015reduced} for the proof of a corresponding result in the steady case. Upon defining the global supremizing operator as in Eq.\eqref{eq: global supremizing operator space--time}, the same proof applies to the time--dependent case, leveraging space--time model order reduction.
\begin{remark}
In \cite{dalsanto2019algebraic}, the authors present a purely algebraic LS--PG--RB method, based on the substitution of the norm matrix with a surrogate $\bm{P}_{\bm{X}}$. This provides significant computational gains if parametrized geometries are considered, since in such cases the norm matrix is parameter--dependent and the online computation of its inverse may represent a computational bottleneck. A smart choice consists then in choosing the surrogate $\bm{P}_{\bm{X}}$ as an easy--to--invert and parameter--independent matrix. The well--posedness of the resulting problem is proven not to be impacted by such a choice. Even if we did not focus on problems featuring parametrized geometries, we nevertheless decided to approximate the spatio--temporal norm matrix $\bm{X}^{st}$ with an easy--to--invert surrogate $\bm{P}_{\bm{X}}$. In particular, we chose $\bm{P_X}$ as the diagonal part of $\bm{X}^{st}$, i.e. $(\bm{P_X})_{ij} = (\bm{X}^{st})_{ij} \ \delta_{ij}$.  
\end{remark}

\section{Numerical results} \label{sec:section4}
We evaluated the performances of the ST--GRB and ST--PGRB methods, taking into account both accuracy and computational efficiency. The ``standard'' RB method (denoted as SRB--TFO), featuring dimensionality reduction only in space by means of a Galerkin projection, served as a baseline. 

\subsection{Setup} \label{subs: subsection 4.1}
We solved the unsteady incompressible Stokes equations endowed with an additional parameter--dependent reaction term (see Eq.\eqref{eq: strong_form_stokes}) in two different geometries (Figure \ref{fig: domain geometries}): (1) an idealized symmetric bifurcation with characteristic angle $\alpha=50^\circ$; (2) a patient--specific geometry of a femoropopliteal bypass. The geometry of the bifurcation is identical to the one employed in \cite{pegolotti2021model} as a building block for the modular geometrical approximation of blood vessels. The geometry of the femoropopliteal bypass --- bridging the circulation between the femoral artery and the popliteal one in case of severe stenotic formations in the former --- has been reconstructed from CT scans as detailed in \cite{marchandise2012quality} and it has been employed e.g. in \cite{colciago2014comparisons, colciago2018reduced}. For all the simulations, we set the blood density and viscosity to $\rho = 1.06 \ g \cdot cm^{-3}$ and $\mu = 3.5 \cdot 10^{-3} \ g \cdot cm^{-1} \cdot s^{-1}$.

\begin{figure}[t]
	\centering
	\includegraphics[width=0.9\textwidth]{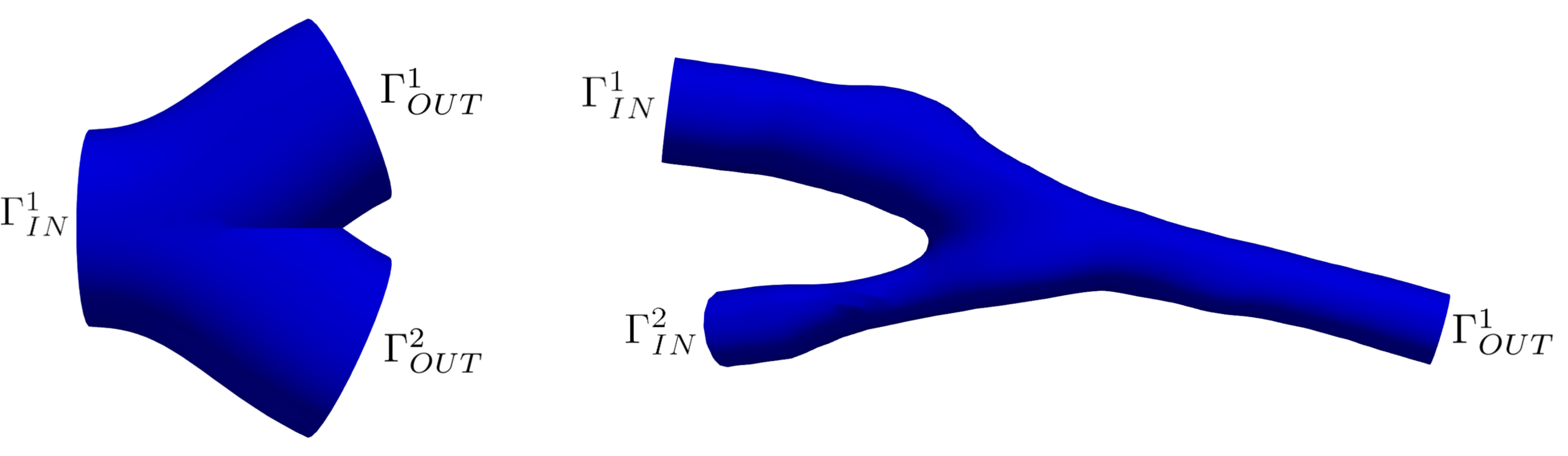}
	\caption{Geometries of the Symmetric Bifurcation (left) and of the Femoropopliteal Bypass (right).}
	\label{fig: domain geometries}
\end{figure}

Parametric dependency concerns the blood clots density functions $\{\rho_c^q(\bm{x}; \bm{\mu})\}_{q=1}^{N_c}$ and the temporal parts of the Dirichlet data $\{g_k^t(t; \bm{\mu})\}_{k=1}^{N_D}$. Blood clots domains $\{\Omega_c^q\}_{q=1}^{N_c}$ are defined as follows:
\begin{equation}
\label{eq: blood clot domain}
\Omega_c^q = \left\{ \bm{x} \in \Omega : \ \vert\vert \bm{x} - \bm{x}_c^q \vert\vert_{*^q} \leq r_c^q \right\}~,
\end{equation}
where $\bm{x}_c^q$ and $r_c^q$ are, respectively, the center and the radius of the $q$--th clot, while $\vert\vert\cdot\vert\vert_{*^q}: \mathbb{R}^3 \to \mathbb{R}^+$ is such that $\vert\vert\bm{y}\vert\vert_{*^q} = \left( \bm{y}^T \bm{X}_q \bm{y} \right)^{1/2} $ with
\begin{equation}
\label{eq: clot norm}
\bm{X}_q := 
\begin{bmatrix}
\vert & \vert & \vert \\
\bm{n}^q & \bm{t}_1^q & \bm{t}_2^q \\
\vert & \vert & \vert
\end{bmatrix} 
\begin{bmatrix}
\sigma_{\bm{n}^q} & & \\
& \sigma_{\bm{t}_1^q} & \\
& & \sigma_{\bm{t}_2^q}
\end{bmatrix} 
\begin{bmatrix}
\text{---} & (\bm{n}^q)^T & \text{---} \\
\text{---} & (\bm{t}_1^q)^T & \text{---} \\
\text{---} & (\bm{t}_2^q)^T & \text{---}
\end{bmatrix}~.
\end{equation}
Here $(\bm{n}^q, \bm{t}_1^q, \bm{t}_2^q)$ is an orthonormal reference system, where $\bm{n}^q$ is the outward unit normal vector to $\partial\Omega$ at $\bm{x}_c^q$ and $\bm{t}_1^q$ is parallel to the main flow direction. The values of $\sigma_{\bm{n}^q}$, $\sigma_{\bm{t}_1^q}$,  $\sigma_{\bm{t}_2^q}$ influence the shape of the clot and, in particular, determine the ``elongation'' of the clot domain $\Omega_c^q$ along the directions $\bm{n}^q$, $\bm{t}_1^q$, $\bm{t}_2^q$, respectively. The blood clot density function is defined as follows:
\begin{equation}
\label{eq: blood clot density}
\rho_c^q(\bm{x}; \bm{\mu}_q^c) = 
\begin{cases}
\bm{\mu}_q^c & 
\text{ if } \vert\vert \bm{x} - \bm{x}_c^q \vert\vert_{*^q} \leq (1 - \varepsilon_c) r_c^q~, \\
\bm{\mu}_q^c \ \cos\left(\dfrac{\pi}{2} \left( \dfrac{\vert\vert \bm{x} - \bm{x}_c^q \vert\vert_{*^q} - (1 - \varepsilon_c) r_c^q}{\varepsilon_c r_c^q} \right) \right) & 
\text{ if } (1 - \varepsilon_c) r_c^q \leq \vert\vert \bm{x} - \bm{x}_c^q \vert\vert_{*^q} < r_c^q~, \\
0 & \text{ otherwise. }
\end{cases}
\end{equation}
For all $q \in \{1, \dots, N_c\}$, we select $\bm{\mu}_q^c \overset{iid}{\sim} Bern(1/N_c) \ \mathcal{U}(10^1, 10^3)$, where $Bern(p)$ indicates a Bernoulli distribution of parameter $p \in [0,1]$ and $\mathcal{U}(a,b)$ denotes a uniform distribution in the interval $[a,b]$, for $a,b \in \mathbb{R}$, $a < b$. We define the blood clots parameter domain $\mathcal{D}_c := [0, 10^3]^{N_c}$. In both test cases, we set $\varepsilon_c=0.1$.

Regarding the Dirichlet datum, let us introduce the function $\vec{g}_{ref}^{\bm{\mu}^f}: \Gamma \times [0,T] \to \mathbb{R}$ --- being $\Gamma$ a circular surface of radius $R$ and center $\bm{x}_0$ --- such that $\vec{g}_{ref}^{\bm{\mu}^f}(\bm{x}, t) = \vec{g}_{ref}^s(\bm{x})g_{ref}^t(t;\bm{\mu}^f)$. The quantity $\bm{\mu}^f$ defines the parameters that are related to the flow rate. Firstly, we define
\begin{equation}
\label{eq: reference g}
\vec{g}_{ref}^s(\bm{x}) := -\dfrac{2}{\pi R^2}\left(1-\dfrac{||\bm{x}-\bm{x}_0||^2}{R^2}\right)\vec{n}_{\Gamma} \qquad \text{ with } \bm{x} \in \Gamma~, \\ 
\end{equation}
which describes a parabolic velocity profile (Poiseuille flow), and $\vec{n}_{\Gamma}$ is the outward unit normal vector to $\Gamma$.

The Dirichlet datum is then defined as follows in the two test cases:
\begin{itemize}
	\item \textbf{Symmetric Bifurcation}: we impose a periodic--in--time parabolic--in--space velocity profile with a parametrized perturbation at the inlet $\Gamma_{IN}^1$ and we prescribe the flow rate, expressed as a given percentage of the inflow rate, at the outlet $\Gamma_{OUT}^1$. In particular, we have
	\begin{equation}
	\label{eq: g_bifurcation}
	\begin{alignedat}{2}
	g^t_{ref}(t;\bm{\mu}^f) &:= 1-\cos{\dfrac{2\pi t}{T}}+\bm{\mu}^f_1\sin{\dfrac{2\pi \bm{\mu}^f_0 t}{T}} & \qquad & t \in  [0,T] \\
	\vec{g}^{\bm{\mu}^f}(\bm{x},t) &:= 
	\begin{cases}
	\vec{g}^{\bm{\mu}^f}_{ref}(\bm{x},t)  & (\bm{x},t)\in \Gamma_{IN}^1 \times [0,T] \\
	\bm{\mu}^f_2 \ \vec{g}^{\bm{\mu}^f}_{ref}(\bm{x},t)  & (\bm{x},t)\in \Gamma_{OUT}^1 \times [0,T] 
	\end{cases}
	\end{alignedat}
	\end{equation}
	where $T = 0.3 \ s$ is the final time of the simulation. At the snapshot generation stage, we select $\bm{\mu}^f \sim \mathcal{U}\left(\mathcal{D}_f\right)$, $\mathcal{D}_f := [4,8] \times [0.1,0.3] \times [0.2,0.8]$. In particular: $\bm{\mu}^f_0$ models the frequency of the perturbation; $\bm{\mu}^f_1$ models the amplitude of the perturbation; $\bm{\mu}^f_2$ models the amount of flow coming out of $\Gamma_{OUT}^1$ (see Figure \ref{fig: flow rates} -- left). On $\Gamma_{OUT}^2$, we impose homogeneous Neumann BCs. This test case features Reynolds' numbers up to $1.5 \cdot 10^2$.
	
	\item \textbf{Femoropopliteal Bypass}: we impose two different parabolic velocity profiles at the two inlets $\Gamma_{IN}^1$,$\Gamma_{IN}^2$, constraining their sum to be the same for all the snapshots, i.e. 
	\begin{equation}
	\label{eq: g_bypass}
	\vec{g}^{\bm{\mu}^f}(\bm{x},t) := 
	\begin{cases}
	\bm{\mu}^f_4 \ \vec{g}^{\bm{\mu}^f_0}_{ref}(\bm{x},t)    & (\bm{x},t)\in \Gamma_{IN}^1 \times [0,T] \\
	(1-\bm{\mu}^f_4) \ \vec{g}^{\bm{\mu}^f_0}_{ref}(\bm{x},t) & (\bm{x},t)\in \Gamma_{IN}^2 \times [0,T]  
	\end{cases}
	\end{equation}
	Here $\vec{g}_{ref}^t(t; \bm{\mu}^f)$ has been chosen as a polynomial parametrization of the systolic part of the inflow profile employed e.g. in \cite{colciago2014comparisons, colciago2018reduced}, which depends on the parameters $\bm{\mu}^f_0$, $\bm{\mu}^f_1$, $\bm{\mu}^f_2$, $\bm{\mu}^f_3$ (see Figure \ref{fig: flow rates} -- right). During the snapshot generation phase, the parameter values are samples uniformly at random, so that $\bm{\mu}^f \sim \mathcal{U}\left(\mathcal{D}_f\right)$, $\mathcal{D}_f := [0.1070, 0.1653] \times [0.9246, 2.1574] \times [9.9127, 18.4093] \times [0.3130, 0.9390] \times [0.2, 0.8]$. In particular: $\bm{\mu}^f_0$ models the time delay of the systolic peak; $\bm{\mu}^f_1$ models the flow rate at the beginning of systole; $\bm{\mu}^f_2$ models the peak systolic flow; $\bm{\mu}^f_3$ models the flow rate at the end of systole; $\bm{\mu}^f_4$ models the partition of flow between the two inlets $\Gamma_{IN}^1$, $\Gamma_{IN}^2$. On $\Gamma_{OUT}^1$, we impose homogeneous Neumann BCs. This test case features Reynolds' numbers up to $8 \cdot 10^3$.
\end{itemize}

\begin{figure}[t]
	\includegraphics[scale=0.235]{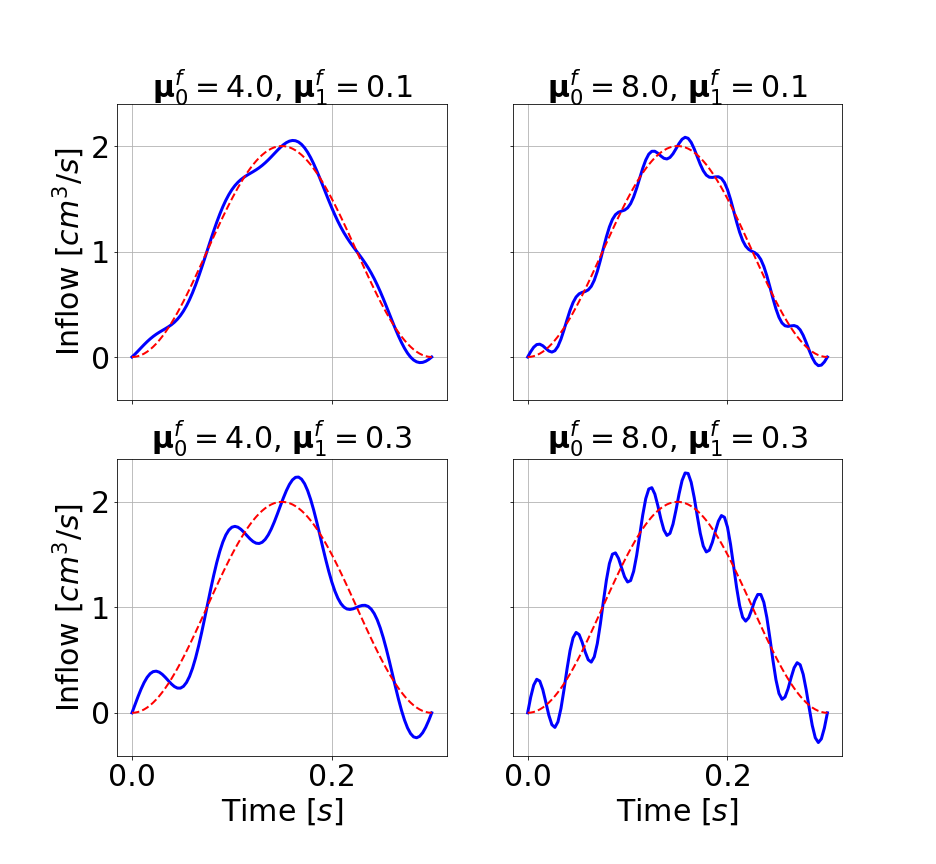}
	\hfill
	\includegraphics[scale=0.235]{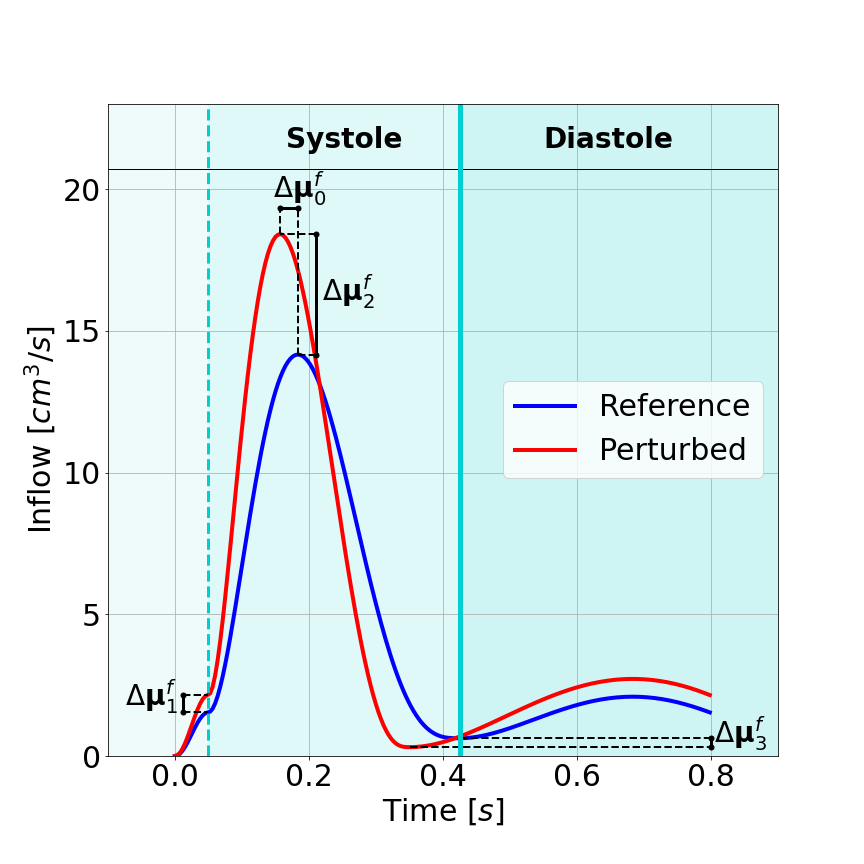}
	\caption{Plots of the parametrized reference flow rate functions $g_{ref}^t(t; \bm{\mu}^f)$ for the two considered test cases. In particular: (left) flow rate of the Symmetric Bifurcation test case for 4 different parameter values, at the extrema of the parameter domain; 
	(right) flow rate of the Femoropopliteal Bypass test case for parameters $\bm{\mu}^f_0=0.107, \ \bm{\mu}^f_1=2.157, \ \bm{\mu}^f_2=18.409, \ \bm{\mu}^f_3=0.313$ (in red), compared to the flow rate imposed in \cite{colciago2014comparisons, colciago2018reduced} (parameters: $\bm{\mu}^f_0=0.134, \ \bm{\mu}^f_1=1.541, \ \bm{\mu}^f_2=14.46, \ \bm{\mu}^f_3=0.626$) (in blue).}
	\label{fig: flow rates}
\end{figure}

As discussed in Subsection \ref{subs: discretization_stokes}, we imposed non--homogeneous Dirichlet BCs weakly, using Lagrange multipliers. Their space is discretized by means of orthonormal basis functions, built from Chebyshev polynomials. We considered polynomials up to the degree $n_{in}=5$ to impose inlet BCs, in order to get a good approximation of parabolic velocity profiles. Conversely, we chose $n_{out}=0$ to impose outlet BCs in the Bifurcation test case, since we are only interested in enforcing the outflow rate. We remark that in this work we use the \emph{cgs} (centimeter--gram--second) unit system; therefore, the velocity is expressed in $(cm/s)$ and the pressure in $(dyn/cm^2)$, where $dyn := g \cdot cm \cdot s^{-2}$. Concerning the computational environment, all simulations were run on the \emph{Scientific IT and Application Support (SCITAS)} clusters\footnote{\url{https://www.epfl.ch/research/facilities/scitas/hardware/}} at EPFL.

In both test cases, we generated $N_{\bm{\mu}} = 100$ training snapshots by solving the FOM problem (see Subsection \ref{subs: discretization_stokes}) for $100$ different parameter values, suitably sampled from $\mathcal{D} := \mathcal{D}_f \times \mathcal{D}_c$. We denote with $\mathcal{D}^{train} \subset \mathcal{D}$ the set of training parameters. To compute the FOM solutions, we employed the same computational framework of \cite{pegolotti2021model}, which is based on \emph{LifeV}, a C++ FE library with support to high--performance computing \cite{bertagna2017lifev}. We employed P2--P1 Taylor--Hood Lagrangian finite elements for the discretization of the velocity and pressure subspaces, respectively. We adopted BDF2 as time integrator; we remark that the space--time full--order linear system in Eq.\eqref{eq: monolitic_FOM_blocks_matrix} has not been explicitly assembled. The sparse linear systems arising at each time step have been solved using the preconditioned GMRES method with the saddle point block preconditioner proposed in \cite{pegolotti2021model}. Concerning the temporal velocity basis enrichment, we set $\varepsilon^t=0.9$, unless otherwise specified (see Algorithm \ref{alg: pressure time supremizers}).

For each test case, we performed $N_{\bm{\mu}}^* = 20$ additional FOM simulations in order to evaluate the performances of the proposed ROMs. We define $\mathcal{D}^{test} := \{\bm{\mu}_i^*\}_{i=1}^{N_{\bm{\mu}}^*}$ as the set of test parameters, sampled from $\mathcal{D}$ and such that $\mathcal{D}^{test} \cap \mathcal{D}^{train} = \emptyset$.
The performances of the proposed ROMs are assessed both in terms of accuracy and of computational efficiency. To evaluate the former, we consider the average (over the test snapshots) relative errors on velocity and pressure, measured in the norms induced by the symmetric and positive definite matrices $\bm{X}_u^{st}$ and $\bm{X}_p^{st}$ (see Eq.\eqref{eq: ST norms matrices}), respectively. Therefore, we define:
\begin{equation}
\label{eq: error definitions}
E_u = \frac{1}{N_{\bm{\mu}}^*}\overset{N_{\bm{\mu}}^*}{\underset{i=1}{\sum}}\dfrac{||\bm{\Pi}^u \widehat{\bm{u}}(\bm{\mu}_i^*) - \bm{u}^{st}_h(\bm{\mu}_i^*)||_{\bm{X}_u^{st}}}{||\bm{u}^{st}_h(\bm{\mu}_i^*) ||_{\bm{X}_u^{st}}};
\qquad 
E_p = \frac{1}{N_{\bm{\mu}}^*}\overset{N_{\bm{\mu}}^*}{\underset{i=1}{\sum}}\dfrac{||\bm{\Pi}^p \widehat{\bm{p}}(\bm{\mu}_i^*) - \bm{p}^{st}_h(\bm{\mu}_i^*)||_{\bm{X}_p^{st}}}{||\bm{p}^{st}_h(\bm{\mu}_i^*) ||_{\bm{X}_p^{st}}};
\end{equation}
being $[\widehat{\bm{u}}(\bm{\mu}_i^*), \widehat{\bm{p}}(\bm{\mu}_i^*), \widehat{\bm{\lambda}}(\bm{\mu}_i^*)] \in \mathbb{R}^{n^{st}}$ the space--time reduced solution obtained with the considered ROM for the parameter value $\bm{\mu}_i^* \in \mathcal{D}^{test}$ and $[\bm{u}_h^{st}(\bm{\mu}_i^*), \bm{p}_h^{st}(\bm{\mu}_i^*), \bm{\lambda}_h^{st}(\bm{\mu}_i^*)] \in \mathbb{R}^{N^{st}}$ the corresponding FOM solution. Notice that Eq.\eqref{eq: error definitions} applies also to the SRB--TFO method, where no dimensionality reduction in time takes place, by setting the temporal reduced bases as equal to the canonical one. Accuracy is then expressed by the ratios $E_u/\varepsilon_u$, $E_p/\varepsilon_p$, where $\varepsilon_u, \varepsilon_p$ are, respectively, the velocity and pressure POD tolerances in space and in time; these values being close to $1$ means that the ST--RB method is working somehow optimally. Concerning computational efficiency, we consider two different indicators, namely (1) the \emph{speedup} (SU), defined as the ratio between the average wall--time of FOM and of ROM simulations; (2) the \emph{reduction factor} (RF), computed as the ratio between the FOM DOFs ($N^{st}$) and the ROM DOFs ($n^{st}$).

\begin{table}[t]
	\centering
	\begin{tabular}{c c c c c c} 
		\toprule
		
		&\multicolumn{3}{c}{\textbf{Space DOFs}}
		&\multicolumn{1}{c}{\textbf{Time DOFs}} & \\
		
		\cmidrule(lr){2-4} \cmidrule(lr){5-5} 
		
		\multicolumn{1}{c}{\textbf{Test case}}
		&\multicolumn{1}{c}{\textbf{$\bm{N_u^s}$}}
		&\multicolumn{1}{c}{\textbf{$\bm{N_p^s}$}}
		&\multicolumn{1}{c}{\textbf{$\bm{N_\lambda}$}}
		&\multicolumn{1}{c}{\textbf{$\bm{N^t}$}}
		&\multicolumn{1}{c}{\textbf{Wall--time}} \\
		
		\midrule
		
		\textbf{Bifurcation}  & 76974 & 3552 & 66  & 120 & 1936 \si{\second}  \\ 
		\textbf{Bypass}  & 234936 &  10158 & 126 & 170 & 9300 \si{\second}  \\ 
		\bottomrule 
	\end{tabular}
	\caption{Space and time FOM DOFs for velocity, pressure and Lagrange multipliers and average wall--time of a full--order simulation in the two test cases.}
	\label{tab: fem info}
	
	\begin{tabular}{c c c c c c c c} 
		\toprule
		
		& &\multicolumn{5}{c}{\textbf{Bases construction}} & \\
		
		\cmidrule(lr){3-7} 
		
		\multicolumn{1}{c}{\textbf{Test case}}
		&\multicolumn{1}{c}{\textbf{Method}}
		&\multicolumn{2}{c}{\textbf{$\bm{u}$}}
		&\multicolumn{2}{c}{\textbf{$\bm{p}$}}
		&\multicolumn{1}{c}{\textbf{$\bm{\lambda}$}}
		&\multicolumn{1}{c}{\textbf{Assembling}} \\
		
		\cmidrule(lr){3-4}\cmidrule(lr){5-6}\cmidrule(lr){7-7}
		
		& &\textbf{Space} & \textbf{Time} &\textbf{Space} & \textbf{Time} & \textbf{Time} & \\
		
		\midrule
		
		\textbf{Bifurcation} & \textbf{SRB--TFO} & 930 \si{\second} & // & 7 \si{\second} & // & // & 5 \si{\second}\\ 
		& \textbf{ST--GRB} & 930 \si{\second} & 4 \si{\second} & 7 \si{\second} & $<1$ \si{\second} & $<1$ \si{\second} & 6 \si{\second} \\ 
		& \textbf{ST--PGRB} & 468 \si{\second} & 2 \si{\second} & 7 \si{\second} & $<1$ \si{\second} & $<1$ \si{\second} & 5 \si{\second} \\
		
		\textbf{Bypass} & \textbf{SRB--TFO} & 3216 \si{\second} & // & 28 \si{\second} & // & // & 29 \si{\second} \\ 
		& \textbf{ST--GRB} & 3216 \si{\second} & 29 \si{\second} & 28 \si{\second} & $<1$ \si{\second} & $<1$ \si{\second} & 30 \si{\second} \\ 
		& \textbf{ST--PGRB} & 1303 \si{\second} & 9 \si{\second} & 28 \si{\second} & $<1$ \si{\second}  & $<1$ \si{\second} & 20 \si{\second} \\
		
		\bottomrule 
	\end{tabular}
	\caption{Offline computational cost, expressed in terms of wall--time (in \si{\second}), of the SRB--TFO, ST--GRB and ST--PGRB methods in the two test cases for $\varepsilon_u=\varepsilon_p=\varepsilon_\lambda=10^{-3}$.}
	\label{tab: offline phase info}
\end{table}

Table~\ref{tab: fem info} reports the dimensionality of the full--order problem and the wall--time of a single high--fidelity simulation for the two considered test cases. Based on the numerical tests carried out in \cite{pegolotti2021model}, we choose a timestep size $\delta = 2.5 \cdot 10^{-3}$ \si{s} in both test cases, from which the reported values of $N^t$ follow. Convergence tests proved that such a choice does not lead to an overrefinement of the temporal grid; hence, ST--RB methods are not artificially favoured over full--order in time ones.
Table~\ref{tab: offline phase info} reports the offline computational cost of the reduced--order simulations for the three considered methods and in the two test cases.  Firstly, we highlight that all computational times in Table~\ref{tab: offline phase info} are a small fraction of a single FOM solve wall--time (see Table~\ref{tab: fem info}); thus snapshot generation configures as the actual bottleneck of the offline stage for all methods. Neglecting the FOM solves, the computation of the velocity reduced basis in space is by far the most expensive operation. Indeed, it involves the SVD of a very large matrix and, for the SRB--TFO and ST--GRB methods, also the supremizers enrichment procedure. The non--negligible cost of the latter makes ST--PGRB the most efficient approach in both test cases. Concerning the reduced bases, we notice that computing the temporal modes is much faster than extracting the spatial ones. This is a direct consequence of exploiting dimensionality reduction in space prior to capturing the most relevant dynamics of the problem. For the ST--GRB method, we remark that the temporal stabilizers enrichment takes $\approx 2$~\si{\second} for the Symmetric Bifurcation test case and $\approx 20$~\si{\second} for the Femoropopliteal Bypass one. Stabilization in time is then much faster than stabilization in space, since the number of temporal FOM DOFs is significantly lower than the one of spatial FOM DOFs. Finally, we underline that the cost of model order reduction in space is dominant also for what concerns the (offline) computation of parameter--independent quantities. Indeed, the assembling wall--time for the ST--GRB method is only $\approx 1$~\si{\second} higher than the one for the SRB--TFO method, even though the latter involves only spatial projections.

\subsection{Symmetric Bifurcation}\label{subs:results bif}

Table~\ref{tab: bif_results} reports the results obtained with the SRB--TFO, ST--GRB and ST--PGRB  methods on the Symmetric Bifurcation test case, for $\varepsilon_u=10^{-3}$, $\varepsilon_p=\varepsilon_\lambda=10^{-4}$. Different enrichments of the velocity reduced basis in time are considered. For each test, we indicate the dimension (in space and in time) of the velocity reduced basis, the \emph{reduction factor}, the \emph{speedup} and the average (normalized) relative test errors on velocity and pressure (see Eq.\eqref{eq: error definitions}). Figure \ref{fig:bifurcation_plots} reports the line integral convolution on the median slice for the velocity fields obtained for $\bm{\mu}^* = [7.85,0.14,0.59,45.2,877]$ with the three considered ROMs at $t=0.15$~\si{\second} (top row) and the corresponding absolute pointwise errors with respect to the FOM solution (bottom row).

\begin{table}[t!]
	\centering
	\begin{tabular}{ccccccc}
		\toprule
		
		\multicolumn{3}{c}{\large{\textbf{Method}}}&
		\multicolumn{1}{c}{\large\textbf{ROM size}}&
		\multicolumn{1}{c}{\large\textbf{Efficiency}}&
		\multicolumn{2}{c}{\large\textbf{Error}}\\	
		
		\cmidrule(lr){1-3} \cmidrule(lr){4-4} \cmidrule{5-5} \cmidrule(lr){6-7}
		
		&
		\textbf{T-sup} &
		$\bm{\varepsilon^t}$ &
		$\bm{(n_u^s,n_u^t)}$ & 
		\textbf{SU}& 
		$\bm{E_u / \varepsilon_u}$ & 
		$\bm{E_p / \varepsilon_p}$ \\
		
		\midrule
		
		\textbf{SRB--TFO} &\textbf{//} &\textbf{//} &(156,120) &9.37e2  &$1.00$ & $1.36$ \\
		\textbf{ST--GRB} &\textbf{//}  &\textbf{//} &(156,17) &1.69e3  &1.11e17 & 2.93e40 \\
		&\textbf{P} &\textbf{0.6} &(156,$23^{*_6}$) &1.06e3  & 1.01 & 1.82  \\
		& &\textbf{0.9} &(156,$23^{*_6}$) &1.05e3 &1.01 & 2.42  \\
		&\textbf{L} &\textbf{0.6} &(156,$23^{*_6}$) &1.04e3  & 1.01 & 2.64 \\
		& &\textbf{0.9} &(156,$23^{*_6}$)  & 1.05e3 & 1.02 & 2.80 \\          
		\textbf{ST--PGRB} &\textbf{//}  &\textbf{//} &(66,17) & 3.62e3  &3.35 & 40.9 \\
		&\textbf{P} &\textbf{0.6} &(66,$23^{*_6}$) & 2.64e3  &1.29 & 2.52 \\
		& &\textbf{0.9} &(66,$23^{*_6}$) & 2.62e3 &1.28 & 3.02 \\
		
		\bottomrule
	\end{tabular}%
	\caption{Dimensions of the velocity reduced bases, speedups and errors obtained on the Symmetric Bifurcation test case for different enrichments of the velocity temporal basis, setting $\varepsilon_u=10^{-3}$, $\varepsilon_p=\varepsilon_\lambda=10^{-4}$. Notation: ``//'' means that no enrichment has been performed, ``P'' that enrichment has been performed with respect to pressure modes, ``L'' that enrichment has been performed with respect to Lagrange multipliers modes. The notation $(\cdot)^{*_n}$ indicates that the velocity temporal basis has been enriched with $n$ modes.}
	\label{tab: bif_results}
	
	\includegraphics[width=0.73\textwidth]{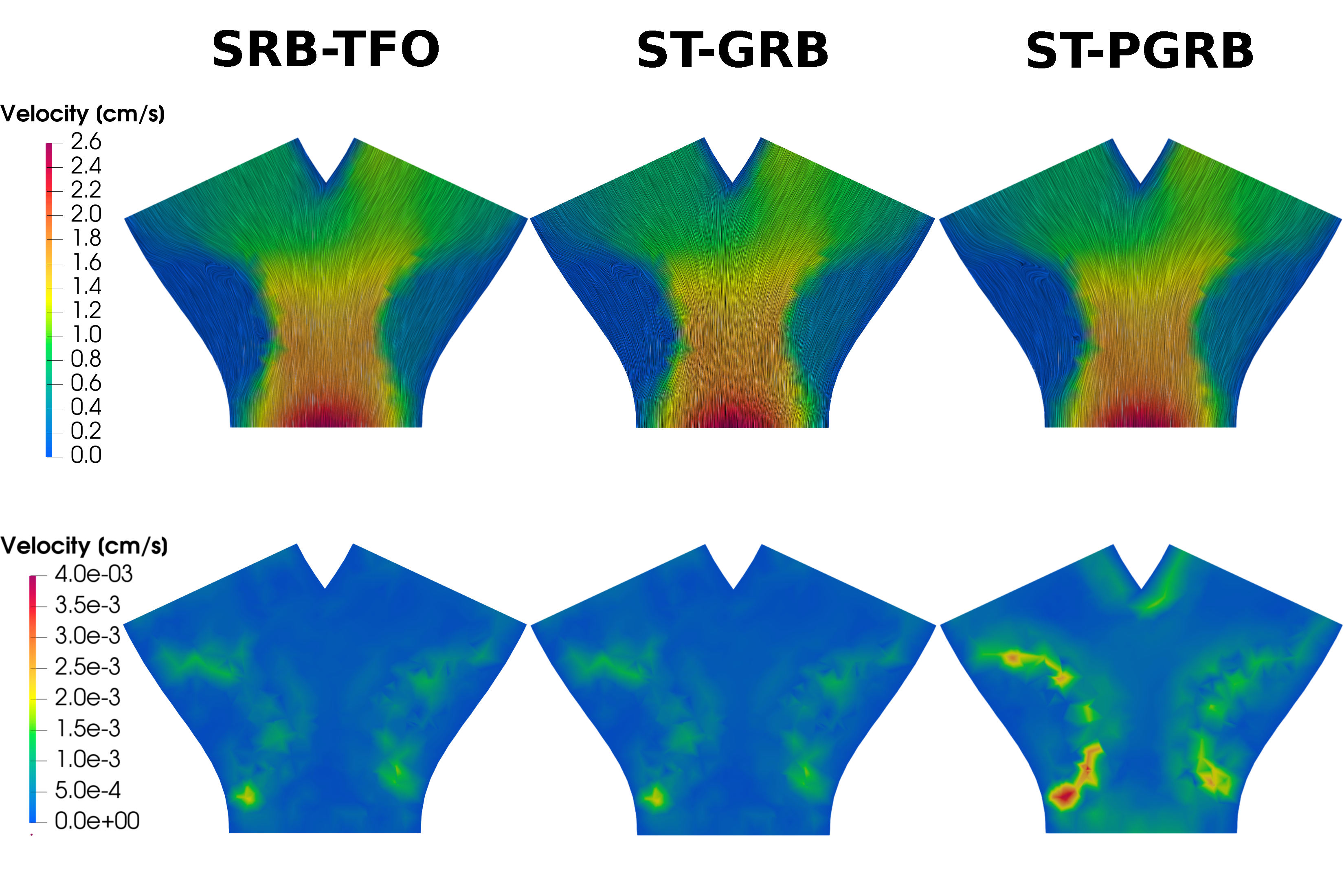}
	\captionof{figure}{Surface line integral convolution of the velocity field on the median slice (top) and corresponding absolute pointwise errors with respect to the FOM solution (bottom), achieved in the Symmetric Bifurcation test case with the three considered ROMs for $\varepsilon_u=10^{-3}$, $\varepsilon_p=\varepsilon_\lambda=10^{-4}$, with $\bm{\mu}^*=[7.85, 0.14, 0.59, 45.2, 877]$ and at $t=0.15$ \si{\second}.}
	\label{fig:bifurcation_plots}
\end{table}

The impact of dimensionality reduction is evident. Indeed, all methods realize significant \emph{SU} with respect to the FOM and attain accuracies of the order of the prescribed POD tolerances. We notice that the efficiency is improved upon dimensionality reduction in time. Indeed, all ST--RB methods are more than 1000 times faster than the FOM, while the SRB--TFO method (implemented by iteratively solving $N^t$ small linear systems, using the BDF2 time marching scheme) is slower. However, the efficiency gain due to dimensionality reduction in time is not dramatic, since the number of timesteps ($N^t=120$) is small compared to the number of spatial DOFs (see Table \ref{tab: fem info}). More significant efficiency gains are expected if a higher number of temporal DOFs is considered (see Subsection \ref{subs:results bypass}). The ST--PGRB method is faster than the ST--GRB one, thanks to its ``automatic'' \emph{inf--sup} stability property (see Subsection \ref{subs: ST--PGRB stability}). This prevents from the enrichment of the velocity reduced bases (in space and in time) and it ultimately leads to solve a smaller linear system. The drawback of such an increased computational efficiency is represented by a slight loss in accuracy. Indeed, while the average relative test errors for the SRB--TFO method are of the order of the POD tolerance, the ones got with the two ST--RB methods are larger, particularly for the pressure field. 

Setting $\varepsilon_u = 10^{-3}$ and $\varepsilon_p = \varepsilon_\lambda = 10^{-4}$, both dual fields feature temporal reduced bases that are larger than the primal one. As a consequence, the matrices $\bm{\Psi}^{u,p}$, $\{\bm{\Psi}^{u, \lambda_k}\}_{k=1}^{N_D}$ cannot be full column rank and --- according to Corollaries \ref{cor: corollary 1} -- \ref{cor: corollary 2} --- the velocity temporal basis enrichment is compulsory in order to retain \emph{inf--sup} stability with the ST--GRB method. Indeed, if the velocity temporal reduced basis is not augmented, the linear system stemming from the application of the ST--GRB method is ill--conditioned and consequently the errors (on both velocity and pressure) explode. Conversely, if suitable modes are added to the velocity temporal basis, the problem is stabilized and the errors are comparable with the prescribed POD tolerances. Three remarks are worth to follow. Firstly, the velocity temporal stabilizers obtained for $\varepsilon^t=\varepsilon^t_1$ are not necessarily a subset of those computed considering $\varepsilon^t=\varepsilon^t_2 > \varepsilon^t_1$. For instance, in this case the temporal modes that are added choosing $\varepsilon^t=0.6$ and $\varepsilon^t=0.9$, although being $6$ in both cases, are different. Secondly, the order in which the dual fields (i.e. pressure and Lagrange multipliers) are considered for the velocity temporal basis enrichment does not significantly affect the results. Lastly, the ST--PGRB method, despite being roughly $3$ times faster, exhibits higher errors than the (stabilized) ST--GRB one. In particular, the relative error on pressure is one order of magnitude larger than the prescribed POD tolerance. Nonetheless, the last two rows of Table~\ref{tab: bif_results} show that accuracies comparable with the ones of ST--GRB are retrieved by enriching the velocity temporal reduced basis (according to Algorithm \ref{alg: pressure time supremizers}), without severely impacting the efficiency of the method. For instance, the addition of $6$ temporal stabilizers via Algorithm \ref{alg: pressure time supremizers} (with $\varepsilon^t = 0.6$ and considering pressure as dual field) leads to a significant drop in the relative error for both velocity (from $3.35$ to $1.29$, $-61\%$) and pressure (from $40.9$ to $2.52$, $-94\%$).

\subsection{Femoropopliteal Bypass}\label{subs:results bypass}
\afterpage{
	\begin{table}[t!]
		\centering
		
		\begin{tabular}{c c c c c c c c c}
			\toprule
			& &
			\multicolumn{3}{c}{\large\textbf{ROM size}}&
			\multicolumn{2}{c}{\large\textbf{Efficiency}}&
			\multicolumn{2}{c}{\large\textbf{Accuracy}}\\
			
			\cmidrule(lr){3-5} \cmidrule(lr){6-7} \cmidrule(lr){8-9}
			
			&
			$\bm{\varepsilon}$ & 
			\multicolumn{1}{c}{$\bm{(n_u^s,n_u^t)}$} & 
			\multicolumn{1}{c}{$\bm{(n_p^s,n_p^t)}$} &
			\multicolumn{1}{c}{$\bm{(\{N_{\lambda_k}\},\{n_{\lambda_k}^t\})}$} &
			\multicolumn{1}{c}{\textbf{RF}} & 
			\multicolumn{1}{c}{\textbf{SU}} & 
			\multicolumn{1}{c}{$\bm{E_u / \varepsilon_u}$} & 
			\multicolumn{1}{c}{$\bm{E_p / \varepsilon_p}$} \\
			
			\midrule
			
			\multirow{3}{*}{\rotatebox[origin=c]{90}{\textbf{\footnotesize SRB--TFO}}} & $\bm{10^{-3}}$ &(184,170) &(7,170) &(\{63,63\},\{170,170\}) &7.74e2 &2.11e3 &11.76 &2.40 \\
			& $\bm{10^{-4}}$ &(220,170) &(14,170) &(\{63,63\},\{170,170\}) &6.81e2 &1.79e3 &14.59 &2.70  \\
			& $\bm{10^{-5}}$ &(303,170) &(31,170) &(\{63,63\},\{170,170\}) &5.33e2 &1.55e3 &12.05 &8.13 \\
			
			\midrule
			
			\multirow{3}{*}{\rotatebox[origin=c]{90}{\textbf{\footnotesize ST--GRB}}} & $\bm{10^{-3}}$ &(184,$10^{*_3}$) &(7,9) &(\{63,63\},\{8,8\}) &1.43e4 &2.18e4 &12.74 &4.29 \\
			& $\bm{10^{-4}}$ &(220,$12^{*_2}$) &(14,11) &(\{63,63\},\{11,11\}) &9.97e3 &8.19e4 &14.91 &3.17  \\
			& $\bm{10^{-5}}$ &(303,$16^{*_4}$) &(31,14) &(\{63,63\},\{14,14\}) &5.92e3 &2.39e3 &10.02 &9.20 \\
			
			\midrule
			
			\multirow{3}{*}{\rotatebox[origin=c]{90}{\textbf{\footnotesize ST--PGRB \hspace{0.5cm}}}} & $\bm{10^{-3}}$ &(51,7) &(7,9) &(\{63,63\},\{8,8\}) &2.92e4 &1.17e5 &16.70 &8.22 \\
			& &(51,$10^{*_3}$) &(7,9) &(\{63,63\},\{8,8\}) &2.64e4 &9.17e4 &12.28 &5.48 \\
			& $\bm{10^{-4}}$ &(87,10) &(14,11) &(\{63,63\},\{11,11\}) &1.73e4 &3.19e4 &11.95 &11.27 \\
			& &(87,$12^{*_2}$) &(14,11) &(\{63,63\},\{11,11\}) &1.61e3 &2.68e4 &9.71 &10.79  \\
			& $\bm{10^{-5}}$ &(146,12) &(31,14) &(\{63,63\},\{14,14\}) &1.06e4 &9.00e3 &21.98 &19.34 \\
			& &(146,$16^{*_4}$) &(31,14) &(\{63,63\},\{14,14\}) &9.19e3 &6.00e3 &13.45 &11.60 \\
			\bottomrule 
		\end{tabular}%
		\caption{Summary of the results obtained with the ST--RB methods (top: ST--GRB, bottom: ST--PGRB) on the Femoropopliteal Bypass test case, for different POD tolerances $\varepsilon$. In particular: (left) number of spatial and temporal reduced basis elements for velocity, pressure and Lagrange multipliers; (center) RF and average SU; (right) normalized average test relative errors on velocity and pressure.}
		\label{tab:bypass ST--RB}
		
		\includegraphics[width=0.495\textwidth]{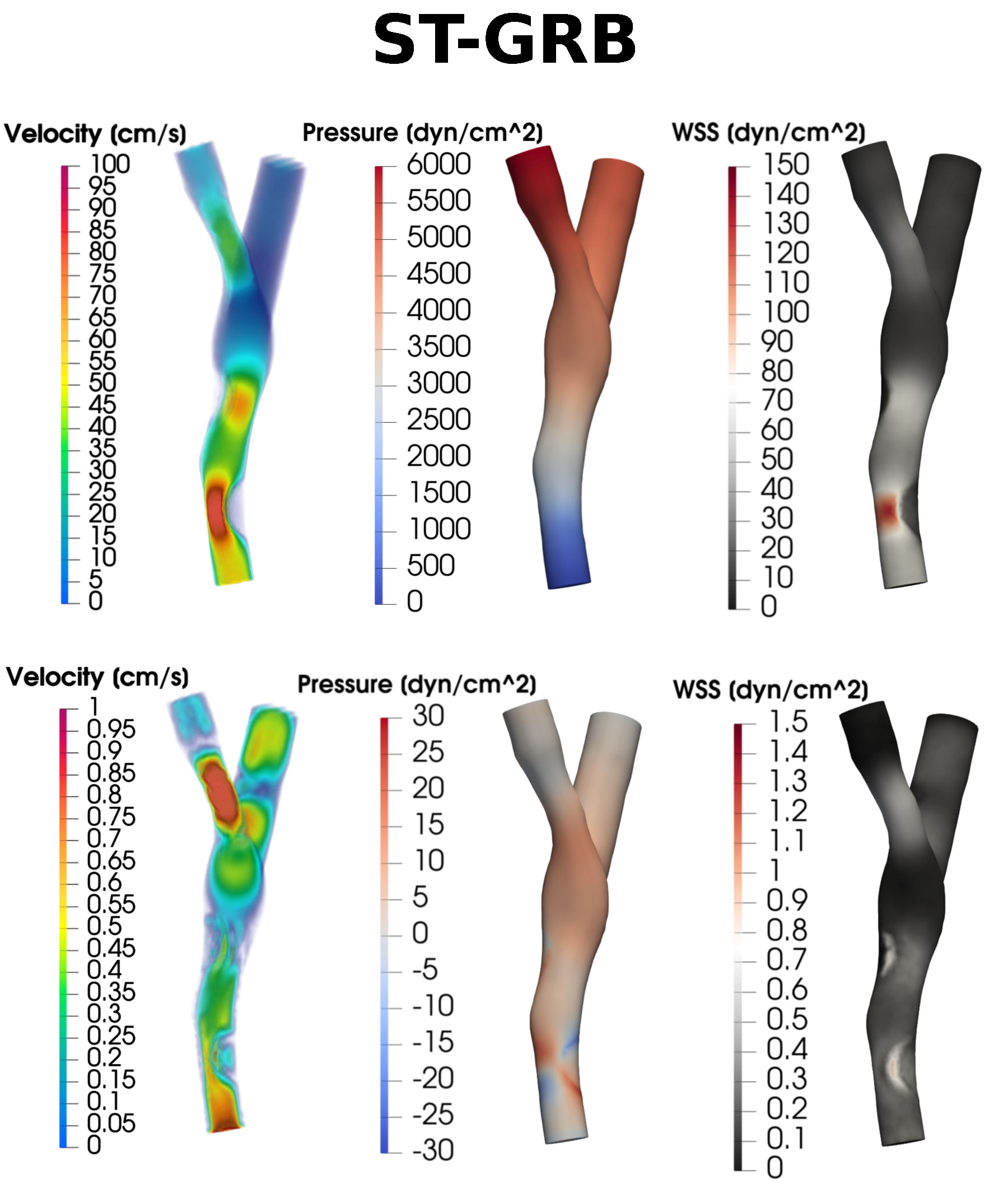}
		\hfill
		\includegraphics[width=0.495\textwidth]{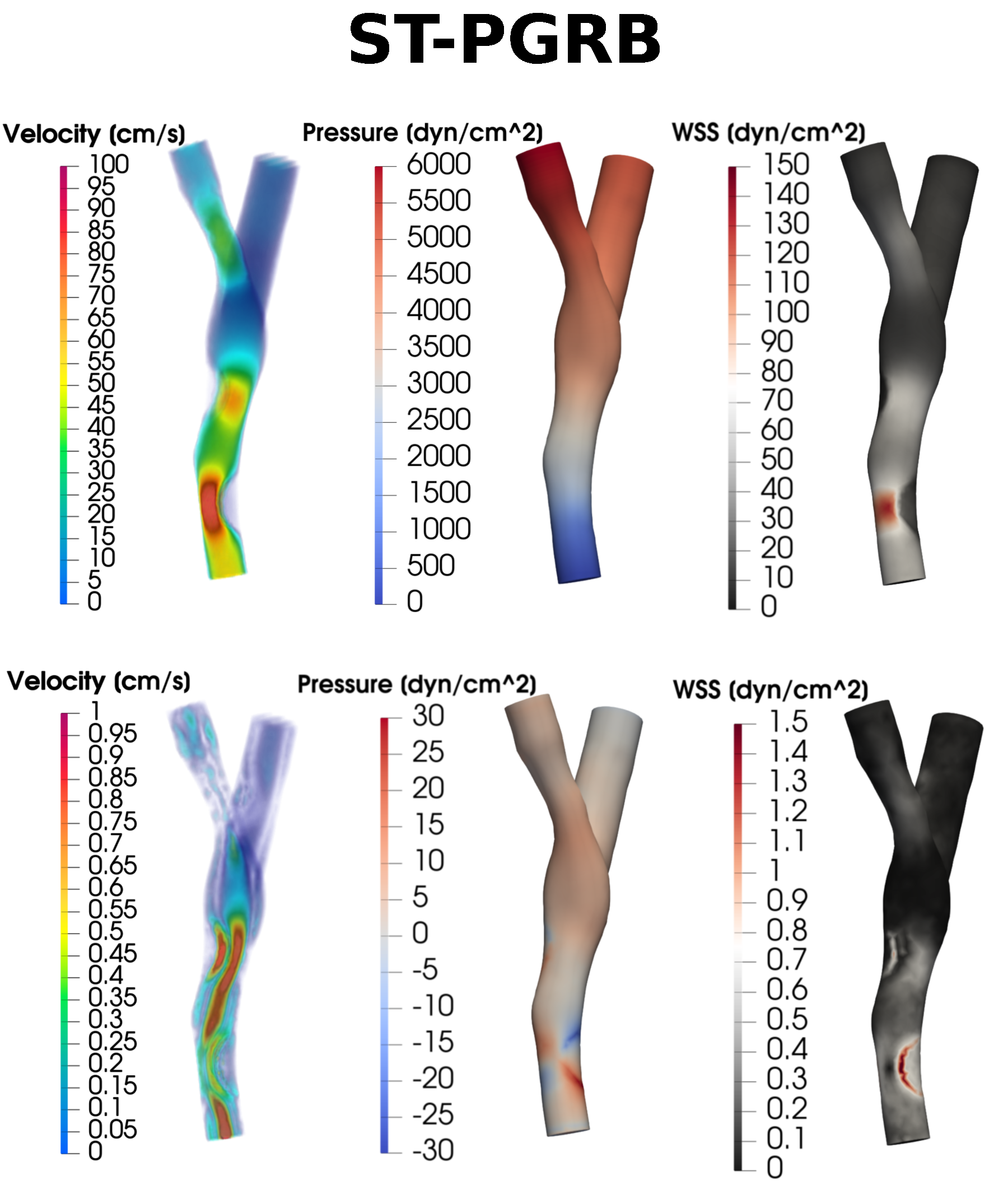}
		\captionof{figure}{Magnitudes of the velocity, pressure and WSS fields (top) and corresponding absolute pointwise errors with respect to the FOM solution (bottom), achieved in the Femoropopliteal Bypass test case with the ST--GRB method (left) and ST--PGRB method (right) for $\varepsilon=10^{-3}$, at $t=0.1075$ \si{\second} (systolic peak) and with $\bm{\mu}^*=[0.130, 1.604, 15.025, 0.714, 0.57, 536, 846, 0, 23.3]$.}
		\label{fig:bypass_plots}	
	\end{table}
}

We now assess the performances of the two considered ST--RB methods on a patient--specific Femoropopliteal Bypass geometry \cite{marchandise2012quality}. Table~\ref{tab:bypass ST--RB} reports the results obtained with the ST--GRB and ST--PGRB methods for three different POD tolerances. The same POD tolerance $\varepsilon \in \mathbb{R}^+$ has been chosen for all the fields, both in space and in time. Figure \ref{fig:bypass_plots} shows the magnitudes of the velocity, pressure and wall shear stress (WSS) fields for $\varepsilon = 10^{-3}$, at $t = 0.1075$~\si{\second} (systolic peak) and for $\bm{\mu}^*=[0.130, 1.604, 15.025, 0.714, 0.57, 536, 846, 0, 23.3]$ obtained with the ST--GRB (left) and ST--PGRB (right) method (top row) and the corresponding absolute pointwise errors with respect to the FOM solution (bottom row).

The performances of both ST--RB methods are good. Indeed, they both realize significant speedups with respect to the FOM and relative errors are approximately one order of magnitude higher than the prescribed POD tolerance. Moreover, the computational gain with respect to the baseline SRB--TFO approach is more evident than for the Symmetric Bifurcation test case, since the number of timesteps is higher ($N^t=170$ vs. $N^t=120$). For instance, for $\varepsilon=10^{-3}$, the ST--GRB and ST--PGRB methods are roughly $10$ and $55$ times faster than the baseline. Nevertheless, we remark that all methods feature larger errors compared to the first test case. We think this is mainly due to the increase of the Reynolds' numbers (maximal values at the systolic peak: ${\approx8\cdot10^3}$ vs. ${\approx1.5\cdot10^2}$), which hinders the approximation quality of the reduced subspaces. As in the Symmetric Bifurcation test case, the ST--PGRB method is faster that the ST--GRB one thanks to its ``automatic'' \emph{inf--sup} stability, which prevents from enriching the velocity reduced bases. However, velocity temporal basis enrichment is necessary in order to retain accuracies comparable with the one of ST--GRB (stabilized). 

\section{Conclusions}
\label{sec:section5}
In this work, we discussed the application of space--time reduced basis methods to the unsteady incompressible Stokes equations in fixed 3D geometries, endowed with a reaction term in order to model the presence of blood clots. We supposed the parametric dependency to characterize the blood clots densities and the inhomogeneous Dirichlet BCs, weakly imposed by means of Lagrange multipliers. As a result, the problem at hand features a twofold saddle point structure. Upon detailing the application of ST--RB methods, we focused on the well--posedness of the resulting problem, which resorts to \emph{inf--sup} stability analysis. To this aim, we proposed two different approaches. The first one, called ST--GRB, involves a Galerkin projection and it is characterized by a suitable enrichment of the spatio--temporal velocity reduced basis. The second one, called ST--PGRB, features instead a Petrov--Galerkin projection, stemming from the minimization of the FOM residual in a suitable norm, so that \emph{inf--sup} stability can be ``automatically'' guaranteed. Both methods showed efficiency gains with respect to the baseline space--reduced approach and accuracies in accordance with theoretical expectations.

Two main limitations can be identified. On the one hand, we focused on a linear problem, neglecting the non--linear convective term that characterizes the Navier--Stokes equations and whose role is of primary importance for the realistic modelling of h{\ae}modynamics. On the other hand, we considered fixed geometries, while one major challenge in patient--specific model order reduction is to efficiently deal with intra--patient and inter--patient geometrical variability. Therefore, we plan to extend the proposed ST--RB approaches to the incompressible unsteady Navier--Stokes equations in parametrized 3D geometries, possibly leveraging modular geometrical approximation \cite{pegolotti2021model}. Furthermore, we envision the use of RFSI models --- as the Coupled Momentum model \cite{figueroa2006coupled, colciago2014comparisons, colciago2018reduced} --- and of physiological BCs \cite{vignon2006outflow, kim2010patient} --- in order to bridge the gap with clinical applications.

\bibliographystyle{unsrtnat}
\bibliography{main_arxiv_V2}

\begin{thebibliography}{44}
\providecommand{\natexlab}[1]{#1}
\providecommand{\url}[1]{\texttt{#1}}
\expandafter\ifx\csname urlstyle\endcsname\relax
  \providecommand{\doi}[1]{doi: #1}\else
  \providecommand{\doi}{doi: \begingroup \urlstyle{rm}\Url}\fi

\bibitem[Lassila et~al.(2013)Lassila, Manzoni, Quarteroni, and
  Rozza]{lassila2013model}
T.~Lassila, A.~Manzoni, A.~Quarteroni, and G.~Rozza.
\newblock Model order reduction in fluid dynamics: challenges and perspectives.
\newblock \emph{MATHICSE-CMCS Modelling and Scientific Computing}, 2013.

\bibitem[Choi and Carlberg(2019)]{choi2019space}
Youngsoo Choi and Kevin Carlberg.
\newblock Space--time least-squares {P}etrov--{G}alerkin projection for
  nonlinear model reduction.
\newblock \emph{SIAM Journal on Scientific Computing}, 41\penalty0
  (1):\penalty0 A26--A58, 2019.

\bibitem[Maday and Turinici(2002)]{maday2002parareal}
Yvon Maday and Gabriel Turinici.
\newblock A parareal in time procedure for the control of partial differential
  equations.
\newblock \emph{Comptes Rendus Mathematique}, 335\penalty0 (4):\penalty0
  387--392, 2002.

\bibitem[Falgout et~al.(2014)Falgout, Friedhoff, Kolev, MacLachlan, and
  Schroder]{falgout2014parallel}
Robert~D Falgout, Stephanie Friedhoff, Tz~V Kolev, Scott~P MacLachlan, and
  Jacob~B Schroder.
\newblock Parallel time integration with multigrid.
\newblock \emph{SIAM Journal on Scientific Computing}, 36\penalty0
  (6):\penalty0 C635--C661, 2014.

\bibitem[Farhat and Chandesris(2003)]{farhat2003time}
Charbel Farhat and Marion Chandesris.
\newblock Time--decomposed parallel time--integrators: theory and feasibility
  studies for fluid, structure, and fluid--structure applications.
\newblock \emph{International Journal for Numerical Methods in Engineering},
  58\penalty0 (9):\penalty0 1397--1434, 2003.

\bibitem[Carlberg et~al.(2015)Carlberg, Ray, and van
  Bloemen~Waanders]{carlberg2015decreasing}
Kevin Carlberg, Jaideep Ray, and Bart van Bloemen~Waanders.
\newblock Decreasing the temporal complexity for nonlinear, implicit
  reduced--order models by forecasting.
\newblock \emph{Computer Methods in Applied Mechanics and Engineering},
  289:\penalty0 79--103, 2015.

\bibitem[Carlberg et~al.(2017)Carlberg, Barone, and
  Antil]{carlberg2017galerkin}
Kevin Carlberg, Matthew Barone, and Harbir Antil.
\newblock {G}alerkin v. least--squares {P}etrov--{G}alerkin projection in
  nonlinear model reduction.
\newblock \emph{Journal of Computational Physics}, 330:\penalty0 693--734,
  2017.

\bibitem[Urban and Patera(2012)]{urban2012new}
Karsten Urban and Anthony~T Patera.
\newblock A new error bound for reduced basis approximation of parabolic
  partial differential equations.
\newblock \emph{Comptes Rendus Mathematique}, 350\penalty0 (3-4):\penalty0
  203--207, 2012.

\bibitem[Urban and Patera(2014)]{urban2014improved}
Karsten Urban and Anthony~T Patera.
\newblock An improved error bound for reduced basis approximation of linear
  parabolic problems.
\newblock \emph{Mathematics of Computation}, 83\penalty0 (288):\penalty0
  1599--1615, 2014.

\bibitem[Yano(2014)]{yano2014space}
Masayuki Yano.
\newblock A space--time {P}etrov--{G}alerkin certified reduced basis method:
  application to the {B}oussinesq equations.
\newblock \emph{SIAM Journal on Scientific Computing}, 36\penalty0
  (1):\penalty0 A232--A266, 2014.

\bibitem[Yano et~al.(2014)Yano, Patera, and Urban]{yano2014space2}
Masayuki Yano, Anthony~T Patera, and Karsten Urban.
\newblock A space--time hp--interpolation--based certified reduced basis method
  for {B}urgers' equation.
\newblock \emph{Mathematical Models and Methods in Applied Sciences},
  24\penalty0 (09):\penalty0 1903--1935, 2014.

\bibitem[Shimizu and Parish(2021)]{shimizu2021windowed}
Y.~S. Shimizu and E.~J. Parish.
\newblock Windowed space--time least--squares {P}etrov--{G}alerkin model order
  reduction for nonlinear dynamical systems.
\newblock \emph{Computer Methods in Applied Mechanics and Engineering},
  386:\penalty0 114050, 2021.

\bibitem[Choi et~al.(2021)Choi, Brown, Arrighi, Anderson, and
  Huynh]{choi2021space}
Y.~Choi, P.~Brown, W.~Arrighi, R.~Anderson, and K.~Huynh.
\newblock Space–time reduced order model for large-scale linear dynamical
  systems with application to {B}oltzmann transport problems.
\newblock \emph{Journal of Computational Physics}, 424, 2021.

\bibitem[Kim et~al.(2021)Kim, Wang, and Choi]{kim2021efficient}
Y.~Kim, K.~Wang, and Y.~Choi.
\newblock Efficient space–time reduced order model for linear dynamical
  systems in {P}ython using less than 120 lines of code.
\newblock \emph{Mathematics 2021}, 9\penalty0 (14), 2021.

\bibitem[Pegolotti et~al.(2021)Pegolotti, Pfaller, Marsden, and
  Deparis]{pegolotti2021model}
Luca Pegolotti, Martin~R Pfaller, Alison~L Marsden, and Simone Deparis.
\newblock Model order reduction of flow based on a modular geometrical
  approximation of blood vessels.
\newblock \emph{Computer methods in applied mechanics and engineering},
  380:\penalty0 113762, 2021.

\bibitem[Gatica and Sayas(2008)]{gatica2008characterizing}
Gabriel~N Gatica and Francisco-Javier Sayas.
\newblock Characterizing the inf--sup condition on product spaces.
\newblock \emph{Numerische Mathematik}, 109\penalty0 (2):\penalty0 209--231,
  2008.

\bibitem[Howell and Walkington(2011)]{howell2011inf}
Jason~S Howell and Noel~J Walkington.
\newblock Inf--sup conditions for twofold saddle point problems.
\newblock \emph{Numerische Mathematik}, 118\penalty0 (4):\penalty0 663--693,
  2011.

\bibitem[Manzoni(2012)]{manzoni2012reduced}
Andrea Manzoni.
\newblock Reduced models for optimal control, shape optimization and inverse
  problems in haemodynamics.
\newblock Technical report, EPFL, 2012.

\bibitem[Quarteroni et~al.(2000)Quarteroni, Tuveri, and
  Veneziani]{quarteroni2000computational}
Alfio Quarteroni, Massimiliano Tuveri, and Alessandro Veneziani.
\newblock Computational vascular fluid dynamics: problems, models and methods.
\newblock \emph{Computing and Visualization in Science}, 2\penalty0
  (4):\penalty0 163--197, 2000.

\bibitem[Boffi et~al.(2013)Boffi, Brezzi, Fortin, et~al.]{boffi2013mixed}
Daniele Boffi, Franco Brezzi, Michel Fortin, et~al.
\newblock \emph{Mixed finite element methods and applications}, volume~44.
\newblock Springer, 2013.

\bibitem[Brezzi(1974)]{brezzi1974existence}
Franco Brezzi.
\newblock On the existence, uniqueness and approximation of saddle--point
  problems arising from {L}agrangian multipliers.
\newblock \emph{Publications math{\'e}matiques et informatique de Rennes},
  \penalty0 (S4):\penalty0 1--26, 1974.

\bibitem[Quarteroni and Valli(2008)]{quarteroni2008numerical}
Alfio Quarteroni and Alberto Valli.
\newblock \emph{Numerical approximation of partial differential equations},
  volume~23.
\newblock Springer Science \& Business Media, 2008.

\bibitem[Hood and Taylor(1974)]{hood1974navier}
P~Hood and C~Taylor.
\newblock Navier-{S}tokes equations using mixed interpolation.
\newblock \emph{Finite element methods in flow problems}, pages 121--132, 1974.

\bibitem[Quarteroni et~al.(2015)Quarteroni, Manzoni, and
  Negri]{quarteroni2015reduced}
Alfio Quarteroni, Andrea Manzoni, and Federico Negri.
\newblock \emph{Reduced basis methods for partial differential equations: an
  introduction}, volume~92.
\newblock Springer, 2015.

\bibitem[Golub and Van~Loan(2013)]{golub2013matrix}
Gene~H Golub and Charles~F Van~Loan.
\newblock \emph{Matrix computations}.
\newblock JHU press, 2013.

\bibitem[Halko et~al.(2011)Halko, Martinsson, and Tropp]{halko2011finding}
Nathan Halko, Per-Gunnar Martinsson, and Joel~A Tropp.
\newblock Finding structure with randomness: probabilistic algorithms for
  constructing approximate matrix decompositions.
\newblock \emph{SIAM review}, 53\penalty0 (2):\penalty0 217--288, 2011.

\bibitem[Chaturantabut and Sorensen(2010)]{chaturantabut2010nonlinear}
Saifon Chaturantabut and Danny~C Sorensen.
\newblock Nonlinear model reduction via discrete empirical interpolation.
\newblock \emph{SIAM Journal on Scientific Computing}, 32\penalty0
  (5):\penalty0 2737--2764, 2010.

\bibitem[Deparis(2008)]{deparis2008reduced}
Simone Deparis.
\newblock Reduced basis error bound computation of parameter-dependent
  {N}avier--{S}tokes equations by the natural norm approach.
\newblock \emph{SIAM journal on numerical analysis}, 46\penalty0 (4):\penalty0
  2039--2067, 2008.

\bibitem[Deparis and Rozza(2009)]{deparis2009reduced}
Simone Deparis and Gianluigi Rozza.
\newblock Reduced basis method for multi--parameter-dependent steady
  {N}avier--{S}tokes equations: applications to natural convection in a cavity.
\newblock \emph{Journal of Computational Physics}, 228\penalty0 (12):\penalty0
  4359--4378, 2009.

\bibitem[Negri et~al.(2015)Negri, Manzoni, and Rozza]{negri2015reduced}
Federico Negri, Andrea Manzoni, and Gianluigi Rozza.
\newblock Reduced basis approximation of parametrized optimal flow control
  problems for the {S}tokes equations.
\newblock \emph{Computers \& Mathematics with Applications}, 69\penalty0
  (4):\penalty0 319--336, 2015.

\bibitem[Rozza et~al.(2013)Rozza, Huynh, and Manzoni]{rozza2013reduced}
Gianluigi Rozza, DB~Huynh, and Andrea Manzoni.
\newblock Reduced basis approximation and a posteriori error estimation for
  {S}tokes flows in parametrized geometries: roles of the inf--sup stability
  constants.
\newblock \emph{Numerische Mathematik}, 125\penalty0 (1):\penalty0 115--152,
  2013.

\bibitem[Rozza(2005)]{rozza2005optimization}
Gianluigi Rozza.
\newblock On optimization, control and shape design of an arterial bypass.
\newblock \emph{International Journal for Numerical Methods in Fluids},
  47\penalty0 (10-11):\penalty0 1411--1419, 2005.

\bibitem[Ballarin et~al.(2015)Ballarin, Manzoni, Quarteroni, and
  Rozza]{ballarin2015supremizer}
Francesco Ballarin, Andrea Manzoni, Alfio Quarteroni, and Gianluigi Rozza.
\newblock Supremizer stabilization of {POD}--{G}alerkin approximation of
  parametrized steady incompressible {N}avier--{S}tokes equations.
\newblock \emph{International Journal for Numerical Methods in Engineering},
  102\penalty0 (5):\penalty0 1136--1161, 2015.

\bibitem[Dal~Santo and Manzoni(2019)]{dalsanto2019hyper}
Niccol{\`o} Dal~Santo and Andrea Manzoni.
\newblock Hyper-reduced order models for parametrized unsteady
  {N}avier--{S}tokes equations on domains with variable shape.
\newblock \emph{Advances in Computational Mathematics}, 45\penalty0
  (5):\penalty0 2463--2501, 2019.

\bibitem[Dal~Santo et~al.(2019)Dal~Santo, Deparis, Manzoni, and
  Quarteroni]{dalsanto2019algebraic}
Niccol{\`o} Dal~Santo, Simone Deparis, Andrea Manzoni, and Alfio Quarteroni.
\newblock An algebraic least squares reduced basis method for the solution of
  nonaffinely parametrized {S}tokes equations.
\newblock \emph{Computer Methods in Applied Mechanics and Engineering},
  344:\penalty0 186--208, 2019.

\bibitem[Tenderini et~al.(2022)Tenderini, Pagani, Quarteroni, and
  Deparis]{tenderini2022pde}
Riccardo Tenderini, Stefano Pagani, Alfio Quarteroni, and Simone Deparis.
\newblock {PDE}--aware deep learning for inverse problems in cardiac
  electrophysiology.
\newblock \emph{SIAM Journal on Scientific Computing}, 44\penalty0
  (3):\penalty0 B605--B639, 2022.

\bibitem[Abdulle and Bud{\'a}{\v{c}}(2015)]{abdulle2015petrov}
Assyr Abdulle and Ondrej Bud{\'a}{\v{c}}.
\newblock A {P}etrov--{G}alerkin reduced basis approximation of the {S}tokes
  equation in parametrized geometries.
\newblock \emph{Comptes Rendus Mathematique}, 353\penalty0 (7):\penalty0
  641--645, 2015.

\bibitem[Marchandise et~al.(2012)Marchandise, Crosetto, Geuzaine, Remacle, and
  Sauvage]{marchandise2012quality}
Emilie Marchandise, Paolo Crosetto, Christophe Geuzaine, Jean-Fran{\c{c}}ois
  Remacle, and Emilie Sauvage.
\newblock Quality open source mesh generation for cardiovascular flow
  simulations.
\newblock In \emph{Modeling of Physiological Flows}, pages 395--414. Springer,
  2012.

\bibitem[Colciago et~al.(2014)Colciago, Deparis, and
  Quarteroni]{colciago2014comparisons}
Claudia~Maria Colciago, Simone Deparis, and Alfio Quarteroni.
\newblock Comparisons between reduced order models and full 3{D} models for
  fluid--structure interaction problems in haemodynamics.
\newblock \emph{Journal of Computational and Applied Mathematics},
  265:\penalty0 120--138, 2014.

\bibitem[Colciago and Deparis(2018)]{colciago2018reduced}
Claudia~M Colciago and Simone Deparis.
\newblock Reduced numerical approximation of reduced fluid--structure
  interaction problems with applications in hemodynamics.
\newblock \emph{Frontiers in Applied Mathematics and Statistics}, 4:\penalty0
  18, 2018.

\bibitem[Bertagna et~al.(2017)Bertagna, Deparis, Formaggia, Forti, and
  Veneziani]{bertagna2017lifev}
Luca Bertagna, Simone Deparis, Luca Formaggia, Davide Forti, and Alessandro
  Veneziani.
\newblock The {L}ife{V} library: engineering mathematics beyond the proof of
  concept.
\newblock \emph{arXiv preprint arXiv:1710.06596}, 2017.

\bibitem[Figueroa et~al.(2006)Figueroa, Vignon-Clementel, Jansen, Hughes, and
  Taylor]{figueroa2006coupled}
C~Alberto Figueroa, Irene~E Vignon-Clementel, Kenneth~E Jansen, Thomas~JR
  Hughes, and Charles~A Taylor.
\newblock A coupled momentum method for modeling blood flow in
  three--dimensional deformable arteries.
\newblock \emph{Computer methods in applied mechanics and engineering},
  195\penalty0 (41-43):\penalty0 5685--5706, 2006.

\bibitem[Vignon-Clementel et~al.(2006)Vignon-Clementel, Figueroa, Jansen, and
  Taylor]{vignon2006outflow}
Irene~E Vignon-Clementel, C~Alberto Figueroa, Kenneth~E Jansen, and Charles~A
  Taylor.
\newblock Outflow boundary conditions for three--dimensional finite element
  modeling of blood flow and pressure in arteries.
\newblock \emph{Computer methods in applied mechanics and engineering},
  195\penalty0 (29-32):\penalty0 3776--3796, 2006.

\bibitem[Kim et~al.(2010)Kim, Vignon-Clementel, Coogan, Figueroa, Jansen, and
  Taylor]{kim2010patient}
Hyun~Jin Kim, IE~Vignon-Clementel, JS~Coogan, CA~Figueroa, KE~Jansen, and
  CA~Taylor.
\newblock Patient--specific modeling of blood flow and pressure in human
  coronary arteries.
\newblock \emph{Annals of biomedical engineering}, 38\penalty0 (10):\penalty0
  3195--3209, 2010.

\end{thebibliography}

\newpage
\appendix

\section{\textit{ST--GRB} method: details of the assembling phase} \label{appendix: ST--RB assembling}

\justifying

In this appendix, we provide a detailed explanation of the assembling of the left--hand side matrix $\widehat{\bm{A}}^{st}$ and of the right--hand side vector $\widehat{\bm{F}}^{st}$ in Eq.\eqref{eq: compact space time reduced system}, in the case of a Galerkin projection (i.e. $\bm{\Pi} = \mytilde{\bm{\Pi}}$).

Firstly, let us consider the left--hand side matrix $\widehat{\bm{A}}^{st}$. In order to compute its five parameter--independent non--zero blocks (see Eq.\eqref{eq: reduced A}), we can exploit the properties of the spatio--temporal reduced basis $\bm{\Pi}$.  For instance, let us compute the value of $(\widehat{\bm{A}}^{st}_1)_{\ell m}$ for a couple of indexes $(\ell m) \in \{1,\cdots,n^{st}_u\} \times \{1,\cdots,n^{st}_u\}$, such that $\ell = \mathcal{F}_u(\ell_s, \ell_t)$ and $m = \mathcal{F}_u(m_s, m_t)$, with $\ell_s,m_s \in \{1,\cdots n_u^s\}$, $\ell_t,m_t \in \{1,\cdots n_u^t\}$. We have that
\begin{equation}
\label{eq: reduced A1}
\begin{split}
&(\widehat{\bm{A}}^{st}_1)_{\ell m} = (\bm{\pi}^u_\ell)^T\bm{A}^{st}_1\bm{\pi}^u_m = \left(\bm{\Phi}_{\ell_s}^u \otimes \bm{\Psi}_{\ell_t}^u\right)^T \bm{A}_4^{st} \left(\bm{\Phi}_{m_s}^u \otimes \bm{\Psi}_{m_t}^u\right)
\\ &=  
\left(\bm{\Phi}_{\ell_s}^u \otimes \bm{\Psi}_{\ell_t}^u\right)^T
\begingroup 
\setlength\arraycolsep{3pt}
\begin{bmatrix*}[c] 
& &+ (\bm{M} + \frac23\delta\bm{A})\bm{\phi}^u_{m_s}(\bm{\psi}^u_{m_t})_1 \\
&- \frac43\bm{M}\bm{\phi}^u_{m_s}(\bm{\psi}^u_{m_t})_1 
&+ (\bm{M} + \frac23\delta\bm{A})\bm{\phi}^u_{m_s}(\bm{\psi}^u_{m_t})_2 \\ 
\frac13\bm{M}\bm{\phi}^u_{m_s}(\bm{\psi}^u_{m_t})_1 
&- \frac43\bm{M}\bm{\phi}^u_{m_s}(\bm{\psi}^u_{m_t})_2 
&+ (\bm{M} + \frac23\delta\bm{A})\bm{\phi}^u_{m_s}(\bm{\psi}^u_{m_t})_3  \\
\vdots & \vdots & \vdots \\
\frac13\bm{M}\bm{\phi}^u_{m_s}(\bm{\psi}^u_{m_t})_{N^t-2} 
&- \frac43\bm{M}\bm{\phi}_{m_s}(\bm{\psi}^u_{m_t})_{N^t-1}
&+ (\bm{M} + \frac23\delta\bm{A}) \bm{\phi}^u_{m_s}(\bm{\psi}_{m_t})_{N^t}
\end{bmatrix*}
\endgroup
\\
&=
\left(\widehat{\bm{M}}+\frac23\delta\widehat{\bm{A}}\right)_{\ell_s m_s}\delta_{\ell_t,m_t}
-\dfrac{4}{3}\widehat{\bm{M}}_{\ell_s m_s}(\bm{\psi}^u_{\ell_t})_{2:}^T(\bm{\psi}^u_{m_t})_{:-1}
+\dfrac{1}{3}\widehat{\bm{M}}_{\ell_s m_s}(\bm{\psi}^u_{\ell_t})_{3:}^T(\bm{\psi}^u_{m_t})_{:-2}~,
\end{split}
\end{equation}
where $\widehat{\bm{M}}$ and $\widehat{\bm{A}} = (\bm{\Phi}^u)^T\bm{A}\bm{\Phi}^u$ are the space--reduced mass and stiffness matrices, respectively (see Eq.\eqref{eq: ROM matrices}). Furthermore, $\delta_{i,j}$ is the Kronecker Delta function, which appears in Eq.\eqref{eq: reduced A1} since the columns of $\bm{\Psi}^u$ are orthonormal in Euclidean norm by construction, and the notations $\bm{v}_{i:}$, $\bm{v}_{:-j}$ denote the sub--vectors of a given vector $\bm{v}$ containing all the entries from the $i$--th to the last one and from the first one to the $j$--th from last, respectively. Similar expressions characterize the other four parameter--independent non--zero blocks of $\widehat{\bm{A}}^{st}$. Thus, the assembly of the parameter--independent part of the reduced left--hand side matrix only requires the spatial projection of the FOM matrices and some trivial multiplications between elements of the reduced bases in time. In particular, recalling the definition of the space--reduced matrices in Eq.\eqref{eq: ROM matrices} and of the time--reduced matrices in Eq.\eqref{eq: temporal bases combined}, the parameter--independent blocks of $\widehat{\bm{A}}^{st}$ (other than $\widehat{\bm{A}}^{st}_1$, already defined in Eq.\eqref{eq: reduced A1}) read as:
\begin{equation}
\label{eq: reduced blocks A}
\begin{alignedat}{2}
\widehat{\bm{A}}^{st}_2 \in \mathbb{R}^{n_u^{st} \times n_p^{st}}: \ \left(\widehat{\bm{A}}^{st}_2\right)_{\ell k} &= \frac23 \delta \widehat{\bm{B}}^T_{\ell_s k_s} \bm{\Psi}_{\ell_t k_t}^{u,p} 
\qquad 
\widehat{\bm{A}}^{st}_4 \in \mathbb{R}^{n_p^{st} \times n_u^{st}}: \ \left(\widehat{\bm{A}}^{st}_4\right)_{k m} &= \widehat{\bm{B}}_{k_s m_s} \bm{\Psi}_{m_t k_t}^{u,p}
\\
\widehat{\bm{A}}^{st}_3 \in \mathbb{R}^{n_u^{st} \times n_\lambda^{st}}: \ \left(\widehat{\bm{A}}^{st}_3\right)_{\ell j} &= \frac23 \delta \widehat{\bm{C}}^T_{\ell_s j_s} \bm{\Psi}_{\ell_t j_t}^{u,\lambda}
\qquad 
\widehat{\bm{A}}^{st}_7 \in \mathbb{R}^{n_\lambda^{st} \times n_u^{st}}: \ \left(\widehat{\bm{A}}^{st}_7\right)_{j m} &= \widehat{\bm{C}}_{j_s m_s} \bm{\Psi}_{m_t j_t}^{u,\lambda}
\end{alignedat}
\end{equation}
for $\ell,m$ defined as in Eq.\eqref{eq: reduced A1}; $k = \mathcal{F}_p(k_s, k_t)$ with $k_s \in \{1, \dots, n_p^s\}$, $k_t \in \{1, \dots, n_p^t\}$; $j \in \mathcal{F}_\lambda(j_s, j_t)$ with $j_s \in \{1, \dots, N_\lambda\}, j_t \in \{1, \dots, _\lambda^t\}$.

Since $\mytilde{\bm{\Pi}} = \bm{\Pi}$ (so parameter--independent), the assembling of the parameter--dependent part of the left--hand side matrix only involves the computation of the velocity--velocity block $\widehat{\bm{R}}_1^{st}(\bm{\mu})$. To this aim, we can exploit the affine parametrization of the reaction term (see Eq.\eqref{eq: ROM matrix reaction ST}). Indeed, we have that
\begin{equation}
\widehat{\bm{R}}_1^{st}(\bm{\mu}) = \sum_{q=1}^{N_c} \rho_c^q (\bm{\mu}) \widehat{\bm{R}}_q^{st} \qquad \text{with} \qquad \widehat{\bm{R}}_q^{st} \in \mathbb{R}^{n_u^{st} \times n_u^{st}} \ : \ \left(\widehat{\bm{R}}_q^{st}\right)_{\ell m} = \frac23 \delta  \left(\widehat{\bm{R}}^q\right)_{\ell_s m_s} \delta_{\ell_t,m_t}~,
\end{equation}
being $\ell,m$ defined as in Eq.\eqref{eq: reduced A1} and $\{\widehat{\bm{R}}^q\}_{q=1}^{N_c}$ the space--reduced affine components of the matrix associated to the reaction term (see Eq.\eqref{eq: ROM matrix reaction}). Hence, the online computational cost of the left--hand side matrix assembling is small, since it only involves the linear combination of $N^c+1$ space--time reduced matrices. In the general case, where the left--hand side term does not feature affine parametric dependency, approximate affine decompositions can be retrieved exploiting the MDEIM algorithm. 

Finally, let us focus on the assembling of the right--hand side vector $\widehat{\bm{F}}^{st}$. Based on Eq.\eqref{eq: monolitic_FOM_blocks_vector}, we only have to compute its third block $\widehat{\bm{F}}^{st}_3(\bm{\mu}^*) = (\bm{\Pi}^\lambda)^T\bm{F}^{st}_3(\bm{\mu}^*) \in \mathbb{R}^{n^{st}_\lambda}$, as all the other ones are null. To this end, we can leverage the space--time factorization of the Dirichlet datum (see Eq.\eqref{eq: factorization_dirichlet_datum}). Indeed,  the $k$--th block of the space--time reduced right--hand side, $\widehat{\bm{F}}^{st}_{3,k}(\bm{\mu}^*)$ ($k \in \{1, \dots, N_D\}$), is given by
\begin{equation}
\label{eq: reduced RHS local}
\widehat{\bm{F}}^{st}_{3,k} = \widehat{\bm{F}}^{st}_{3,k}(\bm{\mu}^*) \in \mathbb{R}^{n^{st}_{\lambda_k}}: \ \left(\bm{\widehat{F}}^{st}_{3,k}\right)_{\mathcal{F}_{\lambda_k}(i,j)} = \left(\tilde{\bm{g}}_k^s\right)_i \left(\left(\bm{\psi}_j^{\lambda_k}\right)^T \bm{g}^t(\bm{\mu}^*)\right)~.
\end{equation}
Since the global space of Lagrange multipliers $\mathcal{L}$ is such that $\mathcal{L} = \prod_{k=1}^{N_D} \mathcal{L}_k$, we have that
\begin{equation}
\label{eq: reduced RHS global}
\widehat{\bm{F}}^{st}_{3}(\bm{\mu}^*) = \left[\left(\widehat{\bm{F}}^{st}_{3,1}(\bm{\mu}^*)\right)^T, \cdots, \left(\widehat{\bm{F}}^{st}_{3,N_D}(\bm{\mu}^*)\right)^T\right]^T \in \mathbb{R}^{n^{st}_{\lambda}}~.
\end{equation}
Notice that this assembling step is extremely cheap, as it only involves $N_D$ inner products between $N^t$--dimensional vectors.

\section{\textit{ST--PGRB} method: details of the assembling phase} \label{appendix: ST--PGRB assembling}

In this appendix, we show how the left--hand side matrix $\widehat{\bm{A}}^{pg}$ and the right--hand side vector $\widehat{\bm{F}}^{pg}$ of the linear system arising from the application of the ST--PGRB method (see Eq.\eqref{eq: linear system ST--PGRB}) can be efficiently computed, leveraging the block structure of the problem at hand and the affine parametrization of its left--hand side term. Firstly, we define the diagonal preconditioners of the spatio--temporal norm matrices for velocity, pressure and Lagrange multipliers as $\bm{P}_{\bm{X}_u}^{st}$, $\bm{P}_{\bm{X}_p}^{st}$, $\bm{P}_{\bm{X}_\lambda}^{st}$, respectively. In addition, we define the diagonal preconditioners of the spatial norm matrices for velocity, pressure and Lagrange multipliers as $\bm{P}_{\bm{X}_u}$, $\bm{P}_{\bm{X}_p}$, $\bm{P}_{\bm{X}_\lambda}$, respectively. We recall that $\bm{X}_\lambda = \bm{P}_{\bm{X}_\lambda} = \bm{I}_{N_\lambda}$. 

Let us consider the left--hand side matrix $\widehat{\bm{A}}^{pg}$. Exploiting the block structure of the FOM left--hand side matrix $\bm{A}^{st}$ and differentiating between parameter--independent and parameter--dependent components (see Eq.\eqref{eq: monolitic_FOM_system}), we have that
\begin{equation*}
\label{eq: ST--PGRB blocks structure}
\widehat{\bm{A}}^{pg} = 
\begin{bmatrix}
\widehat{\bm{A}}^{pg}_{1,1} + \widehat{\bm{A}}^{pg}_{4,4} + \widehat{\bm{A}}^{pg}_{7,7} & \widehat{\bm{A}}^{pg}_{1,2} & \widehat{\bm{A}}^{pg}_{1,3} \\
{\widehat{\bm{A}}^{pg}_{1,2}\phantom{}}^T & \widehat{\bm{A}}^{pg}_{2,2} & \widehat{\bm{A}}^{pg}_{2,3} \\
{\widehat{\bm{A}}^{pg}_{1,3}\phantom{}}^T & {\widehat{\bm{A}}^{pg}_{2,3}\phantom{}}^T & \widehat{\bm{A}}^{pg}_{3,3}
\end{bmatrix}
+
\begin{bmatrix}
\bm{\widehat{R}}^{pg}_{1}(\bm{\mu}) + {\bm{\widehat{R}}^{pg}_{1}(\bm{\mu})}^T + \bm{\widehat{R}}^{pg}(\bm{\mu}) & \bm{\widehat{R}}^{pg}_{2}(\bm{\mu}) & \bm{\widehat{R}}^{pg}_{3}(\bm{\mu}) \\
{\bm{\widehat{R}}^{pg}_{2}(\bm{\mu})}^T & & \\
{\bm{\widehat{R}}^{pg}_{3}(\bm{\mu})}^T & &
\end{bmatrix}~.
\end{equation*}
The parameter--independent blocks have the following expressions:
\begin{equation}
\label{eq: ST--PGRB matrix blocks}
\begin{aligned}
\widehat{\bm{A}}^{pg}_{1,1} &= \left(\bm{A}^{st}_1\bm{\Pi}^u\right)^T\left(\bm{P}_{\bm{X}_u}^{st}\right)^{-1}\left(\bm{A}^{st}_1\bm{\Pi}^u\right) \qquad 
\widehat{\bm{A}}^{pg}_{4,4} &= \left(\bm{A}^{st}_4\bm{\Pi}^u\right)^T(\bm{P}_{\bm{X}_p}^{st})^{-1}\left(\bm{A}^{st}_4\bm{\Pi}^u\right) \\
\widehat{\bm{A}}^{pg}_{7,7} &= \left(\bm{A}^{st}_7\bm{\Pi}^u\right)^T\left(\bm{P}_{\bm{X}_\lambda}^{st}\right)^{-1}\left(\bm{A}^{st}_7\bm{\Pi}^u\right) \qquad 
\widehat{\bm{A}}^{pg}_{1,2} &= \left(\bm{A}^{st}_1\bm{\Pi}^u\right)^T\left(\bm{P}_{\bm{X}_u}^{st}\right)^{-1}\left(\bm{A}^{st}_2\bm{\Pi}^p\right) \\
\widehat{\bm{A}}^{pg}_{1,3} &= \left(\bm{A}^{st}_1\bm{\Pi}^u\right)^T\left(\bm{P}_{\bm{X}_u}^{st}\right)^{-1}\left(\bm{A}^{st}_3\bm{\Pi}^\lambda\right) \qquad
\widehat{\bm{A}}^{pg}_{2,2} &= \left(\bm{A}^{st}_2\bm{\Pi}^p\right)^T\left(\bm{P}_{\bm{X}_u}^{st}\right)^{-1}\left(\bm{A}^{st}_2\bm{\Pi}^p\right) \\
\widehat{\bm{A}}^{pg}_{3,3} &= \left(\bm{A}^{st}_3\bm{\Pi}^\lambda\right)^T\left(\bm{P}_{\bm{X}_u}^{st}\right)^{-1}\left(\bm{A}^{st}_3\bm{\Pi}^\lambda\right) \qquad
\widehat{\bm{A}}^{pg}_{2,3} &= \left(\bm{A}^{st}_2\bm{\Pi}^p\right)^T\left(\bm{P}_{\bm{X}_u}^{st}\right)^{-1}\left(\bm{A}^{st}_3\bm{\Pi}^\lambda\right)
\end{aligned}
\end{equation}
while the parameter--dependent ones write as:
\begin{equation}
\label{eq: ST--PGRB param matrix blocks}
\begin{alignedat}{3}
\widehat{\bm{R}}^{pg}_1 &=
\left(\bm{R}^{st}(\bm{\mu})\bm{\Pi}^u\right)^T \left(\bm{P}_{\bm{X}_u}^{st}\right)^{-1} \left(\bm{A}^{st}_1\bm{\Pi}^u\right) \qquad 
\widehat{\bm{R}}^{pg} &&=
\left(\bm{R}^{st}(\bm{\mu})\bm{\Pi}^u\right)^T \left(\bm{P}_{\bm{X}_u}^{st}\right)^{-1} \left(\bm{R}^{st}(\bm{\mu})\bm{\Pi}^u\right) \\
\widehat{\bm{R}}^{pg}_2 &=
\left(\bm{R}^{st}(\bm{\mu})\bm{\Pi}^u\right)^T \left(\bm{P}_{\bm{X}_u}^{st}\right)^{-1} \left(\bm{A}^{st}_2\bm{\Pi}^p\right) \qquad 
\widehat{\bm{R}}^{pg}_3 &&=
\left(\bm{R}^{st}(\bm{\mu})\bm{\Pi}^u\right)^T \left(\bm{P}_{\bm{X}_u}^{st}\right)^{-1} \left(\bm{A}^{st}_3\bm{\Pi}^\lambda\right)
\end{alignedat}
\end{equation}

Firstly, we focus on the parameter--independent blocks in Eq.\eqref{eq: ST--PGRB matrix blocks}, that can be assembled once and for all during the offline phase of the method. Let us define the following matrices:
\begin{equation}
\label{eq: half-reduced matrices}
\begin{alignedat}{4}
\overline{\bm{A}} &= \bm{A}\bm{\Phi}^u \ \in \mathbb{R}^{N_u^s \times n_u^s} \quad &
\overline{\bm{B}}^T &= \bm{B}^T\bm{\Phi}^p  \in \mathbb{R}^{N_u^s \times n_p^s}
\quad & 
\overline{\bm{B}} &= \bm{B}\bm{\Phi}^u  \in \mathbb{R}^{N_p^s \times n_u^s} 
\\
\overline{\bm{M}} &= \bm{M}\bm{\Phi}^u  \in \mathbb{R}^{N_u^s \times n_u^s}
\quad &
\overline{\bm{C}}^T &= \bm{C}^T \quad \ \in \mathbb{R}^{N_u^s \times N_\lambda}
\quad &
\overline{\bm{C}} &= \bm{C}\bm{\Phi}^u  \in \mathbb{R}^{N_\lambda \times n_u^s}
\end{alignedat}
\end{equation}
Also, let us define the matrix $\overline{\bm{MA}} := \overline{\bm{M}} + \frac23 \delta \overline{\bm{A}}$. Let us define the following indexes:
\begin{itemize}
	\item $\ell = \mathcal{F}_u(\ell_s, \ell_t)$, $m = \mathcal{F}_u(m_s, m_t)$ with $\ell_s, m_s \in \{1, \dots, n_u^s\}$, $\ell_t, m_t \in \{1, \dots, n_u^t\}$;
	\item $k = \mathcal{F}_p(k_s, k_t)$, $r = \mathcal{F}_p(r_s, r_t)$ with $k_s, r_s \in \{1, \dots, n_p^s\}$, $k_t, r_t \in \{1, \dots, n_p^t\}$;
	\item $i = \mathcal{F}_\lambda(i_s, i_t)$, $j = \mathcal{F}_\lambda(j_s, j_t)$ with $i_s, j_s \in \{1, \dots, N_\lambda\}$, $i_t, j_t \in \{1, \dots, n_\lambda^t\}$.
\end{itemize}
As in Eq.\eqref{eq: reduced A1}, the notations $\bm{v}_{i:}$, $\bm{v}_{:-j}, \bm{v}_{i:-j}$ denote the sub--vectors of a given vector $\bm{v}$ containing all the entries from the $i$--th to the last one, from the first one to the $j$--th from last and from the $i$--th one to the $j$--th from last, respectively. Finally, for a given matrix $\bm{Q} \in \mathbb{R}^{N_1 \times N_2}$, we use the notations $\bm{Q}_c$ or $\bm{Q}_{:, c}$ (with $c \in \{1, \dots, N_2\}$) to denote the $c$--th column of $\bm{Q}$ and $\bm{Q}_{r, :}$ (with $r \in \{1, \dots, N_1\}$) to denote the $r$--th row of $\bm{Q}$. Then, we have that:
\begin{fleqn}
	\begin{equation*}
	\label{eq: ST--PGRB A11}
	\begin{split}
	\left(\widehat{\bm{A}}^{pg}_{1,1}\right)_{\ell m} &= 
	\left(\left(\bm{A}^{st}_1\bm{\Pi}^u\right)^T\left(\bm{P}_{\bm{X}_u}^{st}\right)^{-1}\left(\bm{A}^{st}_1\bm{\Pi}^u\right)\right)_{\ell m} = 
	\left(\bm{A}^{st}_1\bm{\Pi}^u\right)^T_{\ell,:}\left(\bm{P}_{\bm{X}_u}^{st}\right)^{-1}\left(\bm{A}^{st}_1\bm{\Pi}^u\right)_{:,m} = \\
	& =
	\left(\bm{\overline{MA}}_{\ell_s}^T(\bm{P}_{\bm{X}_u})^{-1}\bm{\overline{MA}}_{m_s}\right) \delta_{\ell_t,m_t} \\
	& \quad + 
	\left(\bm{\overline{M}}_{\ell_s}^T(\bm{P}_{\bm{X}_u})^{-1}\bm{\overline{M}}_{m_s}\right)
	\left( \dfrac{16}{9}\left(\bm{\psi}^u_{\ell_t}\right)_{:-1}\left(\bm{\psi}^u_{m_t}\right)_{:-1}
	+ \dfrac{1}{9}\left(\bm{\psi}^u_{\ell_t}\right)_{:-2}\left(\bm{\psi}^u_{m_t}\right)_{:-2} \right. \\
	& \qquad \qquad \qquad \qquad \qquad \qquad \left. - \dfrac{4}{9}\left(\bm{\psi}^u_{\ell_t}\right)_{2:-1}\left(\bm{\psi}^u_{m_t}\right)_{:-2}
	- \dfrac{4}{9}\left(\bm{\psi}^u_{\ell_t}\right)_{:-2}\left(\bm{\psi}^u_{m_t}\right)_{2:-1} 
	\right) \\
	& \quad + \left(\bm{\overline{MA}}_{\ell_s}^T(\bm{P}_{\bm{X}_u})^{-1}\bm{\overline{M}}_{m_s}\right)
	\left(
	- \dfrac{4}{3}\left(\bm{\psi}^u_{\ell_t}\right)_{2:}\left(\bm{\psi}^u_{m_t}\right)_{:-1} + \dfrac{1}{3}\left(\bm{\psi}^u_{\ell_t}\right)_{3:}\left(\bm{\psi}^u_{m_t}\right)_{:-2}
	\right)\\
	& \quad + \left(\bm{\overline{M}}_{\ell_s}^T(\bm{P}_{\bm{X}_u})^{-1}\bm{\overline{MA}}_{m_s}\right)
	\left(
	- \dfrac{4}{3}\left(\bm{\psi}^u_{\ell_t}\right)_{:-1}\left(\bm{\psi}^u_{m_t}\right)_{2:} + \dfrac{1}{3}\left(\bm{\psi}^u_{\ell_t}\right)_{:-2}\left(\bm{\psi}^u_{m_t}\right)_{3:}
	\right)
	\end{split}
	\end{equation*}
	
	\begin{equation*}
	\label{eq: ST--PGRB A44}
	\begin{split}
	\left(\widehat{\bm{A}}^{pg}_{4,4}\right)_{\ell m} &= 
	\left(\left(\bm{A}^{st}_4\bm{\Pi}^u\right)^T\left(\bm{P}_{\bm{X}_p}^{st}\right)^{-1}\left(\bm{A}^{st}_4\bm{\Pi}^u\right)\right)_{\ell m} = 
	\left(\bm{A}^{st}_4\bm{\Pi}^u\right)^T_{\ell,:}\left(\bm{P}_{\bm{X}_p}^{st}\right)^{-1}\left(\bm{A}^{st}_4\bm{\Pi}^u\right)_{:,m} = \\
	& =
	\left(\bm{\overline{B}}_{\ell_s}^T(\bm{P}_{\bm{X}_p})^{-1}\bm{\overline{B}}_{m_s}\right) \delta_{\ell_t,m_t} 
	\end{split}
	\end{equation*}
	
	\begin{equation*}
	\label{eq: ST--PGRB A77}
	\begin{split}
	\left(\widehat{\bm{A}}^{pg}_{7,7}\right)_{\ell m} &= 
	\left(\left(\bm{A}^{st}_7\bm{\Pi}^u\right)^T\left(\bm{P}_{\bm{X}_\lambda}^{st}\right)^{-1}\left(\bm{A}^{st}_7\bm{\Pi}^u\right)\right)_{\ell m} = 
	\left(\bm{A}^{st}_7\bm{\Pi}^u\right)^T_{\ell,:}\left(\bm{P}_{\bm{X}_\lambda}^{st}\right)^{-1}\left(\bm{A}^{st}_7\bm{\Pi}^u\right)_{:,m} = \\
	& =
	\left(\bm{\overline{C}}_{\ell_s}^T\bm{\overline{C}}_{m_s}\right) \delta_{\ell_t,m_t} 
	\end{split}
	\end{equation*}
	
	\begin{equation*}
	\label{eq: ST--PGRB A12}
	\begin{split}
	\left(\widehat{\bm{A}}^{pg}_{1,2}\right)_{\ell k} &= 
	\left(\left(\bm{A}^{st}_1\bm{\Pi}^u\right)^T\left(\bm{P}_{\bm{X}_u}^{st}\right)^{-1}\left(\bm{A}^{st}_2\bm{\Pi}^p\right)\right)_{\ell k} = 
	\left(\bm{A}^{st}_1\bm{\Pi}^u\right)^T_{\ell,:}\left(\bm{P}_{\bm{X}_u}^{st}\right)^{-1}\left(\bm{A}^{st}_2\bm{\Pi}^p\right)_{:,k} = \\
	& =
	\dfrac23 \delta 
	\left(\bm{\overline{MA}}_{\ell_s}^T \left(\bm{P}_{\bm{X}_u}\right)^{-1} \bm{\overline{B}}^T_{k_s} \left(\bm{\psi}_{\ell_t}^u\right)^T\left(\bm{\psi}_{k_t}^p\right) - 
	\dfrac43 \bm{\overline{M}}_{\ell_s}^T \left(\bm{P}_{\bm{X}_u}\right)^{-1} \bm{\overline{B}}^T_{k_s} \left(\bm{\psi}_{\ell_t}^u\right)_{:-1} \left(\bm{\psi}_{k_t}^p\right)_{2:} \right. \\
	& \qquad \quad \ \left.
	+ \dfrac13 \bm{\overline{M}}_{\ell_s}^T \left(\bm{P}_{\bm{X}_u}\right)^{-1} \bm{\overline{B}}^T_{k_s} \left(\bm{\psi}_{\ell_t}^u\right)_{:-2} \left(\bm{\psi}_{k_t}^p\right)_{3:}
	\right)
	\end{split}
	\end{equation*}
	
	\begin{equation*}
	\label{eq: ST--PGRB A13}
	\begin{split}
	\left(\widehat{\bm{A}}^{pg}_{1,3}\right)_{\ell j} &= 
	\left(\left(\bm{A}^{st}_1\bm{\Pi}^u\right)^T\left(\bm{P}_{\bm{X}_u}^{st}\right)^{-1}\left(\bm{A}^{st}_3\bm{\Pi}^\lambda\right)\right)_{\ell j} = 
	\left(\bm{A}^{st}_1\bm{\Pi}^u\right)^T_{\ell,:}\left(\bm{P}_{\bm{X}_u}^{st}\right)^{-1}\left(\bm{A}^{st}_3\bm{\Pi}^\lambda\right)_{:,j} = \\
	& =
	\dfrac23 \delta 
	\left(\bm{\overline{MA}}_{\ell_s}^T \left(\bm{P}_{\bm{X}_u}\right)^{-1} \bm{\overline{C}}^T_{j_s} \left(\bm{\psi}_{\ell_t}^u\right)^T\left(\bm{\psi}_{j_t}^\lambda\right) - 
	\dfrac43 \bm{\overline{M}}_{\ell_s}^T \left(\bm{P}_{\bm{X}_u}\right)^{-1} \bm{\overline{C}}^T_{j_s} \left(\bm{\psi}_{\ell_t}^u\right)_{:-1} \left(\bm{\psi}_{j_t}^\lambda\right)_{2:} \right. \\
	& \qquad \quad \ \left.
	+ \dfrac13 \bm{\overline{M}}_{\ell_s}^T \left(\bm{P}_{\bm{X}_u}\right)^{-1} \bm{\overline{C}}^T_{j_s} \left(\bm{\psi}_{\ell_t}^u\right)_{:-2} \left(\bm{\psi}_{j_t}^\lambda\right)_{3:}
	\right)
	\end{split}
	\end{equation*}
	
	\begin{equation*}
	\label{eq: ST--PGRB A22}
	\begin{split}
	\left(\widehat{\bm{A}}^{pg}_{2,2}\right)_{k r} &= 
	\left(\left(\bm{A}^{st}_2\bm{\Pi}^p\right)^T\left(\bm{P}_{\bm{X}_u}^{st}\right)^{-1}\left(\bm{A}^{st}_2\bm{\Pi}^p\right)\right)_{k r} = 
	\left(\bm{A}^{st}_2\bm{\Pi}^p\right)^T_{k,:}\left(\bm{P}_{\bm{X}_u}^{st}\right)^{-1}\left(\bm{A}^{st}_2\bm{\Pi}^p\right)_{:,r} = \\
	& =
	\dfrac49 \delta^2 \left(\left(\bm{\overline{B}}_{k_s}^T\right)^T(\bm{P}_{\bm{X}_p})^{-1}\bm{\overline{B}}_{r_s}^T\right) \delta_{k_t,r_t} 
	\end{split}
	\end{equation*}
	
	\begin{equation*}
	\label{eq: ST--PGRB A33}
	\begin{split}
	\left(\widehat{\bm{A}}^{pg}_{3,3}\right)_{j i} &= 
	\left(\left(\bm{A}^{st}_3\bm{\Pi}^\lambda\right)^T\left(\bm{P}_{\bm{X}_u}^{st}\right)^{-1}\left(\bm{A}^{st}_3\bm{\Pi}^\lambda\right)\right)_{j i} = 
	\left(\bm{A}^{st}_3\bm{\Pi}^\lambda\right)^T_{j,:}\left(\bm{P}_{\bm{X}_u}^{st}\right)^{-1}\left(\bm{A}^{st}_3\bm{\Pi}^\lambda\right)_{:,i} = \\
	& =
	\dfrac49 \delta^2 \left(\left(\bm{\overline{C}}_{j_s}^T\right)^T(\bm{P}_{\bm{X}_p})^{-1}\bm{\overline{C}}_{i_s}^T\right) \delta_{j_t,i_t} 
	\end{split}
	\end{equation*}
	
	\begin{equation*}
	\label{eq: ST--PGRB A23}
	\begin{split}
	\left(\widehat{\bm{A}}^{pg}_{2,3}\right)_{k j} &= 
	\left(\left(\bm{A}^{st}_2\bm{\Pi}^p\right)^T\left(\bm{P}_{\bm{X}_u}^{st}\right)^{-1}\left(\bm{A}^{st}_3\bm{\Pi}^\lambda\right)\right)_{k j} = 
	\left(\bm{A}^{st}_2\bm{\Pi}^p\right)^T_{k,:}\left(\bm{P}_{\bm{X}_u}^{st}\right)^{-1}\left(\bm{A}^{st}_3\bm{\Pi}^\lambda\right)_{:,j} = \\
	& =
	\dfrac49 \delta^2 \left(\left(\bm{\overline{B}}_{k_s}^T\right)^T(\bm{P}_{\bm{X}_p})^{-1}\bm{\overline{C}}_{j_s}^T\right)\left(\left(\bm{\psi}_{k_t}^p\right)^T \bm{\psi}_{j_t}^{\lambda}\right)
	\end{split}
	\end{equation*}	
\end{fleqn}

Let us now focus on the parameter--dependent blocks of the left--hand side matrix. Let us define the matrices
\begin{equation}
\overline{\bm{R}}^q = \bm{R}^q \bm{\Phi}^u \in \mathbb{R}^{N_u^s \times n_u^s} \qquad \text{with} \quad q \in \{1, \dots, N_c\}~.
\end{equation}
For $q,q' \in \{1, \dots N_c\}$ and considering the indexes $\ell, m$ as in Eq.\eqref{eq: reduced A1}, the following space--time--reduced matrices can be assembled offline:
\begin{fleqn}
	\begin{align*}
	\widehat{\bm{R}}_{1,q}^{pg} \in \mathbb{R}^{n_u^{st} \times n_u^{st}} \ & : \quad \left(\widehat{\bm{R}}_{1,q}^{pg}\right)_{\ell m} &&= \left(\left(\overline{\bm{R}}^q_{\ell_s}\right)^T \left(\bm{P}_{\bm{X}_u}\right)^{-1} \overline{\bm{MA}}_{m_s} \right) \delta_{\ell_t,m_t}   \\
	& && \quad + 
	\left(\left(\overline{\bm{R}}^q_{\ell_s}\right)^T \left(\bm{P}_{\bm{X}_u}\right)^{-1} \overline{\bm{M}}_{m_s}\right)
	\left( -\frac43 (\bm{\psi}^u_{\ell_t})_{2:}^T (\bm{\psi}^u_{m_t})_{:-1} \right. \\
	& && \qquad \qquad \qquad \qquad \qquad \qquad \qquad 
	\left. + \frac13 (\bm{\psi}^u_{\ell_t})_{3:}^T (\bm{\psi}^u_{m_t})_{:-2}\right) 
	\\
	\widehat{\bm{R}}_{1^*,(q,q')}^{pg} \in \mathbb{R}^{n_u^{st} \times n_u^{st}} \ & : \quad \left(\widehat{\bm{R}}_{1^*,(q,q')}^{pg}\right)_{\ell m} &&= \left(\left(\overline{\bm{R}}^{q}_{\ell_s}\right)^T \left(\bm{P}_{\bm{X}_u}\right)^{-1} \overline{\bm{R}}^{q'}_{m_s} \right) \delta_{\ell_t,m_t}   
	\\
	\widehat{\bm{R}}_{2,q}^{pg} \in \mathbb{R}^{n_u^{st} \times n_u^{st}} \ & : \quad \left(\widehat{\bm{R}}_{2,q}^{pg}\right)_{\ell m} &&= \left(\left(\overline{\bm{R}}^q_{\ell_s}\right)^T \left(\bm{P}_{\bm{X}_u}\right)^{-1} \overline{\bm{B}}^T_{m_s} \right) \left(\left(\bm{\psi}^u_{\ell_t}\right)^T \left(\bm{\psi}^p_{m_t}\right)\right)   
	\\
	\widehat{\bm{R}}_{3,q}^{pg} \in \mathbb{R}^{n_u^{st} \times n_u^{st}} \ & : \quad \left(\widehat{\bm{R}}_{3,q}^{pg}\right)_{\ell m} &&= \left(\left(\overline{\bm{R}}^q_{\ell_s}\right)^T \left(\bm{P}_{\bm{X}_u}\right)^{-1} \overline{\bm{C}}^T_{m_s} \right) \left(\left(\bm{\psi}^u_{\ell_t}\right)^T \left(\bm{\psi}^\lambda_{m_t}\right)\right)   \\
	\end{align*}
\end{fleqn}
Exploiting the affine parametrization of the reaction term, the left--hand side parameter--dependent blocks can be then assembled online as follows:
\begingroup
\allowdisplaybreaks
\begin{equation}
\begin{aligned}
\label{eq: ST--PGRB param blocks}
\widehat{\bm{R}}^{pg}_1(\bm{\mu}) &= \frac23 \delta \ \sum_{q=1}^{N_c} \rho_c^q(\bm{\mu}) \widehat{\bm{R}}^{pg}_{1,q}
\qquad 
&&\widehat{\bm{R}}^{pg}_{1^*}(\bm{\mu}) = \frac49 \delta^2 \ \sum_{q=1}^{N_c} \sum_{q'=1}^{N_c} \rho_c^q(\bm{\mu}) \rho_c^{q'}(\bm{\mu}) \widehat{\bm{R}}^{pg}_{1^*,(q,q')}
\\
\widehat{\bm{R}}^{pg}_2(\bm{\mu}) &= \frac23 \delta \ \sum_{q=1}^{N_c} \rho_c^q(\bm{\mu}) \widehat{\bm{R}}^{pg}_{2,q}
\qquad
&&\widehat{\bm{R}}^{pg}_3(\bm{\mu}) = \frac23 \delta \ \sum_{q=1}^{N_c} \rho_c^q(\bm{\mu}) \widehat{\bm{R}}^{pg}_{3,q}
\end{aligned}
\end{equation}
\endgroup
Finally, let us consider the right--hand side vector $\widehat{\bm{F}}^{pg}$. Exploiting the block structure of the FOM matrix $\bm{A}^{st}$ and of the FOM right--hand side vector $\bm{F}^{st}$ (see Eq.\eqref{eq: monolitic_FOM_system}), we have that
\begin{equation}
\label{apx9}
\begin{split}
\widehat{\bm{F}}^{pg} &= 
\left(\bm{A}^{st} \bm{\Pi}\right)^T \left(\bm{P}_{\bm{X}}^{st}\right)^{-1} \bm{F}^{st} \\
&=
\begingroup 
\setlength\arraycolsep{1pt}
\begin{bmatrix} 
((\bm{A}^{st}_1 + \bm{R}^{st}(\bm{\mu}))\bm{\Pi}^u)^T & (\bm{A}^{st}_4\bm{\Pi}^u)^T & (\bm{A}^{st}_7\bm{\Pi}^u)^T \\
(\bm{A}^{st}_2\bm{\Pi}^p)^T & & \\
(\bm{A}^{st}_3\bm{\Pi}^{\lambda})^T & &
\end{bmatrix}
\endgroup
\left(
\begingroup 
\setlength\arraycolsep{1pt}
\begin{bmatrix} 
\bm{P}_{\bm{X}_u}^{st}&  &  \\
& \bm{P}_{\bm{X}_p}^{st} & \\
& & \bm{P}_{\bm{X}_\lambda}^{st}
\end{bmatrix}
\endgroup
\right)^{-1}
\begin{bmatrix}
\\ \\ \bm{F}^{st}_3(\bm{\mu})
\end{bmatrix} \\
&=
\begin{bmatrix}
(\bm{A}^{st}_7\bm{\Pi}^u)^T\bm{F}^{st}_3(\bm{\mu}) \\ \\ \\
\end{bmatrix}~,
\end{split}
\end{equation}
since $\bm{P}_{\bm{X}_\lambda}^{st} = \bm{I}_{n_\lambda^{st}}$. Based on the expression of $\bm{F}_3^{st}(\bm{\mu})$ in  Eq.\eqref{eq: monolitic_FOM_blocks_vector}, the only non--zero block in $\widehat{\bm{F}}^{pg}$ --- denoted as $\widehat{\bm{F}}^{pg}_{7,3}(\bm{\mu})$ --- writes as:
\begin{equation}
\widehat{\bm{F}}^{pg}_{7,3}(\bm{\mu}) = \left[\left(\widehat{\bm{F}}^{pg, 1}_{7,3}(\bm{\mu})\right)^T, \cdots,  \left(\widehat{\bm{F}}^{pg, N_D}_{7,3}(\bm{\mu})\right)^T\right]^T \in \mathbb{R}^{n_\lambda^{st}}~.
\end{equation}
For $k \in \{1, \dots, N_D\}$ and considering $j^k = \mathcal{F}_{\lambda_k}(j_s^k, j_t^k)$ with $j_s^k \in \{1, \dots, N_\lambda^k\}$. $j_t^k \in \{1, \dots, n_{\lambda_k}^t\}$, $\widehat{\bm{F}}^{pg, k}_{7,3}(\bm{\mu}) \in \mathbb{R}^{n_{\lambda_k}^{st}}$ is such that
\begin{fleqn}
	\begin{equation}
	\label{eq: ST--PGRB F73}
	\begin{split}
	\left(\widehat{\bm{F}}^{pg, k}_{7,3}(\bm{\mu})\right)_{j^k} &= \left(\left(\left(\bm{A}^{st}_7\right)^k \bm{\Pi}^u\right)^T\right)_{{j^k},:}
	\begin{bmatrix}
	\bm{\tilde{g}}^s_k g^t_k(t^1;\bm{\mu}) \\ \vdots \\ \bm{\tilde{g}}^s_k g^t_k(t^{N^t};\bm{\mu})
	\end{bmatrix} =
	\begin{bmatrix}
	\bm{C}^k_{j^k_s}(\psi^u_{j^k_t})_1 \\ 
	\vdots \\
	\bm{C}^k_{j^k_s}(\psi^u_{j^k_t})_{N^t}
	\end{bmatrix}^T
	\begin{bmatrix}
	\bm{\tilde{g}}^s_k g^t_k(t^1;\bm{\mu}) \\ \vdots \\ \bm{\tilde{g}}^s_k g^t_k(t^{N^t};\bm{\mu})
	\end{bmatrix}
	\\&=
	\left(\left(\bm{C}_{j^k_s}^k\right)^T \bm{\tilde{g}}^s_k\right) \left(\left(\bm{\psi}^u_{j^k_t}\right)^T \bm{g}_k^t(\bm{\mu})\right)~.
	\end{split}
	\end{equation}
\end{fleqn}
Here $\left(\bm{A}^{st}_7\right)^k$ denotes the $k$--th block of $\bm{A}^{st}_7$ along its first dimension and the vector $\bm{g}^t_k(\bm{\mu}) \in \mathbb{R}^{N^t}$ is such that $\left(\bm{g}^t_k(\bm{\mu})\right)_i = g^t_k(t^i; \bm{\mu})$, for $i \in \{1,\cdots,N^t\}$.

\end{document}